\documentclass[12pt,reqno]{amsart}
\usepackage{geometry}
\usepackage[utf8]{inputenc}
\usepackage{amsfonts}
\usepackage{amsmath}
\usepackage{amsthm}
\usepackage{amssymb}
\usepackage{graphicx} 
\usepackage{tikz}
\usetikzlibrary{decorations.pathreplacing}
\usepackage{enumitem}
\usepackage{hyperref}
\usepackage{xcolor}
\usepackage{mathtools}
\usepackage{relsize}
\usepackage{caption}

\usepackage{changes}
\definechangesauthor[name=Warut, color=red]{W}
\definechangesauthor[name=Fu, color=blue]{L}
\newtheorem{thm}{Theorem}[section]

\newtheorem{prop}[thm]{Proposition}
\newtheorem{lem}[thm]{Lemma}

\newtheorem{cor}[thm]{Corollary}

\theoremstyle{definition}
\newtheorem{defn}[thm]{Definition}

\newtheorem{ex}[thm]{Example}

\newtheorem{rem}[thm]{Remark}

\numberwithin{equation}{section}

\usepackage[silent, english]{babel}
\usepackage[autostyle, english = american]{csquotes}
\MakeOuterQuote{"}

\def\0{\mathbf{0}}
\newcommand{\bba}{\mathbf{a}}
\newcommand{\bbb}{\mathbf{b}}
\newcommand{\bbd}{\mathbf{d}}
\newcommand{\bbc}{\mathbf{c}}
\newcommand{\bbe}{\mathbf{e}}

\newcommand{\bbh}{\mathbf{h}}

\newcommand{\bbm}{\mathbf{m}}
\newcommand{\bbn}{\mathbf{n}}
\newcommand{\bbr}{\mathbf{r}}

\newcommand{\bbu}{\mathbf{u}}
\newcommand{\bbv}{\mathbf{v}}
\newcommand{\bbw}{\mathbf{w}}
\newcommand{\bbx}{\mathbf{x}}
\newcommand{\bby}{\mathbf{y}}

\newcommand{\one}{\mathbf{1}}
\newcommand{\des}{\mathrm{des}}
\newcommand{\asc}{\mathrm{asc}}

\newcommand{\conv}{\mathrm{conv}}

\newcommand{\cone}{\mathrm{co}}

\newcommand{\vol}{\mathrm{vol}}

\newcommand{\ncone}{\mathrm{ncone}}

\newcommand{\pf}{\mathrm{PF}}
\newcommand{\Z}{\mathbb Z}
\newcommand{\N}{\mathbb N}

\newcommand{\R}{\mathbb R}

\newcommand\cov{\mathrel{\ooalign{$\prec$\cr
  \hidewidth\raise0.0ex\hbox{$\cdot\mkern0.9mu$}\cr}}}
\def\P{\mathbb P}

\newcommand{\calA}{\mathcal{A}}
\newcommand{\calB}{\mathcal{B}}
\newcommand{\calC}{\mathcal{C}}
\newcommand{\calD}{\mathcal{D}}

\newcommand{\calF}{\mathcal{F}}

\newcommand{\calM}{\mathcal{M}}

\newcommand{\calX}{\mathcal{X}}

\newcommand{\ssigma}{\Tilde{\sigma}}

\newcommand{\curly}{\mathrel{\leadsto}}

\newcommand{\fS}{\mathfrak S}

\def\bwp{\mathcal{SBP}}

\newcommand\commentout[1]{}
\def\type{\operatorname{type}}

\newcommand\textbox[1]{\parbox{0.75\textwidth}{\raggedright #1}}

\title{Parking Function Polytopes}

\author{Fu Liu}
\address{Fu Liu, Department of Mathematics, University of California, Davis, One Shields Avenue, Davis, CA 95616 USA.}
\email{fuliu@ucdavis.edu}

\author{Warut Thawinrak}
\address{Warut Thawinrak, Beijing International Center for Mathematical Research, Beijing, China.}
\email{warutthawinrak@gmail.com}

\date{}

\begin{document}

\begin{abstract}
We extend the notion of parking function polytopes and study their geometric and combinatorial structure, including normal fans, face posets, and $h$-polynomials, as well as their connections to other classes of polytopes. To capture their combinatorial features, we introduce generalizations of ordered set partitions, called binary partitions and skewed binary partitions. Using properties of preorder cones, we characterize the skewed binary partitions that are in bijection with the cones of the normal fan of a parking function polytope. This description of the normal fan yields an explicit formula for the $h$-polynomials of simple parking function polytopes in terms of generalized Eulerian polynomials. Finally, we relate parking function polytopes to several well-known polytopes, leading to additional results, including formulas for their volumes and Ehrhart polynomials.

\end{abstract}

\maketitle
\thispagestyle{empty}
\setlength\parindent{24pt}

\section{Introduction}

Suppose that $\bbu = (u_1, \dots, u_n) \in \R^n_{\geq 0}$ is a vector satisfying $0 \leq u_1 \leq \cdots \leq u_n.$ For $\bba = (a_1, \dots, a_n) \in \R^n_{\geq 0}$, let $b_1 \leq b_2 \leq \dots \leq b_n$ denote the non-decreasing rearrangement of its entries. We call $\bba$ a $\bbu$-\emph{parking function} if $b_i \leq u_i$ for all $i = 1, \dots, n.$ The \emph{parking function polytope} associated to $\bbu$, denoted by $\pf(\bbu)$, is defined as the convex hull of all $\bbu$-parking functions. If $\bbu \neq \mathbf{0}$, then $\pf(\bbu)$ contains $n+1$ affinely independent points $\mathbf{0}$, $(u_n,0, 0, \dots, 0), (0, u_n, 0, \dots, 0), \dots,$ $ (0, 0, \dots, 0, u_n)$. Therefore, $\pf(\bbu)$ is $n$-dimensional for all $\bbu \in \R^n_{\geq 0}\backslash\{\0\}.$ Thus, for non-triviality, we will always assume that $\bbu$ is a nonzero vector.

\begin{center}
    \begin{figure}[ht]
        \centering
        \includegraphics[scale = 0.24]{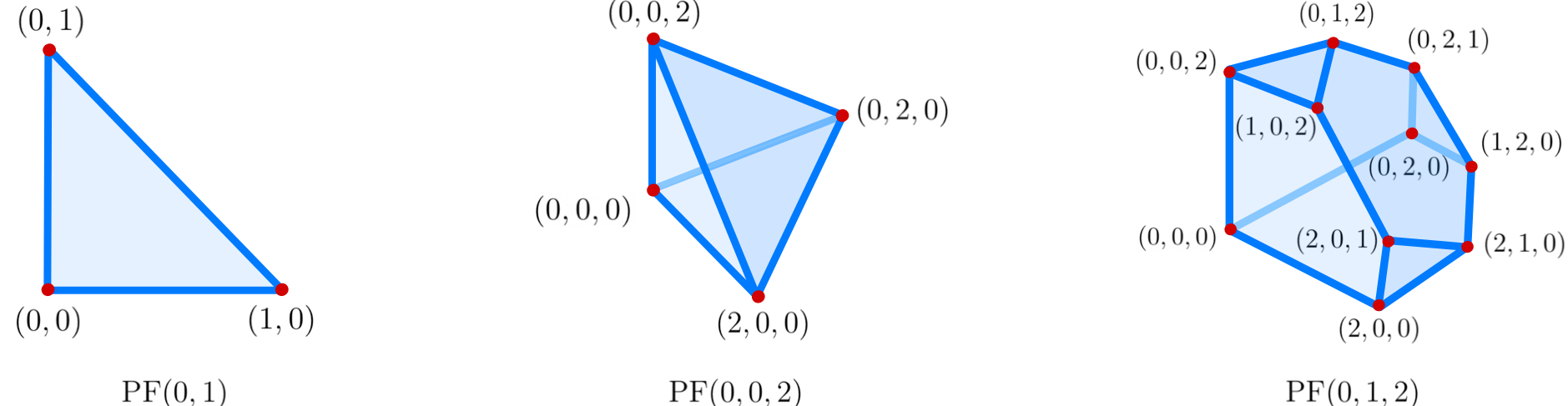}
        \caption{Three examples of parking function polytopes}
        \label{fig:enter-label}
    \end{figure}
\end{center}

Classical parking functions correspond to the case $\bbu=(0,1,\dots, n-1)$: a parking function of length $n$ is a sequence of positive integers $(a_1, \dots, a_n)$ whose non-decreasing rearrangement $b_1 \leq \cdots \leq b_n$ satisfies $b_i \leq i$ for all $i \in [n]$ where $[n]:=\{1, 2, \dots, n\}.$ Since their introduction by Konheim and Weiss as a model for parking \cite{Konheim1966AnOD}, parking functions have played a central role in algebraic and enumerative combinatorics, with deep connections to labeled trees \cite{Chassaing2001}, hyperplane arrangements, and noncrossing partitions \cite{Stanley1997,Stanley1998}. Stanley later defined the \emph{parking function polytope} as the convex hull of all classical parking functions \cite[Problem~12191]{stanley2020question}, and asked questions concerning its faces, volume, and lattice points. These questions were answered by Amanbayeva and Wang \cite{Amanbayeva_2021}.
More recently, Hanada et al.~\cite{hanada2024} and Bayer et al.~\cite{bayer2024parkingfunctionpolytopes} studied a broader family of parking function polytopes 
$\pf(\bbu)$ in which $u_1, \dots, u_n$ are strictly increasing integers. Their work focused on the combinatorial and geometric structure of these polytopes, including formulas for volumes and $h$-polynomials, and connections to other well-known families of polytopes.
The definition we introduced here further generalizes this framework by allowing $\bbu$ to be an arbitrary nondecreasing vector of nonnegative real numbers.

\medskip \noindent\textbf{Main contributions.} Our first contribution is the introduction of a new combinatorial model, \emph{binary partitions}, which generalize ordered set partitions. Via their associated preorder cones, we define a natural contraction relation on binary partitions and, in Theorem~\ref{thm: graph-characterization}, characterize this relation in terms of an associated bipartite graph.
We then present a speicial subclass, called \emph{skewed binary partitions}, and associate to each vector $\bbu$ a \emph{multiplicity vector} $\bbm(\bbu)$. One of our main results, Theorem~\ref{thm: fulldimcones}, establishes a bijection between skewed binary partitions determined by $\bbm(\bbu)$ and the vertices of $\pf(\bbu)$, together with an explicit description of the normal cone at each vertex. As consequences, we obtain a complete description of the normal fan and the face poset of parking function polytopes, and show that the combinatorial type of $\pf(\bbu)$ depends only on the multiplicity vector $\bbm(\bbu)$.
Another important corollary of Theorem \ref{thm: fulldimcones} is the following proposition. 
\begin{prop}\label{prop: pf-stellahedron}
	The normal fan of every parking function polytope $\pf(\bbu)$ is a coarsening of the normal fan of $\pf(1, 2, \dots, n)$. Equivalently, each $\pf(\bbu)$ may be viewed as a deformation of $\pf(1, 2, \dots, n)$. 
\end{prop}

When $\bbu = (1,2, \dots, n)$ or more generally when the entries of $\bbu$ are strictly increasing, the polytope $\pf(\bbu)$ is a \emph{stellahedron}, i.e., the graph associahedron of a star graph introduced by Carr and Devadoss \cite{carr2005coxetercomplexes}. Proposition~\ref{prop: pf-stellahedron} therefore shows that the normal fan of any parking function polytope is a coarsening of the stellahedral fan. This viewpoint aligns with recent work of Eur, Huh, and Larson \cite{EurHuhLar2023}, which uses the geometry of the stellahedral toric variety to study matroids and relates deformations of $\pf(1,2,\dots,n)$ to polymatroids; see Section~\ref{subsec:polymatroid} for further discussion.

Finally, we derive a formula for the $h$-polynomials of simple parking function polytopes in terms of generalized Eulerian polynomials. We also establish connections between parking function polytopes and several other families of polytopes, from which we deduce additional results, including explicit formulas for volumes and Ehrhart polynomials.

\subsection*{Paper organization} 
We begin with background on preposets and preorder cones, and introduce binary partitions and skewed binary partitions. We then develop the tools needed to describe the normal fan of a parking function polytope and derive the formula for the $h$-polynomials of simple parking function polytopes. The final section explores connections to other families of polytopes and uses these relationships to obtain volume and Ehrhart polynomial formulas.


\section*{Acknowledgement}
Both authors are partially supported by an NSF grant \#2153897-0.

\section{Preliminaries}
In this section, we introduce the necessary notation and preliminary results that will be used in the proofs of our main results.

\subsection{Polyhedra}

A \emph{polyhedron} $P$ in $\R^n$ is the solution set to a finite set of linear inequalities: 
\begin{align}\label{eq:polyhedra-inq}
    P = \left\{\bbx=(x_1, \dots, x_n) \in \R^n\ \Big| \  \bba_i \cdot \bbx \leq b_i \text{ for } i \in I\right\}
\end{align}
for some \(\bba_i \in \R^n\), \(b_i \in \R\), with \(\cdot\) the usual dot product, and \(I\) a finite index set.
The dimension of $P$, denoted by $\dim(P),$ is the dimension of $\mathrm{aff}(P)$, the affine span of $P$. 
A subset $F$ of a polyhedron $P$ is a \emph{face} of $P$ if there exists $\bbh \in \R^n$ such that
\[F = \{\bbx \in P\ |\ \bbh\cdot \bbx \geq \bbh\cdot \bby, \text{ for all } \bby \in P\}.\] 
 A face of dimension $\dim(P) - 1$ is called a \emph{facet}, a face of dimension $1$ is an \emph{edge}, and a face of dimension $0$ is a \emph{vertex}. 
By convention, the empty set is also considered a face of dimension \(-1\). Note that the relative interior of a face cannot lie in the relative interior of any other face.
The partially ordered set $\calF(P)$ of all nonempty faces of $P$, ordered by inclusion, is the \emph{face poset} of $P$. 
A \emph{facet-defining inequality} of a polyhedron $P$ is an inequality \(\bba \cdot \bbx \le b\) such that \(\{\bbx \in P \mid \bba \cdot \bbx = b\}\) is a facet. A \emph{minimal inequality description} is a system of inequalities defining \(P\) with minimal cardinality.

A \emph{polytope} is a bounded polyhedron. By the Minkowski-Weyl Theorem \cite{Minkowski1968, Weyl1934}, a polytope $P \subset \R^n$ can equivalently be defined as the convex hull of finitely many points:
\[P = \conv(\bbx_1, \dots, \bbx_k) := \left\{ \sum_{i=1}^k \lambda_i \bbx_i \ \Big| \ \sum_{i=1}^k \lambda_i = 1, \ \lambda_i \ge 0 \text{ for all } i \right\}.\]

Given a finite set of vectors $R=\{\bbr_1, \bbr_2, \dots, \bbr_k\} \subset \R^n$, the \emph{(polyhedral) cone} $\sigma \subset \R^n$ \emph{generated} by $R$ is
\[ \sigma = \cone(R) := \left\{ \bbx \in \R^n \ \big| \ \bbx = \sum_{i=1}^k c_i \bbr_i, \ c_i \ge 0\right\}.\]
Equivalently, a cone $\sigma$ can be defined by a system of homogeneous linear inequalities, i.e., inequalities of the form $\bba\cdot \bbx \leq 0$, and thus is polyhedral. 
We denote by $\sigma^\circ$ the \emph{relative interior} of $\sigma$. A $k$-dimensional cone is \emph{simplicial} if it can be generated by $k$ linearly independent rays.

Let $F$ be a nonempty face of a polytope $P$. The \emph{normal cone} of $P$ at $F$ is the set
\[\ncone(F,P) := \{\bbw \in \R^n \ | \ \bbw\cdot \bbx \geq \bbw \cdot \bby \text{ for all } \bbx \in F \text{ and all } \bby \in P\}.\]
Equivalently, $\ncone(F,P)$ is the set of vectors $\bbw \in \R^d$ for which $\bbw \cdot \bbx$ achieves its maximum over $P$ at $F$. 
The \emph{normal fan} of $P$, denoted $\Sigma(P)$, is the set of normal cones of $P$ at all of its nonempty faces. 
A polytope $Q$ is a \emph{deformation} of $P$ if $\Sigma(Q)$ is a coarsening of $\Sigma(P)$, i.e., each cone in $\Sigma(Q)$ is a union of cones in $\Sigma(P)$. We denote by $\calF(\Sigma(P))$ the poset on $\Sigma(P)$ ordered by inclusion. 
The following well-known lemma establishes a correspondence between faces of a polytope and their normal cones.

\begin{lem}\label{lem: normalfan-faceposet}
    The map $F \mapsto \ncone(F,P)$ for nonempty faces $F$ is a poset isomorphism from the dual poset of $\calF(P)$ to the poset $\calF(\Sigma(P)).$
\end{lem}

The next result is a slight variation of \cite[Lemma 2.4]{CastilloFu2023permutoassociahedron}. It provides a way to verify normal cones at vertices of polytope. We omit the proof as it is very similar to the proof of the original result.

\begin{lem}\label{lem:nfan}
    Suppose that $\calM = \{\sigma_1, \dots, \sigma_k\}$ is a set of cones satisfying $\sigma_1 \cup \dots \cup \sigma_k = \R^n$ and that $\{\bbv_1, \dots, \bbv_k\} \subseteq\R^n$ is a set of points in which for every $i\in \{1, \dots, k\}$
    \[\bbc\cdot \bbv_i > \bbc\cdot \bbv_j \text{ for all } \bbc \in \sigma^\circ_i \text{ and all } j \neq i.\]
    Let $P$ be a polytope defined by $P := \conv(\bbv_1, \dots, \bbv_k)$. Then the set of vertices of $P$ is $\{\bbv_1, \dots, \bbv_k\}$. In addition, we have that $\sigma_i = \ncone(\bbv_i, P)$ for all $i \in \{1,\dots, k\}$. As a consequence, the set of cones in $\calM$ and their faces form the normal fan $\Sigma(P)$ of $P$.
\end{lem}

The following lemma follows from \cite[Lemma 2.5]{CastilloFu2023permutoassociahedron}.
\begin{lem}\label{lem:det-ineq}
		Suppose that $P$ is a full-dimensional polytope in $\R^n$ with vertex set $\{\bbv_1, \dots, \bbv_k\}.$
		Let $R$ be the set of one dimensional cones in $\Sigma(P)$. Suppose that $\{\rho_1, \rho_2, \dots, \rho_m\}$ is a set of one-dimensional cones that contains $R,$ and let $\bbn_j$ be a nonzero vector in the cone $\rho_j$ (or equivalently a generator for $\rho_j$). Then the polytope $P$ has the following inequality description: 
		\[ P = \left\{\bbx \in 
		\R^n \ :  \  \bbn_j \cdot \bbx  \leq \max_{1 \le i \le k} (\bbn_j \cdot \bbv_i), \quad 1 \le j \le m \right\}, \]
	in which the $j$th inequality defines a facet of $P$ if and only if $\rho_j \in R.$
        %
	\end{lem}

A \emph{lattice point} in $\R^n$ is a point whose coordinates are integers. A polytope is said to be \emph{integral} if every vertex of it is a lattice point. For a polytope $P$ in $\R^n$ and a nonnegative integer $t,$ the $t^{\mathrm{th}}$-dilation $tP$ is the set $\{tx \ |\ x \in P\}.$ We define $i(P,t) := |\Z^n \cap tP|$ to be the number of lattice points in the $t^\mathrm{th}$-dilation $tP.$ Due to Ehrhart \cite{Ehrhart1962}, we have that if $P$ is an integral polytope, then $i(P,t)$ is a polynomial in $t$ of degree equal to the dimension of $P$. Thus, we call $i(P,t)$ the \emph{Ehrhart polynomial} of $P$. It is well-known that the leading coefficient of the Ehrhart polynomial of $P$ is the volume of $P$.

Two integral polytopes $P, Q$ such that $P \subset \R^n$ and $Q \subset \R^m$ are said to be \emph{integrally equivalent} if there exists an invertible affine transformation from $\mathrm{aff}(P)$ to $\mathrm{aff}(Q)$ that preserves the lattice points in the two polytopes. When two integral polytopes are integrally equivalent, they have the same face poset, volume, and Ehrhart polynomials. 

A $d$-dimensional polytope is said to be \emph{simple} if all its vertices are incident to exactly $d$ edges. If $P$ is full-dimensional, i.e., $P$ is a $d$-dimensional polytope in $\R^d$, then one can show that being simple is equivalent to having the normal cone at every vertex being simplicial.  Given a $d$-dimensional polytope $P,$ we let $f_i(P)$ be the number of its $i$-dimensional faces. The \emph{$f$-vector} of $P$ is defined to be the vector $(f_0(P), \dots, f_d(P)),$ and the \emph{$f$-polynomial} of $P$ is given by $f_P(t):= f_0(P) + f_1(P)t + \dots + f_d(P)t^d.$ If a $d$-dimensional polytope $P$ is also simple, we define its \emph{$h$-polynomial} $h_P(t) := h_0(P) + h_1(P)t + \cdots + h_d(P)t^d$ and its \emph{$h$-vector} $(h_0(P), \dots, h_d(P))$ to satisfy the relation $f_P(t) = h_P(t+1)$. This is equivalent to having
\begin{equation}\label{eq: htof}
    f_j(P) = \sum^d_{i \geq j}\binom{i}{j}h_i(P) \ \text{ for all } j = 0, \dots, d.
\end{equation}

It is well-known that the $h$-polynomial $h_P(t)$ of a simple polytope $P$ has nonnegative coefficients \cite[Section 8.2]{Ziegler1995} and is palindromic \cite{Dehn1906, Sommerville1927} as it satisfies the Dehn-Sommerville symmetry. That is, $0 \leq h_i(P) = h_{d-i}(P)$ for all $i = 0, \dots, d$. Thus, we only need to know half of the coefficients of $h_P(t)$ to recover the $f$-vector using equation $\eqref{eq: htof}$.

\subsection{Preposets and preorder cones}
    We introduce the notion of preposets which is, in a sense, a generalization of posets, and then introduce their associated preorder cones. Readers are expected to be familiar with basic notations regarding poset as appear, for example, in \cite[Section 3.1]{Stanley2011EC1}.

A binary operator $\preceq$ on a finite set $A$ is called a \emph{preorder} if it is reflexive and transitive on $A$. A \emph{preposet} is an ordered pair $(A,\preceq)$ of a finite set $A$ and a preorder $\preceq$ on it. We write $i \equiv j$ if $i \preceq j$ and $j \preceq i$. The relation $\equiv$ is an equivalence relation on $A$ and thus partitions $A$ into equivalence classes. We denote by $A/_\equiv$ the set of equivalence classes of $A$ and $\Bar{i}$ the equivalence class of $i$. One sees that we recover the definition of a poset if we require a preposet $(A, \preceq)$ satisfies $i \equiv j$ if and only if $i = j$, i.e., the relation $\equiv$ is antisymmetric. 

Note that the preorder $\preceq$ on $A$ induces a partial order on $A/_\equiv$ by letting $\Bar{i} \preceq \Bar{j}$ if $i \preceq j$ in $A$, and thus defines a poset $(A/_\equiv, \preceq)$ which is closely related to the preposet $(A, \preceq)$. This allows us to define concepts for the preposet $(A, \preceq)$ from concepts for the poset $(A/_\equiv, \preceq)$:
We say that $j$ is a \emph{cover} of $i$ in the preposet $(A, \preceq)$, denoted $i \lessdot j$, if $\Bar{j}$ is a cover of $\Bar{i}$ in the poset $(A/_\equiv, \preceq).$ 
The Hasse diagram of of a preposet $(A, \preceq)$ is the Hasse diagram of the poset $(A/_\equiv, \preceq)$, labeling nodes by equivalence classes without parentheses.

A preoder $\preceq_1$ on $A$ is said to be a \emph{contraction} of another preorder $\preceq_2$ on $A$ if the Hasse diagram of $(A, \preceq_1)$ can be obtained by a sequence of edge contractions of the Hasse diagram and merges of the vertex labels of $(A, \preceq_2)$ .

\begin{ex}
We draw in Figure \ref{fig:3contractions} the Hasse diagrams of three different preposets on $[0,8]$, among which $([0,8], \preceq_1)$ is a poset. The preorder $\preceq_2$ is a contraction of the preorder $\preceq_1$ by contracting the edge $6-8$ and the edge $5-7$. The preorder $\preceq_3$ is a contraction of the preorder $\preceq_2$ by contracting the edge $3-0$. As a result, the preorder $\preceq_3$ is also a contraction of the preorder $\preceq_1$. 
\end{ex}

\begin{figure}
    \centering
    \includegraphics[scale = 1]{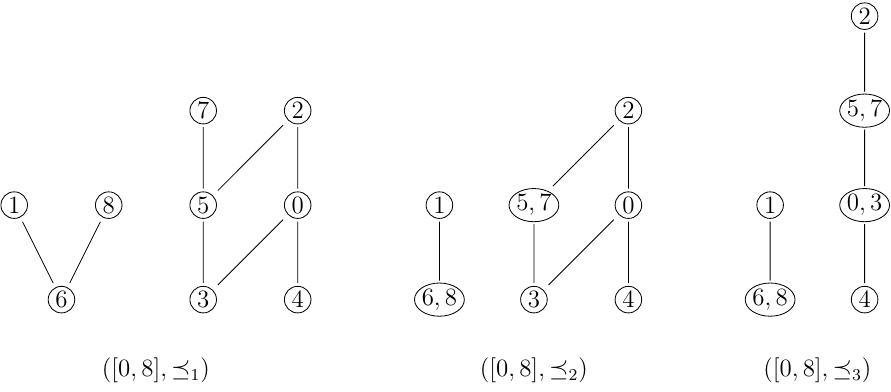}
    \caption{Both $\preceq_2$ and $\preceq_3$ are contractions of $\preceq_1$}
    \label{fig:3contractions}
\end{figure}

A map $f$ from a preposet $(A_1, \preceq_1)$ to another preposet $(A_2, \preceq_2)$ is \emph{order-preserving} if for every $x,y \in A_1$ such that $x \preceq_1 y$ one has $f(x) \preceq_2 f(y).$ If an order-preserving map $f$ is bijective and its inverse is also order-preserving, we say that $f$ is an \emph{isomorphism}.

The \emph{dual} of a preposet $(A, \preceq)$ is the preposet $(A, \preceq^*)$ such that $i \preceq^* j$ if and only if $j \preceq i.$ Clearly, the Hasse diagram of the dual poset $(A, \preceq^*)$ is obtained by turning the Hasse diagram of $(A, \preceq)$ upside down.

In \cite[Section 3]{Postnikov2006}, Postnikov, Reiner, and Williams introduce a natural correspondence between cones in quotient space $\R^n/(1, \dots, 1)\R$ and preorders of the set $[n]$ in the study of faces of generalized permutahedra. Later, Castillo and the first author \cite{CastilloFu2023permutoassociahedron} call these cones \emph{preorder cones}, indicating that they arise from some preposets. They also introduce variations of preorder cones, including ones that are defined in the first orthant of $\R^n$. In this paper, we will start with a preposet on $[0,n]$ and consider preorder cones without quotienting out $(1,\dots,1)\R$. More precisely, given a preposet $([0,n],\preceq)$, we define its associated \emph{preorder cone} to be the cone
\begin{align}\label{eq:usualpreordcone}
    \sigma_\preceq := \{(c_0,c_1, \dots, c_n) \in \R^{n+1}  \ | \ c_i \le c_j \text{ if } i \preceq j, i,j \in [0,n]\}.
\end{align}
We will utilize preorder cones to study the face structure of parking function polytopes. However, it turns out that the slice of $\sigma_\preceq$ at $c_0 = 0$ will mostly play an important role. This leads us to introduce the following definition. 

\begin{defn}\label{defn:preordcone}
    Let $\preceq$ be a preorder on $[0,n]$. The \emph{sliced preorder cone} $\ssigma_\preceq$ associated to $\preceq$ is given by
    \begin{align}\label{eq:preordcone}
    \ssigma_\preceq := \{(c_1, \dots, c_n) \in \R^{n}  \ | \ c_0 = 0 \text{ and } c_i \le c_j \text{ if } i \preceq j, i,j \in [0,n]\}.
    \end{align}
\end{defn}

\begin{ex}\label{ex:preorder-cones}
Let $\preceq$ be the third preorder $\preceq_3$ on $[0,8]$ shown in Figure \ref{fig:3contractions}. 
Then
\[    \ssigma_{\preceq} = \{(c_1, \dots, c_n) \in \R^{n}  \ | \ c_6=c_8 \le c_1 \text{ and } c_4 < 0 = c_3 < c_5=c_7 < c_2\}.
  \]
\end{ex}

A \emph{linear extension} of the preposet $([0,n],\preceq)$ is a bijective order-preserving map from the preposet $([0,n],\preceq)$ to the poset $([0,n], \leq)$ where $\leq$ is the usual order of integers. We denote by $L(\preceq)$ the set of all linear extensions of the preposet $([0,n],\preceq)$.

The next lemma contains a variation of results from \cite[Proposition 3.5]{Postnikov2006} and \cite{CastilloFu2023permutoassociahedron} regarding sliced preorder cones. 

\begin{lem}\label{lem: preordcone}
    Let $\preceq$ and $\preceq'$ be preorders on $[0,n]$. We have that
    \begin{enumerate}[label=(\arabic*)]
        \item\label{itm: conefaces} The sliced preorder cone $\ssigma_{\preceq'}$ is a face of the sliced preorder cone $\ssigma_\preceq$ if and only if $\preceq'$ is a contraction of $\preceq.$
        \item \label{itm:minimaldescription} If a preposet $([0,n],\preceq)$ is a poset, then the associated sliced preorder cone has the following minimal inequality description:
        \[\ssigma_\preceq = \{(c_1, \dots, c_n) \in \R^n  \ | \ c_0 = 0 \text{ and } c_i \le c_j \text{ if } i \lessdot j, i,j \in [0,n]\}\]
        and, hence, the relative interior $\ssigma_\preceq^\circ$ of $\ssigma_\preceq$ is given by
        \[\ssigma_\preceq^\circ = \{(c_1, \dots, c_n) \in \R^n  \ | \ c_0 = 0 \text{ and } c_i < c_j \text{ if } i \lessdot j, i,j \in [0,n]\}.\]
        \item \label{itm: dimension} The dimension of the sliced preorder cone $\ssigma_\preceq$ is the number of equivalence classes in $([0,n], \preceq)$ minus $1$.  
        \item \label{itm: simplicial} The sliced preorder cone $\ssigma_\preceq$ is an $n$-dimensional simplicial cone if and only if $([0,n],\preceq)$ is a poset and its Hasse diagram is a tree (a connected graph with no cycles).
        \item\label{itm:poset}  If a preposet $([0,n],\preceq)$ is a poset, then
        \[\ssigma_\preceq = \bigcup_{\pi \in L[\preceq]} \ssigma(\pi)\] 
        where 
    \begin{equation}\label{eq:sigmapi} 
            \ssigma(\pi) := \{(c_1, \dots, c_n) \in \R^n \ | \ c_0 = 0 \text{ and } c_{\pi(0)} \leq c_{\pi(1)} \leq \cdots \leq c_{\pi(n)}\}.
        \end{equation}
   \end{enumerate}
\end{lem}

The proofs of these results can be obtained by setting $c_0 = 0$ in the proofs of the original results. Hence, the readers who are interested in the proofs may check out \cite[Proposition 3.5]{Postnikov2006}, \cite{CastilloFu2023permutoassociahedron}, and \cite[Lemma 4.1.4]{WarutThesis2025}.

\section{Binary partition and contraction}

In this section, we consider a special family of preorders on $[0,n]$ that can be represented by what we call \emph{binary partitions} of $[0,n]$. We will then characterize the contractions of these preorders in terms of binary partitions. In the next section, we will consider special cases of these partitions that will be useful for describing the normal cones of parking function polytopes.

Recall that an ordered partition of a nonempty set $S$ is a tuple $\calB=(B_1, \dots, B_k)$ of nonempty disjoint subsets of $S$ such that $B_1 \sqcup \cdots \sqcup B_k = S.$ Each subset $B_i$ is called a \emph{block}. To represent a special family of preorders on $[0,n]$, we introduce an analogue of ordered partition called "binary partition" of the set $S = [0,n]$ by separating blocks into two different types: \emph{homogeneous}, and \emph{non-homogeneous}. These block types will be useful for expressing the inequality description of preorder cones. A homogeneous block is marked with a superscript $\star$ for differentiation; for example, $\{1, 3\}^\star$ is a homogeneous block.  We are allowed to apply usual set operations such as union and intersection to homogeneous and non-homogeneous blocks as we normally do to regular sets. 

\begin{defn}\label{defn: partition} Let $k \in \P.$ A \emph{binary partition} of $[0,n]$ into $k$ blocks is an ordered tuple $(B_1, \dots, B_k)$ of nonempty disjoint subsets of $[0,n]$ such that $B_1 \sqcup \cdots \sqcup B_k = [0,n]$ and satisfies the following additional properties.
\begin{enumerate}
    \item Every block is either homogeneous or non-homogeneous.
    \item \label{itm:bp-singleton} Every singleton block is non-homogeneous.
\end{enumerate}
\end{defn}

\begin{rem} For a singleton block, it has the property of both homogenous and non-homogeneous blocks. However, because we do not want to allow both kinds, we make a choice to make it always non-homogenous. Hence, we have Condition (\ref{itm:bp-singleton}) in the above definition. 
\end{rem}

\begin{defn}\label{defn: preord-partition}
For each binary partition $\calB = (B_1, \dots, B_k)$ of $[0,n]$, we associate the preorder $\preceq_\calB$ on the set $[0,n]$ by letting
\begin{align*}
    p &\preceq_\calB  q & &\text{ if } p \in B_i \text{ and } q \in B_j \text{ and } i < j\\
    p &\equiv_\calB q & &\text{ if $p, q \in B_i$ for some homogeneous block $B_i$. }
\end{align*}
If a preposet $([0,n], \preceq)$ satisfies $([0,n], \preceq) = ([0,n], \preceq_\calB)$ for some binary partition $\calB$, then we say that the preorder $\preceq$ is \emph{representable}.

A binary partition $\calC$ is a \emph{contraction} of another binary partition $\calB$, denoted by $\calC \leq \calB$, if $\preceq_\calC$ is a contraction of $\preceq_\calB.$
\end{defn}

\begin{ex}\label{ex:binarypartition} Figure \ref{fig:binarypartitions-preorders} shows preorders $\preceq_\calB, \preceq_\calC$ and $\preceq_\calD$ associated to the binary paritions
$\calB = (\{0,2,3\}, \{1,6,7\}, \{8\}, \{4,5\})$, $\calC = (\{1,2,5\}, \{3,6\}^\star, \{7\}, \{0,4\}^\star, \{8\})$, and $\calD = (\{2,3\}, \{0,7\}^\star, \{6\},\{1,8\}^\star, \{4,5\})$, respectively. 
It is not difficult to see that $\preceq_\calD$ is a contraction of $\preceq_\calB$ and so $\calD \leq \calB$.
\end{ex}

One sees that labeling a block as being homogeneous is simply a way to represent an equivalent class of a preposet. Not every preorder on $[0,n]$ is representable by a binary partition. In fact, a preorder on $[0,n]$ is representable by a binary partition if and only if it induces a graded poset with the properties $\Bar{i} \preceq
 \Bar{j}$ if the rank of $\Bar{j}$ is higher than the rank of $\Bar{i}$, and every equivalent class of size at least two is comparable with all other equivalent classes.

\begin{lem}\label{lem: contraction-representable}
    Every contraction of a representable preorder on $[0,n]$ is representable.
\end{lem} 

\begin{proof}
     Suppose that a preorder is representable by $\calB = (B_1,\dots, B_k).$ To prove the statement, it suffices to show that contracting one edge of the Hasse diagram of $([0,n],\preceq_\calB)$ gives a preorder $([0,n], \preceq_C)$ for some binary partition $\calC$.

    Recall that the nodes of the Hasse diagram of a preorder are equivalent classes. Consider the preorder $([0,n], \preceq)$ obtained by contracting an edge $\Bar{g}-\Bar{h}$ of the Hasse diagram of $(\preceq_\calB, [0,n])$ where $\Bar{g}$ and $\Bar{h}$ are two equivalent classes of $[0,n]/_{\equiv_\calB}$. As an edge is contracted, we see that $\Bar{g}$ and $\Bar{h}$ come from two consecutive blocks of $\calB$, that is, there is a positive integer $i$ such that $\Bar{g} \subseteq B_i$ and $\Bar{h} \subseteq B_{i+1}$. We note that $B_i\backslash \Bar{g} = \emptyset$ (resp. $B_{i+1}\backslash\Bar{h} = \emptyset$) if and only if $B_i$ is a homogeneous or singleton block. If both $B_i$ and $B_{i+1}$ are neither homogeneous nor singleton blocks, we let $X = B_i\backslash \Bar{g}$ be a block of the same type as $B_i$, and $Y = B_{i+1}\backslash \Bar{h}$ be a block of the same type as $B_{i+1}$. We define a binary partition $\calC$ having $(\Bar{g}\cup \Bar{h})^\star$ as its homogeneous block as follows. 
    \begin{equation}\label{eq: contract-an-edge}
        \calC =
\begin{cases}
  (B_1, \dots, B_{i-1}, (\Bar{g}\cup \Bar{h})^\star, B_{i+2}, \dots, B_p)&\text{ if } B_{i}\backslash \Bar{g} = B_{i+1}\backslash \Bar{h} = \emptyset\\
  (B_1, \dots, B_{i-1}, X,(\Bar{g}\cup \Bar{h})^\star, B_{i+2}, \dots, B_p)&\text{ if } B_{i}\backslash \Bar{g} \neq \emptyset \text{ and } B_{i+1}\backslash \Bar{h} = \emptyset\\
    (B_1, \dots, B_{i-1}, (\Bar{g}\cup \Bar{h})^\star, Y, B_{i+2}, \dots, B_p)&\text{ if } B_{i}\backslash \Bar{g} = \emptyset \text{ and } B_{i+1}\backslash \Bar{h} \neq \emptyset\\
    (B_1, \dots, B_{i-1}, X, (\Bar{g}\cup \Bar{h})^\star, Y, B_{i+2}, \dots, B_p)&\text{ if } B_{i}\backslash \Bar{g} \neq \emptyset \text{ and } B_{i+1}\backslash \Bar{h} \neq \emptyset
\end{cases}
    \end{equation}

It's then easy to check that $([0,n], \preceq) = ([0,n], \preceq_\calC).$
\end{proof}

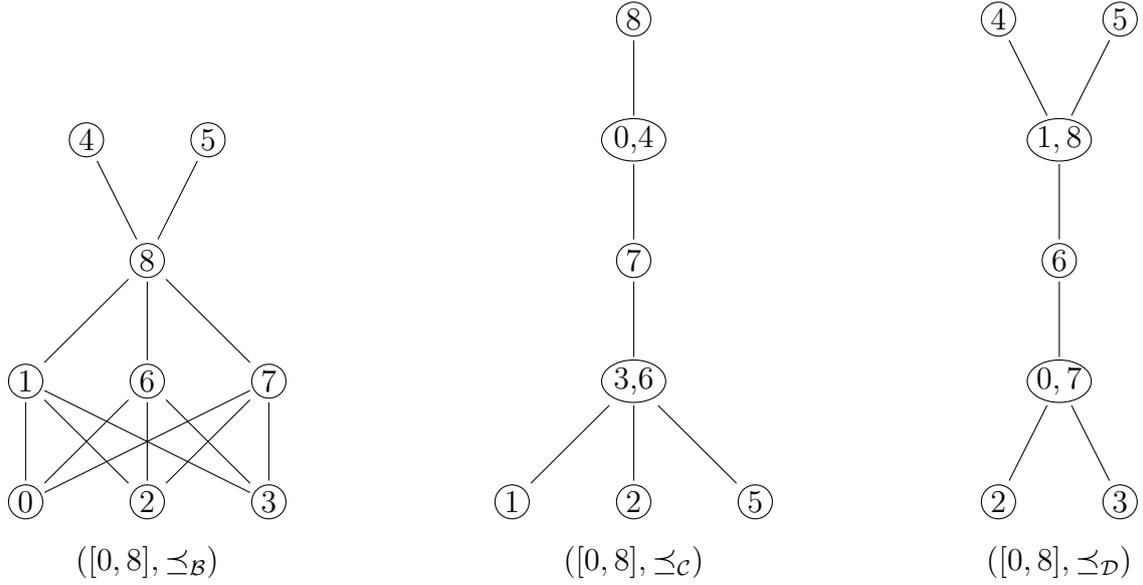
\begin{figure}
    \centering
\begin{tikzpicture}[scale = 0.8]
        \node (zero) at (-9,1) {0};
        \draw (-9,1) circle (8pt);
        \node (min1) at (-7,1) {2};
        \draw (-7,1) circle (8pt);
        \node (min2) at (-5,1) {3};
        \draw (-5,1) circle (8pt);
        \node (a) at (-9,3) {1};
        \draw (-9,3) circle (8pt);
        \node (b) at (-7,3) {6};
        \draw (-7,3) circle (8pt);
        \node (c) at (-5,3) {7};
        \draw (-5,3) circle (8pt);
        \node (d) at (-7,5) {8};
        \draw (-7,5) circle (8pt);
        \node (max1) at (-8,7) {4};
        \draw (-8,7) circle (8pt);
        \node (max2) at (-6,7) {5};
        \draw (-6,7) circle (8pt);
        \draw (zero)--(a) (zero)--(b) (zero)--(c) (a)--(d)--(max1) (d)--(max2) (min1)--(a) (min1)--(b) (min1)--(c) (min2)--(a) (min2)--(b) (min2)--(c) (b)--(d) (c)--(d);
        \node (poset1) at (-6-1,0) {$([0,8], \preceq_\calB)$};

        \node (min11) at (0-1,1) {1};
        \draw (0-1,1) circle (8pt);
        \node (min21) at (2-1,1) {2};
        \draw (2-1,1) circle (8pt);
        \node (min31) at (4-1,1) {5};
        \draw (4-1,1) circle (8pt);
        \node (s11) at (2-1,3) {3,6};
        \draw (2-1,3) ellipse (15pt and 10pt);
        \node (s31) at (2-1,7) {0,4};
        \draw (2-1,7) ellipse (15pt and 10pt);
        \node (s21) at (2-1,5) {7};
        \draw (2-1,5) circle (8pt);
        \node (s41) at (2-1,9) {8};
        \draw (2-1,9) circle (8pt);
        \draw (min11)--(s11) (min21)--(s11) (min31)--(s11) (s11)--(s21) (s21)--(s31) (s31)--(s41);
        \node (poset1) at (2-1,0) {$([0,8], \preceq_{\calC})$};

        \node (min1) at (7,1) {2};
        \draw (7,1) circle (8pt);
        \node (min2) at (9,1) {3};
        \draw (9,1) circle (8pt);
        \node (b) at (8,3) {$0,7$};
        \draw (8,3) ellipse (15pt and 10pt);
        \node (c) at (8,5) {6};
        \draw (8,5) circle (8pt);
        \node (d) at (8,7) {$1,8$};
        \draw (8,7) ellipse (15pt and 10pt);
        \node (max1) at (7,9) {4};
        \draw (7,9) circle (8pt);
        \node (max2) at (9,9) {5};
        \draw (9,9) circle (8pt);
        \draw (d)--(max1) (d)--(max2) (min1)--(b)  (min2)--(b) (b)--(c) (c)--(d);
        \node (poset1) at (9-1,0) {$([0,8], \preceq_{\calD})$};
        
\end{tikzpicture}
    \caption{The preorders associated to $\calB, \calC,$ and $\calD$ in Example \ref{ex:binarypartition}}
    \label{fig:binarypartitions-preorders}
\end{figure}

The set of all binary partitions of $[0,n]$ becomes a poset when partially ordered by contraction.  We now aim to characterize contraction in terms of graphs defined by binary partitions 

Given two binary partitions $\calB = (B_1,\dots, B_p)$ and $\calC= (C_1, \dots, C_q)$ of $[0,n]$, we associate the bipartite graph $G(\calB,\calC)$ whose two disjoint sets of vertices are $V_1 = \{B_1,\dots B_p\}$ and $V_2=\{C_1, \dots, C_q\}$ (written in this order), and a vertex $B_i \in V_1$ is adjacent to a vertex $C_j \in V_2$ if $B_i \cap C_j \neq \emptyset.$ The vertices in $V_1$ will be called \emph{left vertices} and the vertices in $V_2$ will be called \emph{right vertices}. A vertex of $G(\calB,\calC)$ is said to be \emph{non-homogeneous} (resp. \emph{homogeneous}) if it corresponds to a non-homogeneous (resp. \emph{homogeneous}) block of either $\calB$ or $\calC$. 
For a vertext $v$ of $G(\calB,\calC),$ we define 
\begin{align*}
    \deg^\star(v) &:=\# (\text{homogeneous vertices adjacent to } v),\\
    \deg^\vee(v) &:=\# (\text{non-homogeneous vertices adjacent to } v).
\end{align*}
Clearly, $\deg(v) = \deg^\star(v) + \deg^\vee(v)$.

When the edges of $G(\calB,\calC)$ are not crossing, we say that $G(\calB,\calC)$ is \emph{non-crossing}. See Figure \ref{fig:noncrossing-bipartite} for examples of crossing and non-crossing bipartite graphs. 

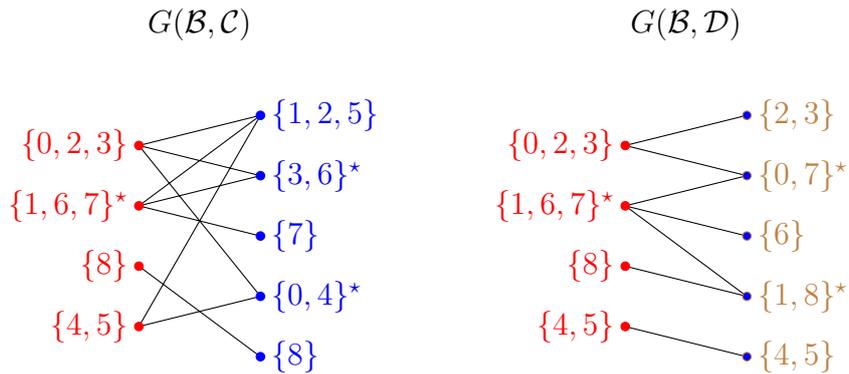
\begin{figure}[ht]
    \centering
\begin{tikzpicture}[scale = 0.8]
        \draw (-5,4)--(-3,4.5) (-5,4)--(-3,3.5) (-5,4)--(-3,1.5) (-5,3)--(-3,4.5) (-5,3)--(-3,3.5) (-5,3)--(-3,2.5) (-5,2)--(-3,0.5) (-5,1)--(-3,1.5) (-5,1)--(-3,4.5);
        \node (g1) at (-4,6) {$G(\calB,\calC)$};
        \draw[red, fill = red] (-5,4) circle (2pt) node[left] {$\{0,2,3\}$};
        \draw[red, fill = red] (-5,3) circle (2pt) node[left] {$\{1,6,7\}^\star$};
        \draw[red, fill = red] (-5,2) circle (2pt) node[left]  {$\{8\}$};
        \draw[red, fill = red] (-5,1) circle (2pt) node[left] {$\{4,5\}$};
        \draw[blue, fill = blue] (-3,4.5) circle (2pt) node[right] {$\{1,2,5\}$};
        \draw[blue, fill = blue] (-3,3.5) circle (2pt) node[right] {$\{3, 6\}^\star$};
        \draw[blue, fill = blue] (-3,2.5) circle (2pt) node[right] {$\{7\}$};
        \draw[blue, fill = blue] (-3,1.5) circle (2pt) node[right] {$\{0,4\}^\star$};
        \draw[blue, fill = blue] (-3,0.5) circle (2pt) node[right] {$\{8\}$};

        \node (g2) at (4,6) {$G(\calB,\calD)$};
        \draw (3,4)--(5,4.5) (3,4)--(5,3.5) (3,3)--(5,3.5) (3,3)--(5,2.5) (3,3)--(5,1.5) (3,2)--(5,1.5) (3,1)--(5,0.5);
        \draw[red, fill = red] (3,4) circle (2pt) node[left] {$\{0,2,3\}$};
        \draw[red, fill = red] (3,3) circle (2pt) node[left] {$\{1,6,7\}^\star$};
        \draw[red, fill = red] (3,2) circle (2pt) node[left]  {$\{8\}$};
        \draw[red, fill = red] (3,1) circle (2pt) node[left] {$\{4,5\}$};
        \draw[brown, fill = blue] (5,4.5) circle (2pt) node[right] {$\{2,3\}$};
        \draw[brown, fill = blue] (5,3.5) circle (2pt) node[right] {$\{0,7 \}^\star$};
        \draw[brown, fill = blue] (5,2.5) circle (2pt) node[right] {$\{6\}$};
        \draw[brown, fill = blue] (5,1.5) circle (2pt) node[right] {$\{1,8\}^\star$};
        \draw[brown, fill = blue] (5,0.5) circle (2pt) node[right] {$\{4,5\}$};
        
\end{tikzpicture}
    \caption{$G(\calB,\calC)$ is crossing but $G(\calB,\calD)$ is non-crossing}
    \label{fig:noncrossing-bipartite}
\end{figure}

When $G(\calB,\calC)$ is non-crossing, it is not difficult to verify the following result regarding the intersections of the blocks of $\calB$ and  $\calC.$

\begin{lem}\label{lem: block-intersection}
    Let $\calB = (B_1, \dots, B_p)$ and $\calC = (C_1, \dots, C_q)$ be two binary partitions of $[0,n]$. Suppose that $G(\calB,\calC)$ is non-crossing. Then, for $i \in [p]$ and $j \in [q]$, we have
    \begin{enumerate}
        \item\label{itm: some-blocks-intersection} If $|C_1| + \cdots |C_j| < |B_1| + \cdots + |B_i|$, then there exist positive integers $s$ and $t$ such that $s \leq i$, $t > j$, and both $C_j\cap B_s$ and $C_t\cap B_i$ are non-empty.
        \item\label{itm: blocks-intersection}  $C_j \cap B_i$ is nonempty if and only if 
        \begin{align*}
            |C_1| + |C_2| + \cdots + |C_{j-1}| &< |B_1| + |B_2| + \cdots +
|B_{i}|, \text{ and }\\
|B_1| + |B_2| + \cdots + |B_{i-1}| &< |C_1| + |C_2| + \cdots +
|C_{j}|.
        \end{align*}
        \item\label{itm: blocks-inclusion}  $B_i \subseteq C_j \neq \emptyset$ if and only if
        \begin{align*}
            |C_1| + |C_2| + \cdots + |C_{j-1}| &\le |B_1| + |B_2| + \cdots + |B_{i-1}|, \text{ and}\\
         |B_1| + |B_2| + \cdots + |B_{i}| &\le |C_1| + |C_2| + \cdots + |C_{j}|.
        \end{align*}

\item\label{itm: blocks-revinclusion}  $C_j \subseteq B_i \neq \emptyset$ if and only if
        \begin{align*}
            |C_1| + |C_2| + \cdots + |C_{j}| &\le |B_1| + |B_2| + \cdots + |B_{i}|, \text{ and}\\
         |B_1| + |B_2| + \cdots + |B_{i-1}| &\le |C_1| + |C_2| + \cdots + |C_{j-1}|.
        \end{align*}
    \end{enumerate}
\end{lem}

The next theorem is the main result of this section. It provides a characterization of binary partition contractions in terms of bipartite graphs and their vertex degrees. We will devote the rest of this section to proving it.

\begin{thm}\label{thm: graph-characterization}
    Let $\calB$ and $\calC$ be binary partitions of $[0,n]$. We have that $\calC \leq \calB$ if and only if $G(\calB,\calC)$ satisfies the following conditions.
    \begin{enumerate}[label=(\arabic*)] 
    \item \label{itm: non-crossing} $G(\calB,\calC)$ is non-crossing.
    \item \label{itm: left-nonhomog1} Every left non-homogeneous vertex $v$ satisfies $\deg^\vee(v) \leq 1.$
    \item \label{itm: left-homog1} Every left homogeneous vertex $v$ satisfies $\deg^\star(v) = 1$ and $\deg^\vee(v) = 0$
    \item \label{itm: right-nonhomog1} Every right non-homogeneous vertex $v$ satisfies $\deg^\star(v) = 0$ and $\deg^\vee(v) = 1$.
    \item \label{itm: right-homog1} If a right homogeneous vertex $v$ satisfies $\deg(v) =1,$ then $\deg^\star(v) =1.$
    \end{enumerate}
\end{thm}

To prove Theorem \ref{thm: graph-characterization}, we first need to develop several observations and lemmas. Let us start by describing the covering relations $\calC \lessdot \calB.$ Recall that the nodes of the Hasse diagram of a preorder are equivalent classes. Lemma \ref{lem: contraction-representable} implies that $([0,n], \preceq_{\calC})$ is obtained by contracting an edge $\Bar{g}-\Bar{h}$ of the Hasse diagram of $([0,n], \preceq_\calB)$ where $\Bar{g}$ and $\Bar{h}$ are two equivalent classes of $[0,n]/_{\equiv_\calB}$. As in the proof of Lemma \ref{lem: contraction-representable}, $\calC$ can be written as in equation \eqref{eq: contract-an-edge}. 
We can describe this in terms of $G(\calB,\calC)$ as follows.

\begin{lem}\label{lem: covering}
    We have that $\calB = (B_1, \dots, B_p)$ is a cover of $\calC = (C_1, \dots, C_q)$ if and only if $G(\calB,\calC)$ is a non-crossing bipartite graph with a unique right vertex $C_j$ of degree two satisfying the following properties.
    \begin{enumerate}[label=(\arabic*)]
        \item The vertex $C_j$ is homogeneous.
        \item \label{itm: otherRightVert_cover} Every right vertex that is not $C_j$ has degree one and is adjacent to a vertex of the same type
        \item\label{itm: lnonhomog_cover} Every left non-homogeneous vertex $B_i$ that is adjacent to $C_j$ has degree at most two and satisfies $|B_i\cap C_j| = 1.$
        \item\label{itm: otherLeftVert_cover} Every left vertex that is either homogeneous or is not adjacent to $C_j$ has degree one and is adjacent to a vertex of the same type.
    \end{enumerate}
    \end{lem}

\begin{proof}
    Suppose that $\calC \lessdot \calB.$ Then, there are four possible cases of $\calC$ to check as shown in equation \eqref{eq: contract-an-edge}. We note that the unique right homogeneous vertex $C_j$ of degree two in $G(\calB,\calC)$ corresponds to the block $(\Bar{g}\cup \Bar{h})^\star$ of $\calC.$ It is easy to verify using these four cases that $G(\calB,\calC)$ satisfies the non-crossing property and the four conditions. Conversely, suppose that $G(\calB,\calC)$ is non-crossing and satisfies the two conditions. Then, it is also not difficult to check that $\calC$ can only have the form shown in equation $\eqref{eq: contract-an-edge}.$ Thus, $\calC \lessdot \calB.$
\end{proof}

\begin{lem}\label{lem: bipatie-properties}
    If $\calC \leq \calB$, then $G(\calB,\calC)$ satisfies conditions \ref{itm: left-nonhomog1}-\ref{itm: right-homog1} in Theorem \ref{thm: graph-characterization}
\end{lem}

\begin{proof}
    To prove this statement, we first establish the following two steps.  The first step is to show that the statement holds for every $G(\calB,\calC)$ such that $\calB$ is a cover of $\calC$. This is a straightforward application of Lemma \ref{lem: covering} and is left to the reader to verify. The second step is to show that for all $\calB,\calB',\calC$ such that $G(\calB,\calB')$ and $G(\calB',\calC)$ satisfy conditions \ref{itm: left-nonhomog1}-\ref{itm: right-homog1}, we must have that $G(\calB,\calC)$ also satisfy these conditions. To see this, suppose that $\calB,\calB',\calC$ are binary partitions in which $G(\calB,\calB')$ and $G(\calB',\calC)$ satisfy conditions \ref{itm: left-nonhomog1}-\ref{itm: right-homog1}. Let $B_i$ be a left homogeneous vertex of $G(\calB,\calB')$. By condition \ref{itm: left-homog1}, $B_i$ is adjacent to exactly one homogeneous vertex $B'_j$. Thus, $B_i \subseteq B'_j$. Similarly for $G(\calB',\calC)$, we have $B'_j \subseteq C_k$ for some right homogeneous vertex $C_k$. Hence, $B_i \subseteq C_k.$ Therefore, in $G(\calB,\calC)$, the left homogeneous vertex $B_i$ has degree one and is adjacent to the right homogeneous vertex $C_k.$ This shows that condition \ref{itm: left-homog1} holds for $G(\calB,\calC)$. 

    Now let $C_k$ be a right non-homogeneous vertex of $G(\calB',\calC)$. Since $G(\calB,\calB')$ and $G(\calB',\calC)$ satisfy condition \ref{itm: right-nonhomog1}, one can use a similar argument to show that, in $G(\calB,\calC)$, the left non-homogeneous vertex $C_k$ has degree one and is adjacent to a non-homogeneous vertex. Thus, condition \ref{itm: right-nonhomog1} holds for $G(\calB,\calC)$.

     Assume for the sake of contradiction that $G(\calB,\calC)$ doesn't satisfy condition \ref{itm: left-nonhomog1}. Then, there are a left non-homogeneous vertex $B_i$ and two right non-homogeneous vertices $C_{k_1}$ and $C_{k_2}$ adjacent to $B_i$ in $G(\calB,\calC)$. We deduce from conditions \ref{itm: right-nonhomog1} and \ref{itm: left-nonhomog1} for $G(\calB', \calC)$ that $C_{k_1} \subseteq B'_{j_1}$ and $C_{k_2} \subseteq B'_{j_2}$ for some distinct left non-homogeneous vertices $B'_{j_1}$ and $B'_{j_2}$ of $G(\calB',\calC)$. Since $B_i\cap C_{k_1} \neq \emptyset$ and $B_i\cap C_{k_2} \neq \emptyset$, it follows that $B_i\cap B'_{j_1} \neq \emptyset$ and $B_i\cap B'_{j_2} \neq \emptyset$. Hence, in $G(\calB,\calB')$, the left non-homogeneous $B_i$ is adjacent to the two non-homogeneous vertices $B'_{j_1}$ and $B'_{j_2},$ a contradiction to condition \ref{itm: left-nonhomog1} for $G(\calB,\calB')$.

     Now assume for the sake of contradiction that $G(\calB,\calC)$ doesn't satisfy condition \ref{itm: right-homog1}. Then, there is a right homogeneous vertex $C_k$ of degree one and a left non-homogeneous vertex $B_i$ adjacent to $C_k$. Thus, $C_k \subseteq B_i.$ Consider the following two cases. Cases I: $C_k$ is adjacent to a left homogeneous vertex $B'_j$ in $G(\calB',\calC).$ Applying condition \ref{itm: left-homog1} to $G(\calB',\calC)$, we have $B'_j \subseteq C_k$ and hence $B'_j \subseteq B_i$. Thus, in $G(\calB,\calB')$, the the right non-homogeneous vertex $B'_j$ has degree one and is adjacent to the non-homogeneous vertex $B_i$, a contradiction to condition \ref{itm: right-homog1} for $G(\calB,\calB')$. Case II: $C_k$ is not adjacent to any left homogeneous vertex in $G(\calB',\calC).$ By condition \ref{itm: right-homog1},  we deduce that $C_k$ must be adjacent to two non-homogeneous $B'_{j_1}$ and $B'_{j_2}$ of $G(\calB',\calC).$ Consequently, by condition \ref{itm: right-nonhomog1}, both $B'_{j_1}$ and $B'_{j_2}$ are vertices of $G(\calB,\calB')$ of degree one and are adjacent to $B_i.$ Hence, $B'_{j_1} \subseteq B_i$ and $B'_{j_2} \subseteq B_i$. In particular, this implies that, in $G(\calB,\calB')$, the left non-homogeneous vertex $B_i$ is adjacent to at least two non-homogeneous vertices, a contradiction to \ref{itm: left-nonhomog1} for $G(\calB,\calB')$. Since both cases lead to contradictions, we must have that $G(\calB,\calC)$ satisfies condition \ref{itm: right-homog1}.

     Suppose that $\calC \leq \calB$. Then, $\calC = \calB^t \lessdot \cdots \lessdot \calB^1 \lessdot \calB$ for some $\calB^1, \dots, \calB^t$. The two steps we established above implies that $G(\calB,\calC)$ satisfies conditions \ref{itm: left-nonhomog1}-\ref{itm: right-homog1} in Theorem \ref{thm: graph-characterization}.
\end{proof}

The proof of Lemma \ref{lem: bipatie-properties} shows that conditions \ref{itm: left-nonhomog1}-\ref{itm: right-homog1} in Theorem \ref{thm: graph-characterization} define a transitive relation on the set of all binary partitions on $[0,n].$ That is, we can define a transitive relation $\curly$ on the set of all binary partitions on $[0,n]$ by letting $\calB \curly \calC$ if $G(\calB,\calC)$ satisfies conditions \ref{itm: left-nonhomog1}-\ref{itm: right-homog1} in Theorem \ref{thm: graph-characterization}. 

\begin{lem}\label{lem: crossing-bipartite}
    Suppose that $\calB,$ $\calC = (C_1, \dots, C_q)$, and $\calD$ be binary partitions of $[0,n]$ in which both $G(\calB, \calC)$ and $G(\calC,\calD)$ are non-crossing.  If $G(\calB, \calD)$ is crossing, then there exists an integer $j$ such that $C_j$ is a vertex of degree at least two in $G(\calB,\calC)$ and a vertex of degree at least two in $G(\calC, \calD)$.
\end{lem}

\begin{proof}
    Suppose that $G(\calB, \calD)$ is crossing. Then there exists an $i$ such that $B_{i}$ is adjacent to $D_{k_1}$ and $B_{i+1}$ is adjacent to $D_{k_2}$ for some $k_1 > k_2.$ Moreover, because $G(\calB,\calC)$ is non-crossing, there must exists $j_1 \leq j_2$ such that such that $B_i$ is adjacent to $C_{j_1}$ in $G(\calB,\calC)$ and $C_{j_1}$ is adjacent to $D_{k_1}$ in $G(\calC, \calD)$, and $B_{i+1}$ is adjacent to $C_{j_2}$ in $G(\calB,\calC)$ and $C_{j_2}$ is adjacent to $D_{k_2}$ in $G(\calC, \calD)$. Since $G(\calC, \calD)$ is non-crossing, we deduce that $j_1 \geq j_2$. Thus, $j_1 = j_2.$ Let $j = j_1 = j_2.$ Then $C_j$ is adjacent to both $B_i$ and $B_{i+1}$ in $G(\calB, \calC)$ and is adjacent to both $D_{k_1}$ and $D_{k_2}$ in $G(\calC,\calD)$. Hence, we found an integer $j$ with the desired property.
\end{proof}

We can now give a proof of the characterization in Theorem \ref{thm: graph-characterization}.

\begin{proof}[Proof of Theorem \ref{thm: graph-characterization}]
    Suppose that $\calC \leq \calB$. Then by Lemma \ref{lem: bipatie-properties}, $G(\calB,\calC)$ must satisfy conditions \ref{itm: left-nonhomog1}-\ref{itm: right-homog1}. Thus, it only remains to be shown that $G(\calB,\calC)$ is non-crossing. Clearly, the non-crossing condition is satisfied when $\calC = \calB$. Thus, we may assume that $\calB \neq \calC$. Let $\calC = \calB^t \lessdot \calB^{t-1} \lessdot \cdots \lessdot \calB^0 = \calB$ be a maximal chain of strictly decreasing binary partitions from $\calB$ to $\calC$, i.e., $\calB^{i-1}$ is a cover of $\calB^{i}$ for all $i \in [t]$. To see that $G(\calB,\calC)$ is non-crossing, we proceed by induction on $t$. When $t = 1$, $\calB^0$ is a cover of $\calB^1$. Thus, by Lemma \ref{lem: covering}, $G(\calB^0,\calB^1)$ is non-crossing. This establishes the base case. Now suppose that $G(\calB^0,\calB^t)$ is non-crossing for a positive integer $t$. Assume for the sake of contradiction $G(\calB^0,\calB^{t+1})$ is crossing. Then by Lemma \ref{lem: crossing-bipartite}, there exists an integer $j$ such that $B^1_j$ is a right vertex of degree at least two of $G(\calB^0,\calB^1)$ and a left vertex of degree at least two of $G(\calB^1, \calB^{t+1})$. Because $\calB^1$ is a cover of $\calB^0,$ it follows from Lemma \ref{lem: covering} that $B^1_j$ is a right homogeneous vertex of $G(\calB^0,\calB^1)$. Condition \ref{lem: bipatie-properties}/\ref{itm: left-homog1} then implies that the vertex $B^t_j$ in $G(\calB^t, \calB^{t+1})$ has degree one, a contradiction. Thus, $G(\calB^{0},\calB^{t+1})$ must be non-crossing. By induction, $G(\calB^0,\calB^t) = G(\calB,\calC)$ is non-crossing.
    
    Conversely, suppose that $G(\calB, \calC)$ meets these conditions. Let us first consider $G(\calB, \calC)$ without any right vertex of $G(\calB,\calC)$ of degree greater than one. We claim that in this case $\calC = \calB$, which will automatically gives $\calC \leq \calB$ as desired. To see this, we show that every left vertex of $G(\calB,\calC)$ has degree one and is adjacent to a vertex of the same type. Assume that there is a left vertex, say $B_i$, of degree at least two. Then by condition \ref{itm: left-homog1}, the left vertex $B_i$ must be non-homogeneous. By condition \ref{itm: left-nonhomog1}, one of the right vertex adjacent to $B_i$, say $C_j$, has to be homogeneous. However, by condition \ref{itm: right-homog1}, $C_j$ must have degree at least two, a contradiction to the assumption $\deg(C_j) = 1$. Hence, every left vertex of $G(\calB,\calC)$ has degree one. From here, one can easily verify using conditions \ref{itm: left-homog1} and \ref{itm: right-nonhomog1} that every edge of $G(\calB, \calC)$ connects two vertices of the same types. This implies $\calC = \calB$ as claimed. 
    
    Now we consider $G(\calB, \calC)$ with at least one right vertex of degree at least two. To see that $\calC \leq \calB$, we will construct a maximal chain of strictly increasing binary partitions $\calC = \calB^t \lessdot \calB^{t-1} \lessdot \cdots \lessdot \calB^0 = \calB$. 
    
    Let $j$ be an integer such that $C_j$ is a right vertex of degree at least two. We note that condition \ref{itm: right-nonhomog1} implies that $C_j$ is a homogeneous vertex. Because $G(\calB,\calC)$ is a non-crossing bipartite graph, $C_j$ is adjacent to two consecutive blocks, say $B_i$ and $B_{i+1}$. This implies that there exist an equivalent class $\Bar{g} \subseteq B_i$ such that $\Bar{g} \subseteq B_i\cap C_j$ and an equivalent class $\Bar{h} \subseteq B_{i+1}$ such that $\Bar{h} \subseteq B_{i+1}\cap C_j$. Let $\calB^1$ be the binary partition corresponding to contracting the edge $\Bar{g}-\Bar{h}$ in the Hasse diagram of the preposet $([0,n], \preceq_\calB)$. Hence, by construction, $\calB^1 \lessdot \calB.$ One sees that $\Bar{g}\cup \Bar{h}$ is a homogeneous block of $\calB^1$ and that $G(\calB^1, \calC)$ has $\Bar{g}\cup\Bar{h}$ as a left vertex of degree one and adjacent $C_j.$  It is then easy to see by considering the four cases described in equation \eqref{eq: contract-an-edge} that $G(\calB^1, \calC)$ meets all of the five conditions. Now repeat the same construction with $G(\calB^1, \calC)$ to produce $\calB^2$ such that $\calB^2 \lessdot \calB^1.$ By repeatedly applying this procedure, we can eventually produce $\calB^t \lessdot \calB^{t-1} \lessdot \cdots \lessdot \calB^0 = \calB$ such that the every right vertex of $G(\calB^t, \calC)$ has degree one. This implies $\calC = \calB^t \leq \calB$ as desired. 
\end{proof}

\section{Skewed binary composition and skewed binary partition}

We now introduce "skewed binary partition" which is a special case of binary partition, and "skewed binary composition". These two combinatorial objects will provide us with sufficient information to describe the normal fans of parking function polytopes. Similarly to how a composition records the sizes of blocks in an ordered partition, a skewed binary composition will be used for storing information of the blocks of a skewed binary partition. We begin by describing the skewed binary composition notation. First, the entries of our composition are nonzero integers (as opposed to being positive integers). Additionally, we allow two different variations of the entries: $i^\circ$ and $i^\star$. We consider these two variations to have the same numerical values as $i,$ and use the absolute value sign to take their numerical values. Hence, $|i^\circ| = |i^\star| = i= |i|.$ 

For convenience, we let $\N^\circ := \{ i^\circ \ | \ i \in \N\}$,
$\P := \N_{>0}$ and $\P_{\ge 2}^\star :=\{ i^\star \ | \ i \in \P, i \ge 2\}.$

\begin{defn}
Let $n \in \P$ and $k \in \N.$ A \emph{skewed binary composition} of $n$ into $k+2$ parts is an ordered tuple $\bbb = (b_{-1}, b_0, b_1, \dots, b_k)$ such that $\sum^k_{i = -1} |b_i| = n$ and
the entries of $\bbb$ satisfy
\begin{align*}
    (b_{-1}, b_0) \in & \ (\N \times \N^\circ) \cup (\P \times \{0\}) \text{ and } b_i \in \P \cup \P_{\ge 2}^\star \text{ for all } 1 \le i \le k.
\end{align*}
\end{defn}

\begin{ex}
    The following are all possible skewed binary compositions of $n = 3.$
    \begin{gather*}
         (0,0^\circ,1,1,1), (0,0^\circ,1,2), (0,0^\circ,1,2^\star), (0,0^\circ,2,1), (0,0^\circ,2^\star,1),(0,0^\circ,3), (0,0^\circ,3^\star), \\
        (0,1^\circ,1, 1), (0,1^\circ,2), (0,1^\circ,2^\star), \\
        (1,0^\circ,1, 1), (1,0^\circ,2), (1,0^\circ,2^\star), \\
        (1,0,1, 1), (1,0,2), (1,0,2^\star), \\
        (0,2^\circ,1),(1, 1^\circ, 1), (2,0^\circ,1),   (2,0,1)\\
        (0,3^\circ), (1, 2^\circ), (2, 1^\circ),  (3,0^\circ), (3,0). 
    \end{gather*}
\end{ex}

Next, we introduce a similar notion to binary partition called \emph{ordered skewed binary partition} of the set $S = [0,n]$ by allowing empty blocks together with additional restrictions.

\begin{defn}\label{defn:bwp}
Let $n \in \P$ and $k \in \N.$ An \emph{(ordered) skewed binary partition} of $[0,n]$ into $k+2$ blocks is an ordered tuple $(B_{-1}, B_0, \dots, B_k)$ of disjoint subsets of $[0,n]$ such that \linebreak $B_{-1}\sqcup B_0 \sqcup B_1 \sqcup \cdots \sqcup B_k = [0,n]$ satisfying the following conditions:
    \begin{enumerate}[label=(S\arabic*)]
    \item\label{itm:firsttwo} $B_0$ is homogeneous, provided $|B_{0}| \geq 2$, and $B_{-1}$ is non-homogeneous.
    \item\label{itm:zerobelong} $0 \in B_{-1}$ or $ 0 \in B_0$. Moreover, if $0 \in B_{-1}$, then $B_{-1}$ contains at least another element and $B_0 = \emptyset$. Hence, if $ 0 \in B_{-1},$ then $|B_{-1}| \geq 2$ and $|B_0| = 0.$ 
    \item \label{itm:singleton} For each $0 \le i \le k$, if $B_i$ is a singleton, then it is non-homogeneous.
    \item $B_i \neq \emptyset$ for all $1 \le i \leq k.$
    \end{enumerate}
\end{defn}

See the first column of Table \ref{tab:ex-BWP} for examples of skewed binary partitions of $[0,8]$. Comparing Definition \ref{defn:bwp} to Definition \ref{defn: partition}, one sees that a skewed binary partition is simply a binary partition with extra requirements (conditions \ref{itm:firsttwo} and \ref{itm:zerobelong}).  In fact, removing empty blocks from a skewed binary partition yields a binary partition. For instance, removing the empty block from the skewed binary partition shown at the top of Table \ref{tab:ex-BWP} gives a binary partition in Example \ref{ex:binarypartition}. 
Thus, properties of binary partitions extend naturally to skewed binary partitions when regarded in this way.

\begin{table}
\centering
\begin{tabular}{r|l}
skewed binary partition $\calB$  & $\type(\calB)$   \\ \hline 
$(\{0,2,3\}, \emptyset, \{1,6,7\},\{8\}, \{4,5\})$ &  $(2, 0, 3, 1, 2)$ \\ \hline
$(\{2,3\}, \{0,7\}^\star, \{6\},\{1,8\}^\star, \{4,5\})$ &  $(2, 1^\circ, 1, 2^\star,2)$ \\ \hline
$(\{1,3,4,5,8\}, \{0\}, \{2\}, \{6,7\})$ &  $(5, 0^\circ, 1,2)$ \\ \hline
$(\emptyset, \{0\}, \{2,3,8\}, \{1,6,7\}^\star,\{4,5\})$ &  $(0, 0^\circ, 3,3^\star,2)$ \\ \hline
$(\{5,7\}, \{0,1,3\}^\star, \{2,4\}^\star, \{6,8\}^\star)$ & $(2, 2^\circ, 2^\star, 2^\star)$\\ \hline
$(\emptyset, \{0,1,2,3,4,5,6,7,8\}^\star)$ &  $(0, 8^\circ)$
\end{tabular}
\caption{Examples of skewed binary partitions and their types}
\label{tab:ex-BWP}
\end{table}

\begin{defn}\label{defn: bwp-preorder}
For a skewed binary partition $\calB$ of $[0,n]$, let $\hat{\calB}$ be the binary partition obtained by removing the empty block from $\calB.$ 
We define the associate preorder $\preceq_\calB$ on the set $[0,n]$ to be the preorder $\preceq_{\hat{\calB}}.$ We also say that a skewed binary partition $\calC$ is a contraction of another skewed binary partition $\calB$ if $\preceq_\calC$ is a contraction of $\preceq_\calB.$

We say a binary partition can be \emph{represented as a skewed binary partition} if it is equal to $\hat{\calB}$ for some skewed binary partition $\calB$.

\end{defn}

\begin{ex}\label{ex:normalcones}
See Figure \ref{fig:contractions} for examples of three skewed binary partitions $\calB, \calC$ and $\calD$ of $[0,8]$ together with their respective associated preorder cones. (Note that $\calB$ and $\calD$ are the first two skewed binary partitions given in Table \ref{tab:ex-BWP}.)

\begin{figure}
    \centering
\begin{tikzpicture}[scale = 0.76]
    \begin{scope}
        \node (zero) at (-2,-2) {0};
        \draw (-2,-2) circle (8pt);
        \node (min1) at (0,-2) {2};
        \draw (0,-2) circle (8pt);
        \node (min2) at (2,-2) {3};
        \draw (2,-2) circle (8pt);
        \node (a) at (-2,0) {1};
        \draw (-2,0) circle (8pt);
        \node (b) at (0,0) {6};
        \draw (0,0) circle (8pt);
        \node (c) at (2,0) {7};
        \draw (2,0) circle (8pt);
        \node (d) at (0,2) {8};
        \draw (0,2) circle (8pt);
        \node (max1) at (-1,4) {4};
        \draw (-1,4) circle (8pt);
        \node (max2) at (1,4) {5};
        \draw (1,4) circle (8pt);
        \draw (zero)--(a) (zero)--(b) (zero)--(c) (a)--(d)--(max1) (d)--(max2) (min1)--(a) (min1)--(b) (min1)--(c) (min2)--(a) (min2)--(b) (min2)--(c) (b)--(d) (c)--(d);
        \node (poset1) at (-4,1) {$([0,8], \preceq_\calB)$};
        \draw [latex-latex](3,1) -- (5,1);
        \node (B1) at (10,1.8) {$\calB = (\{0,2,3\}, \emptyset, \{1,6,7\},\{8\}, \{4,5\})$};
        \node (type) at (10,1) {$\type(\calB) = (2, 0, 3, 1, 2)$};
        \node (coneB1) at (10,0.2) {$\ssigma_\calB = \{0,c_2,c_3 \leq c_1,c_6,c_7 \leq c_8 \leq c_4, c_5 \}$};
        \draw [-latex,double](0,-3) -- (0,-5);
        \node (text1) at (2.3,-4) {contracting $1-8$};

        \node at (10, -1)
        {Note that the notation "$0, c_2, c_3 \le c_1, c_6, c_7$"};
        \node at (10, -1.6) {in the description of $\ssigma_\calB$ above means that};
        \node at (10, -2.2) {each of the three numbers on the left is};
        \node at (10, -2.8) {less or equal to each of the three numbers};
        \node at (10, -3.4) {on the right.};
    \end{scope}

    \begin{scope}[yshift=-10cm,xshift=0cm]
        \node (zero) at (-2,-2) {0};
        \draw (-2,-2) circle (8pt);
        \node (min1) at (0,-2) {2};
        \draw (0,-2) circle (8pt);
        \node (min2) at (2,-2) {3};
        \draw (2,-2) circle (8pt);
        \node (b) at (-1,0) {6};
        \draw (-1,0) circle (8pt);
        \node (c) at (1,0) {7};
        \draw (1,0) circle (8pt);
        \node (d) at (0,2) {$1,8$};
        \draw (0,2) ellipse (15pt and 10pt);
        \node (max1) at (-1,4) {4};
        \draw (-1,4) circle (8pt);
        \node (max2) at (1,4) {5};
        \draw (1,4) circle (8pt);
        \draw (zero)--(b) (zero)--(c) (d)--(max1) (d)--(max2) (min1)--(b) (min1)--(c) (min2)--(b) (min2)--(c) (b)--(d) (c)--(d);
        \node (poset1) at (-4,1) {$([0,8], \preceq_{\calC})$};
        \draw [latex-latex](3,1) -- (5,1);
        \node (B1) at (10,1.8) {$\calC = (\{0,2,3\}, \emptyset, \{6,7\},\{1,8\}^\star, \{4,5\})$};
        \node (type) at (10,1) {$\type(\calC) = (2, 0, 2, 2^\star,2)$};
        \node (coneB1) at (10,0.2) {$\ssigma_{\calC} = \{0,c_2,c_3 \leq c_6,c_7 \leq c_1=c_8 \leq c_4, c_5 \}$};
        \draw [-latex,double](0,-3) -- (0,-5);
        \node (text1) at (2.3,-4) {contracting $0-7$};
    \end{scope}

    \begin{scope}[yshift=-22cm,xshift=0cm]
        \node (min1) at (-1,-2) {2};
        \draw (-1,-2) circle (8pt);
        \node (min2) at (1,-2) {3};
        \draw (1,-2) circle (8pt);
        \node (b) at (0,0) {$0,7$};
        \draw (0,0) ellipse (15pt and 10pt);
        \node (c) at (0,2) {6};
        \draw (0,2) circle (8pt);
        \node (d) at (0,4) {$1,8$};
        \draw (0,4) ellipse (15pt and 10pt);
        \node (max1) at (-1,6) {4};
        \draw (-1,6) circle (8pt);
        \node (max2) at (1,6) {5};
        \draw (1,6) circle (8pt);
        \draw (d)--(max1) (d)--(max2) (min1)--(b)  (min2)--(b) (b)--(c) (c)--(d);
        \node (poset1) at (-4,2) {$([0,8], \preceq_{\calD})$};
         \draw [latex-latex](3,1) -- (5,1);
        \node (B1) at (10,1.8) {$\calD = (\{2,3\}, \{0,7\}^\star, \{6\},\{1,8\}^\star, \{4,5\})$};
        \node (type) at (10,1) {$\type(\calD) = (2, 1^\circ, 1, 2^\star,2)$};
        \node (coneB1) at (10,0.2) {$\ssigma_{\calD} = \{c_2,c_3 \leq 0 = c_7 \leq c_6\leq c_1=c_8 \leq c_4, c_5 \}$};

    \end{scope}
  
\end{tikzpicture}
    \caption{Both $\preceq_{\calC}$ and $\preceq_{\calD}$ are contractions $\preceq_{\calB}$}
    \label{fig:contractions}
\end{figure}
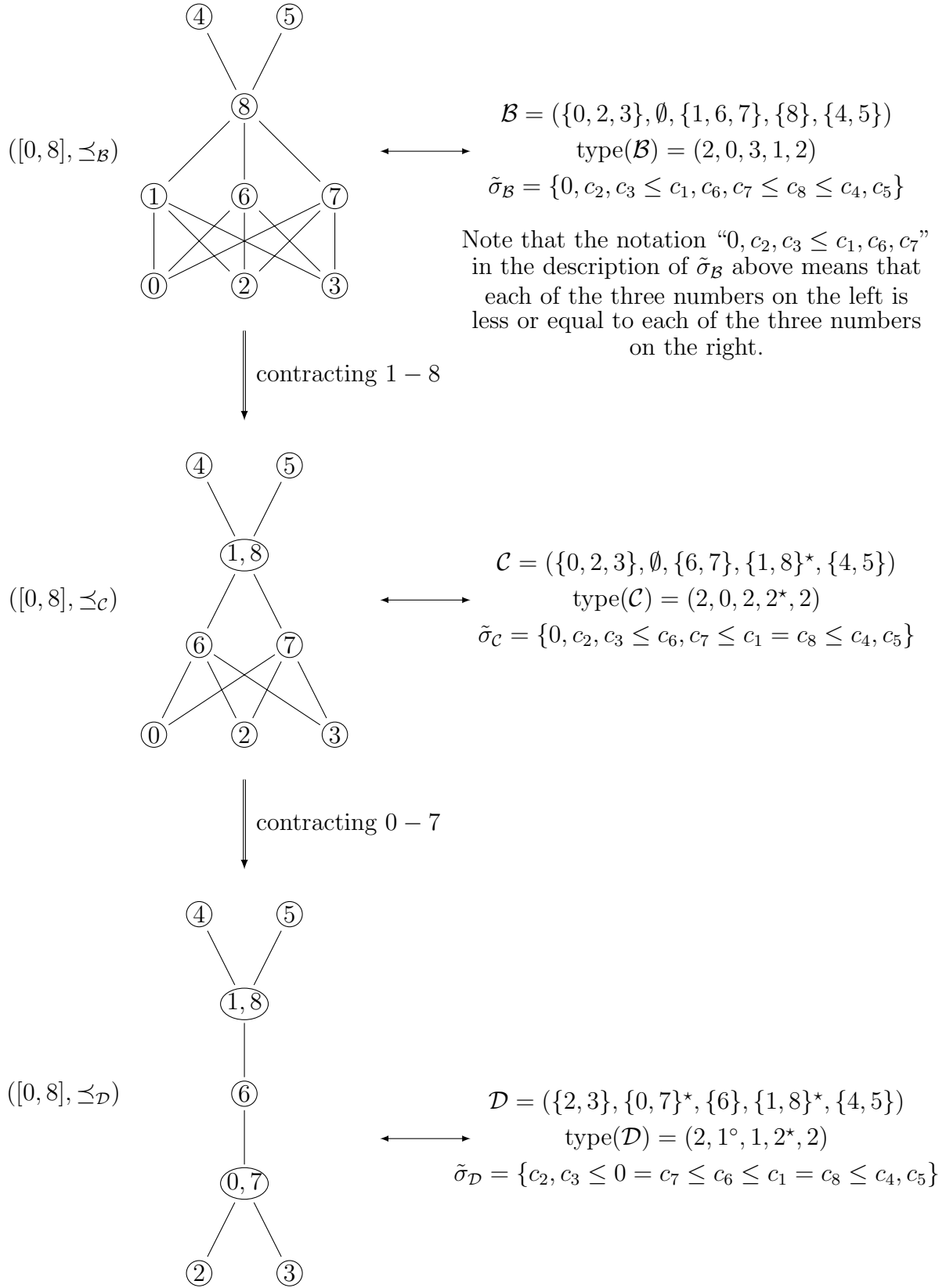
\end{ex}

In the example above, we see that contractions of a skewed binary partition are still skewed binary partitions. This is indeed true in general.

\begin{lem}\label{lem:contraction-skewed}
    Let $\calB$ be a skewed binary partition of $[0,n]$. Then every contraction of $\calB$ can be represented as a skewed binary partition of $[0,n]$. More precisely, if $\preceq$ is a contraction of $\preceq_\calB$, then there exists a skewed binary partition $\calC$ such that $\preceq_\calC$ is $\preceq.$ 
\end{lem}

\begin{proof}  
One sees that a binary partition $\calA = (A_1, A_2, \dots, A_p)$ can be represented as a skewed binary partition if and only if one of the followings is true:
\begin{enumerate}
    \item $0 \in A_1$;
    \item $0 \in A_2$,  $A_1$ is non-homogeneous, and $A_2$ is a homogenous block as long as $A_2$ contains at least two elements.
\end{enumerate}
It is easy to verify that the above property is preserved under contraction. This together with Lemma \ref{lem: contraction-representable} implies the conclusion of this lemma.
\end{proof}

Notice that we also include a column of "$\type(\calB)$" on the right of Table \ref{tab:ex-BWP}. We introduce this concept in the definition below.

\begin{defn}\label{defn:type}
Let $\calB= (B_{-1}, B_0, B_1, \dots, B_k)$ be a skewed binary partition of $[0,n]$. We associate a skewed binary composition $\bbb =(b_{-1}, b_0, \dots, b_k)$ of $n$ to it in the following way:
\begin{enumerate}
    \item For $1 \le i \le k$, let $b_i = |B_i|$ if $B_i$ is non-homogeneous, and $b_i = |B_i|^\star$ if $B_i$ is homogenous.
    \item If $0 \in B_0$, then let $b_0 = h^\circ$, where $h = |B_0|-1$, and let $b_{-1} = |B_{-1}|.$
    \item If $0 \in B_{-1}$, then let $b_0 = |B_0| = 0$ and let $b_{-1} = |B_{-1}|-1.$
\end{enumerate}
We say this vector $\bbb$ is the \emph{type} of $\calB$ and denote it by $\type(\calB)$.
\end{defn}

\begin{rem}
Suppose $B= (B_{-1}, B_0, B_1, \dots, B_k)$ has type $\bbb =(b_{-1}, b_0, \dots, b_k)$. 
It is easy to see that for $1 \le i \le k$, the number $b_i$ tells us the cardinality of $B_i$ and whether $B_i$ is homogeneous or not. In particular, Condition \ref{itm:singleton} of Definition \ref{defn:bwp} implies that $b_i \neq 1^\star$, and thus $b_i \in \P \cup \P_{\ge 2}^\star.$

For $i = -1$ or $0$, the number $|b_i|$ is the cardinality of $B_i \setminus \{0\}$.  Furthermore, one checks that
\begin{align}
 \text{   $0 \in B_0$} \quad \text{if and only if}& \quad (b_{-1}, b_0) \in \N \times \N^\circ, \quad \text{and} \label{eq:B0iff} \\
 \text{ $0 \in B_{-1}$} \quad \text{if and only if}& \quad  (b_{-1}, b_0) \in \P \times \{0\}. \label{eq:B-1iff}
\end{align}
Hence, the type of each skewed binary partition of $[0,n]$ is a skewed binary composition of $n$.

By \eqref{eq:B0iff}, we have that $0 \in B_0$ if and only if $b_0 \in \N^\circ$. 
Moreover, when $b_0 = h^\circ$, we know that $B_0$ consists of $h$ positive integers and $0$. This is the reason we use the notation $h^\circ$ in which the superscript $\circ$ indicates that $0$ needs to be included. 
\end{rem}

Applying Lemma \ref{lem: preordcone}/\ref{itm: dimension}, we can compute the dimension of the sliced preorder cone $\ssigma_\calB$ using its type vector $\type(\calB)$ as stated in the next proposition. 
\begin{prop}\label{prop: dim-preordcone}
Suppose that $\calB = (B_{-1}, B_0, B_1, \dots, B_k)$ is a skewed binary partitions of $[0,n]$
with $\type(\calB) = (b_{-1}, b_0, \dots, b_k)$. Then the dimension of the slice preorder cone $\ssigma_\calB$ is 
    \begin{equation}\label{eq:conedim}
        \dim(\ssigma_\calB) = \left( \sum_{b_i \in \P} b_i \right) +  \#(b_i \in \P_{\ge 2}^\star) 
    \end{equation}
    and the co-dimension of $\ssigma_\calB$ (with respect to the space $\R^n$) is
    \[|b_0| + \sum_{b_i \in \P_{\ge 2}^\star}(|b_i|-1).\]
\end{prop}

\begin{proof}
     By Lemma \ref{lem: preordcone}/\ref{itm: dimension}, the dimension of $\ssigma_\calB$ equals the number of equivalence classes of the preposet $([0,n],\preceq_\calB)$ minus one. 
     One observes that for each nonempty homogeneous block $B_i$, it gives arise one equivalence class of the preposet $([0,n],\preceq_\calB)$, and for each non-homogeneous block $B_i$, each singleton subset of $B_i$ is an equivalence class and hence it gives arise $|B_i|$ equivalence classes. Therefore, the total number of equivalences classes of $([0,n],\preceq_\calB)$ arising from $B_1, \dots, B_k$ is
     \[ 
\left(\sum_{b_i \in \P,\ i \geq 1} b_i \right) +  \#(b_i \in \P_{\ge 2}^\star)
\]
     and the total number equivalence classes arising from $B_{-1}$ and $B_0$ is
     \[ |B_{-1}| + \chi_{B_0 \neq  \emptyset},\]
where $\chi_{B_0 \neq \emptyset}$ is $1$ if $B_0$ is not the empty block $\emptyset$, and is $0$ otherwise. However, by Condition \ref{itm:zerobelong} of Definition \ref{defn:bwp}, we have that $B_0 = \emptyset$ if and only if $0 \in B_{-1}$. Then it follows from Definition \ref{defn:type} that $b_{-1} = |B_{-1}| - (1-\chi_{B_0 \neq  \emptyset}).$ One sees that \eqref{eq:conedim} follows from all the discussion above. 

    Finally, the co-dimension formula for $\ssigma_\calB$ follows from \eqref{eq:conedim} and the fact that $n = \sum_{i = -1}^k |b_i|.$
\end{proof}

The following result is an immediate consequence of Lemma \ref{lem: preordcone}/\ref{itm:poset}.
\begin{cor}
    Let $\calB$ be a skewed binary partition of $[0,n]$ and $\ssigma_B$ be the sliced preorder cone associated to $\calB$. If the preorder $\preceq_B$ defines a poset on $[0,n]$, then   
            \[\ssigma_\calB = \bigcup_{\pi \in L[\preceq_\calB]} \ssigma(\pi),\]
        where $L[\preceq_\calB]$ is the set of linear extensions of the poset $([0,n], \preceq_\calB)$ and $\ssigma(\pi)$ is defined as in \eqref{eq:sigmapi}.
\end{cor}

\subsection*{Standard skewed binary  partition}
It is easy to see that two skewed binary partitions are of the same type if and only if they differ from one another by a permutation of nonzero numbers between blocks. For convenience, we often focus on one particular skewed binary partition of a given type.

\begin{defn}\label{defn: standard-skewed binary} A skewed binary partition $\calB = (B_{-1}, B_0, B_1, \dots, B_p)$ is \emph{standard} if for every nonnegative integer $i \leq p$ every positive integer in $B_i$ is greater than every positive integer in $B_{i-1}.$
\end{defn}

One sees for every skewed binary partition $\calB$, there exists a unique standard one that has the same type as $\calB.$
\begin{ex}\label{ex:standard}
Let 
        \[\calB = (\{0,2,3\}, \emptyset, \{1,6,7\},\{8\}, \{4,5\}) \text{ and } \calC = (\{1,3,4,5,8\}, \{0\}, \{2\}, \{6,7\})\]
        be the skewed binary partitions on the first and third rows of Table \ref{tab:ex-BWP}.
        The unique standard skewed binary partitions of the same types are
        \[ \calB' = (\{0,1,2\}, \emptyset, \{3,4,5\},\{6\}, \{7,8\}) \text{ and } \calC' = (\{1,2, 3,4,5\}, \{0\}, \{6\}, \{7,8\}), \]
        respectively.
\end{ex}

We finish with a useful lemma regarding non-crossingness of the graph associated with standard skewed binary partitions, which follow directly from the definition of standard. 
\begin{lem}\label{lem: standard-crossing}
	Suppose that $\calA = (A_{-1}, A_0, A_1, \dots, A_q)$ and $\calB = (B_{-1}, B_0, B_1, \dots, B_p)$ are standard skewed binary partitions of $[0,n]$.
	Assume that $0 \in A_a$ and $0 \in B_b$ where $a, b \in \{-1, 0\}.$
	Then $G(\hat{\calA}, \hat{\calB})$ is non-crossing if and only if the edge connecting $A_a$ and $B_b$ does not intersect with (or cross) any other edges in  $G(\hat{\calA}, \hat{\calB})$.
\end{lem}

\section{Face Structure}

Recall that we define the parking function polytope $\pf(\bbu)$ as the convex hull of all $\bbu$-parking functions.  Although named a polytope, it is not immediately evident that $\pf(\bbu)$ qualifies as one, since it is defined as the convex hull of infinitely many points. The following proposition justifies its name.

\begin{prop}\label{prop:polytope}
    The parking function polytope $\pf(\bbu)$ is indeed a polytope, that is, it is a convex hull of finitely many points.
\end{prop}

\begin{defn}\label{defn:extreme}
    A point $\bbv = (v_1, \dots, v_n)\in \R^n$ is \emph{$\bbu$-extreme} if it is a permutation of a point of the form
    \begin{equation}\label{eq:extreme}
    (\underbrace{0,\dots, 0}_{k}, u_{k+1}, \dots, u_n)
    \end{equation} 
for some $0 \leq k \leq n.$ We denote by $\calX(\bbu)$ the set of all $\bbu$-extreme points.
\end{defn}

\begin{ex}\label{ex: Extremepoints}
    Let $\bbu=(0,0,4,4,4,6,8,8)$. Then $\calX(\bbu)$ is the set of all permutations of the $7$ points:
    \begin{center} $(0,0,4,4,4,6,8,8), ~(0, 0, 0, 4,4,6,8,8), ~(0, 0, 0, 0,4,6,8,8), ~(0, 0, 0, 0,0,6,8,8),$
    $(0, 0, 0, 0,0,0,8,8), ~(0, 0, 0, 0,0,0,0,8),$ and $(0, 0, 0, 0,0,0,0,0).$
    \end{center}
    \end{ex}

\begin{proof}[Proof of Proposition \ref{prop:polytope}]
    We claim that $\pf(\bbu) = \conv(\calX(\bbu))$ the convex hull of $\bbu$-extreme points. Since there are only finitely many $\bbu$-\emph{extreme} points, this will show that $\pf(\bbu)$ is indeed a polytope. 

    Since every $\bbu$-extreme point is a $\bbu$-parking function, we have $\conv(\calX(\bbu)) \subseteq \pf(\bbu).$ It remains to be shown that $\pf(\bbu) \subseteq \conv(\calX(\bbu))$. To see this, it suffices to show that every $\bbu$-parking function is a convex combination of $\bbu$-extreme points. 

 Suppose $\bba = (a_1, \dots, a_n)$ is a $\bbu$-parking function and $b_1 \leq \dots \leq b_n$ is the increasing rearrangement of $a_1, \dots, a_n$. Let $I_\bba$ be the set of indices $i$ such that $0 < b_i < u_i$. If $I_\bba$ is empty, we define the \emph{inner width} of $\bba$ to be $0$; otherwise, we let $t = \max(I_\bba)$ and $s$ be the least integer such that $0 < u_s$, and define the inner width of $\bba$ to be $t-s+1$. One sees that the inner width of $\bba$ is always nonnegative, and is $0$ if and only if $\bba$ is $\bbu$-extreme. We will prove that $\bba$ is a convex combination of $\bbu$-extreme points by induction on the inner width of $\bba.$

    The base case when the inner width of $\bba$ is $0$ is true, since $\bba$ itself is $\bbu$-extreme. Now suppose that $\bba$ has inner width $t-s+1 \ge 1$, and that every $\bbu$-parking function of inner width less than $t+s-1$ is a convex combination of $\bbu$-extreme points. 
    Let $\tau \in \fS_n$ be a permutation such that $(a_{\tau(1)}, \dots, a_{\tau(n)}) = (b_{1}, \dots, b_{n})$ and let $\tau(k) = t.$ Then
    \[(a_1, \dots, a_n) = \frac{u_t - a_k}{u_t}(a_1, \dots, a_{k-1}, 0, a_{k+1}, \dots, a_n) + \frac{a_k}{u_t}(a_1, \dots, a_{k-1}, u_t, a_{k+1}, \dots, a_n)\]
    is a convex combination of $\bbu$-parking functions $\bba^* := (a_1, \dots, a_{k-1}, 0, a_{k+1}, \dots, a_n)$ and \\
    $\bba' := (a_1, \dots, a_{k-1}, u_t, a_{k+1}, \dots, a_n)$. Notice that the inner width of both $\bba^*$ and $\bba'$ decrease from the inner width of $\bba$ by at least one. Thus, by our induction hypothesis, we have that both $\bba^*$ and $\bba'$ are convex combinations of $\bbu$-extreme points. Hence, $\bba$ is also a convex combination of $\bbu$-extreme points, completing the proof. 
\end{proof}

It turns out that $\calX(\bbu)$ is also the set of all vertices of $\pf(\bbu).$ One can give a proof of this directly, but we will apply Lemma \ref{lem:nfan} to compute the normal cones of $\pf(\bbu)$ and obtain this as a consequence later in Theorem \ref{thm: fulldimcones}. Hence, the parking function polytope is integral (resp. rational) if and only if $\bbu$ is an integral (resp. rational) vector in $\R^n.$

Since permuting the coordinates of a $\bbu$-parking function $(a_1, \dots, a_n)$ still gives a $\bbu$-parking function, the polytope $\pf(\bbu)$ itself also inherits this symmetric property. The next proposition restates this observation more precisely.

\begin{prop} For every parking function polytope $\pf(\bbu),$ we have that
\begin{enumerate}
    \item If $(a_1, \dots, a_n)$ lies in $\pf(\bbu)$, then so does every permutation of $(a_1, \dots, a_n)$.
    \item If $F$ is a face of $\pf(\bbu),$ then for every permutation $\tau \in \fS_n$ the set $$F_\tau := \{(a_{\tau(1)}, \dots, a_{\tau(n)}) \in \R^n \mid (a_1, \dots, a_n) \in F \}$$ is also a face of $\pf(\bbu)$.
    \item If $\sigma$ is a normal cone of $\pf(\bbu)$ at a face $F,$ then for every permutation $\tau \in \fS_n$ the set $\sigma_\tau := \{(c_{\tau(1)}, \dots, c_{\tau(n)}) \in \R^n \mid (c_1, \dots, c_n) \in \sigma \}$ is the normal cone of $\pf(\bbu)$ at the face $F_\tau$.
\end{enumerate}
    
\end{prop}

\subsection{Face Poset and Normal Fan}

In this part, we will study the face poset and the normal fan of the parking polytope $\pf(\bbu).$ By Lemma \ref{lem: normalfan-faceposet}, for every polytope $P$, the dual poset of $\calF(P)$ is isomorphic to the poset $\calF(\Sigma(P)).$ Therefore, rather than describing the face poset of parking function polytopes, we can alternatively describe their normal fans. It turns out that these fans only depend on the \emph{multiplicity vector} of $\bbu.$

\begin{defn}\label{defn:multiplicity}
Suppose that $\bbu \neq \0$ is nonzero, and there are $\ell \ge 1$ positive integers appearing in $\bbu$: $d_1 < d_2< \cdots < d_\ell.$ We define $m_0(\bbu)$ to be the number of $0$'s in $\bbu$, and $m_i(\bbu)$ be the number of $d_i$'s in $\bbu$ for each $1 \le i \le \ell$. We then define the \emph{multiplicity vector} of $\bbu$ to be $\bbm(\bbu)=(m_0(\bbu), m_1(\bbu), \dots, m_\ell(\bbu))$ and the \emph{data vector} of $\bbu$ to be $\bbd(\bbu)=(d_1, d_2, \dots, d_\ell).$ We call $(\bbm, \bbd)$ the \emph{MD pair} of $\bbu$.
\end{defn}

\begin{ex}
    If $\bbu=(0,0,4,4,4,6,8,8)$, then $\bbm(\bbu)=(2,3,1,2)$ and $\bbd(\bbu)=(4,6,8).$ 
\end{ex}

We say that $\bbd=(d_1, \dots, d_\ell)$ is a \emph{data vector} if it is a data vector of some nonzero $\bbu$, and $\bbm = (m_0, m_1, \dots, m_\ell)$ is a \emph{multiplicity vector} of \emph{magnitude} $n= m_0 +m_1 + \cdots m_\ell$ if it is a multiplicity vector of some nonzero $\bbu$. (Note that the magnitude is the length of $\bbu$.) Clearly, $\bbm=(m_0, m_1, \dots, m_\ell)$ is a multiplicity vector magnitude $n$ if and only if $\bbm$ is a weak composition of $n$ and $\bbm \neq (n)$. Hence, there are $2^n-1$ multiplicity vectors of magnitude $n.$ Finally, we say $(\bbm, \bbd)$ is an \emph{MD pair} if $(\bbm, \bbd)$ is the MD pair of some nonzero $\bbu.$ 

\begin{rem}\label{rem:mzero1}
Because a multiplicity vector $\bbm$ of magnitude $n$ can never be $(n)$, we always have that $m_0 < n.$
\end{rem}

It is clear that starting with an MD pair $(\bbm, \bbd),$ there exists a unique $\bbu$ such that $(\bbm, \bbd)$ is its MD pair.
Hence, we may interchangeably use $(\bbm, \bbd)$ for $\bbu$ and write $\pf(\bbm, \bbd)$ as $\pf(\bbu)$. As we mentioned above, the face poset and the normal fan of the parking polytope $\pf(\bbm,\bbd)$ only depend on the multiplicity vector $\bbm.$ Therefore, we will mostly use the notation $\pf(\bbm,\bbd)$ in this section.

\begin{defn}\label{def: vertex-type}
    Suppose $\bbm = (m_0, m_1, \dots, m_\ell)$ is a multiplicity vector of magnitude $n$. Let $r=n-m_0= m_1 + \cdots m_\ell$. We let $\bba_0, \dots, \bba_r$ be the following $r+1$ skewed binary compositions of $n$:
\begin{enumerate}
    \item We let \begin{align*}
        \bba_0 := \begin{cases}
            (m_0, 0, m_1, \dots, m_\ell) &\text{ if } m_0 > 0\\
            (0, 0^\circ, m_1, \dots, m_\ell) &\text{ if } m_0 = 0
        \end{cases}.
    \end{align*}
    \item Suppose $1 \le i \le r$. Let $g$ be the unique positive integer in which 
    $m_1 + \cdots + m_{g-1} \le i < m_1 + \cdots + m_g.$
    We define $\bba_i := (m_0 + i, 0^\circ, m_1 + \cdots + m_g - i, m_{g+1}, \dots, m_\ell).$ By convention, if $i = r = m_1 + m_2 + \cdots + m_\ell$, we choose $g= \ell+1$ and $\bba_r := (n, 0^\circ).$
\end{enumerate}
We denote by $\Omega_{\bbm}$ the set of these $r+1$ skewed binary compositions of $n$.
\end{defn} 

\begin{rem}
    \label{rem:mzero2}
    For any $(a_{-1}, a_0, \dots, a_q) \in \Omega_\bbm,$ we have $|a_{-1}| \ge m_0$ and $|a_0| = 0$. 
\end{rem}

\begin{ex}\label{ex: Omega_m}
Let $\bbm=(2, 3, 1, 2)$ and $\bbm'= (0, 3, 5)$ be two multiplicity vectors of magnitude $8$. Then
\begin{align*}
    \Omega_\bbm &= \{(2, 0, 3, 1, 2), (3, 0^\circ,2,1,2), (4, 0^\circ ,1,1,2), (5,0^\circ,1,2), (6,0^\circ,2), (7,0^\circ,1), (8,0^\circ)\}\\
    \Omega_{\bbm'} &= \{(0, 0^\circ, 3,5), (1, 0^\circ, 2,5), (2, 0^\circ, 1, 5), (3, 0^\circ, 5), (4, 0^\circ, 4), (5, 0^\circ, 3), (6,0^\circ, 2),\\
    & \qquad (7, 0^\circ, 1), (8, 0^\circ)\}.
\end{align*}
\end{ex}

\begin{rem}\label{rem: omega-full-dim-cones}
    It is easy to see that every preposet $([0,n], \preceq_\calB)$ in which $\type(\calB) \in \Omega_\bbm$ is a poset (every equivalent class in $([0,n], \preceq_\calB)$ is a singleton) whose Hasse diagram is a connected graph. Hence, by Lemma \ref{lem: preordcone}/\ref{itm: dimension}, every sliced preorder cone $\ssigma_\calB$ such that $\type(\calB) \in \Omega_\bbm$ is an $n$-dimensional cone.
\end{rem}

\begin{prop}\label{dissection}
    Suppose that $\bbm = (m_0, m_1, \dots, m_\ell)$ is a multiplicity vector of magnitude $n$. Let
    \[\calM := \lbrace \ssigma_\calB \subseteq \R^n\ \Big| \ \type(\calB) \in \Omega_\bbm \rbrace\]
     be the collection of sliced preorder cones corresponding to skewed binary partitions whose types are in $\Omega_\bbm.$ Then we have that
     \[\bigcup_{\ssigma \in \calM} \ssigma = \R^n.\]
\end{prop}
\begin{proof}
    Let $\bbw = (w_1, \dots, w_n) \in \R^n.$ We need to show that there exists $\ssigma_\calB \in \calM$ such that $\bbw \in \ssigma_\calB$. To see this, it suffices, due to symmetry, to only consider $w_1 \leq \dots \leq w_n.$ 
    
    If $w_n \leq 0,$ then $\bbw$ lies in the cone
    \[\{\bbc \in \R^n\ |\ c_1, \dots, c_n \leq 0\} = \ssigma_\calB \]
    where $\calB = (\{1, \dots, n\}, \{0\})$. Since $\type(\calB) = (n, 0^\circ) \in \Omega_\bbm$, it follows that $\ssigma_\calB \in \calM.$ 
    
    If $w_n > 0$, we let $i \in [n]$ be the least positive integer such that $w_i > 0$. For $s \in [0,\ell]$, let $t_s := m_0 + m_1 + \cdots + m_s.$ If $m_0 < i,$ by letting $j$ be the least integer such that $i \leq t_j$, we see that $j \geq 1$ and that $\bbw$ lies in the cone
    \[\{\bbc \in \R^n\ |\ c_1, \dots, c_{i-1} \leq 0 \leq c_{i}, c_{i+1},\dots, c_{t_j} \leq c_{t_j + 1}, \dots, c_{t_{j+1}} \leq \cdots \leq c_{t_{\ell-1}+1}, \dots, c_{t_{\ell}}\} = \ssigma_\calB\] 
    where $\calB =(\{1, \dots, i-1\}, \{0\}, \{i, i+1, \dots, t_j\}, \{t_j +1, \dots, t_{j+1}\}, \dots, \{t_{\ell-1}+ 1, \dots, t_{\ell}\}).$
    Because $\type(\calB) = (i-1, 0^\circ, m_1 + \cdots + m_j - i, m_{j+1}, \dots, m_\ell) \in \Omega_\bbm,$ we see that $\ssigma_\calB \in \calM.$ If $i \leq m_0$, then $0 < m_0 = t_0$ and $\bbw$ lies in the cone  
    \[\{\bbc \in \R^n\ |\ 0, c_1, \dots, c_{t_0} \leq c_{t_0+1}, \dots, c_{t_1} \leq \cdots \leq c_{t_{\ell-1}+1}, \dots, c_{t_{\ell}}\} = \ssigma_\calB \] 
    where $\calB = (\{0,1, \dots, t_0\}, \emptyset, \{t_0 +1, \dots, t_1\}, \dots, \{t_{\ell-1} +1, \dots, t_\ell\}).$ Since $\type(\calB) =$ \linebreak $ (m_0, 0, m_1, \dots, m_\ell) \in \Omega_\bbm,$ we have  $\ssigma_\calB \in \calM.$ This shows that $\bbw$ lies in some $\ssigma \in \calM.$ Hence, $\bigcup_{\ssigma \in \calM} \ssigma = \R^n.$
\end{proof}

\begin{defn}\label{def: v_b}
    Suppose that $(\bbm,\bbd)$ is an MD pair in which $\bbd = (d_1, \dots, d_\ell)$ and $\bbm = (m_0, m_1, \dots, m_\ell)$ is a multiplicity vector of magnitude $n$. 
    Let $\calB = (B_{-1}, B_0, B_1, \dots, B_k)$ be a skewed binary partition such that $\type(\calB) \in \Omega_\bbm.$ We define $\bbv_\calB = (v_1, \dots, v_n)$ to be the point in $\R^n$ given by
    \begin{align*}
        v_i = \begin{cases}
            0 &\text{ if } i \in B_{-1}\\
            d_{\ell - k + j} &\text{ if } i \in B_j \text{ for some } j \in [\ell]
        \end{cases}.
    \end{align*}
\end{defn}

\begin{ex}\label{ex:vertexcorrespondence} Consider the MD pair $(\bbm, \bbd)$ where $\bbm=(2, 3, 1, 2)$ and $\bbd = (4,6,8)$ that correspond to $\bbu = (0,0,4,4,4,6,8,8).$ Let 
        \[\calB = (\{0,2,3\}, \emptyset, \{1,6,7\},\{8\}, \{4,5\}) \text{ and } \calC = (\{1,3,4,5,8\}, \{0\}, \{2\}, \{6,7\})\]
        be the two skewed binary partitions given in Example \ref{ex:standard}. 
        
One sees that $\type(\calB)=(2,0,3,1,2)$ and $\type(\calC) = (5,0^\circ, 1,2)$, which are elements in $\Omega_\bbm$ as shown in Example \ref{ex: Omega_m}. Following Definition \ref{def: v_b}, we have
    \begin{align*}
        \bbv_\calB = (4,0,0,8,8,4,4,6) \text{ and } \bbv_{\calC} =  (0,6,0,0,0,8,8,0),
    \end{align*}
    which are two $\bbu$-extreme points in $\pf(\bbm,\bbd)=\pf(\bbu)$.
If we let $\calB'$ and $\calC'$ be the unique standard skewed binary partitions of types $(2,0,3,1,2)$ and $(5,0^\circ, 1,2)$ as in Example \ref{ex:standard}, then we get two more $\bbu$-extreme points
\begin{align*}
        \bbv_{\calB'} = (0,0,0, 4, 4, 4, 6, 8,8) \text{ and } \bbv_{\calC'} =  (0,0,0,0,0, 6, 8,8).
    \end{align*}
    
    \end{ex}

    One observes that the set of  
    $\bbu$-extreme points
    given in Example \ref{ex: Extremepoints} is exactly the set $\{\bbv_\calB \mid \type(\calB) \in \Omega_\bbm\}$. This is true in general for any parking function polytope $\pf(\bbm, \bbd)$ as we discuss in the remark below. 
    \begin{rem}\label{rem:vertex}
        Suppose that $(\bbm,\bbd)$ is the MD pair of $\bbu = (u_1, \dots, u_n)$ and $\calB$ is a skewed binary partition with $\type(\calB)=(a_{-1}, a_0, a_1, \dots, a_k) \in \Omega_\bbm$. 
        If $\calB$ is standard, then $\bbv_\calB = (w_1, \dots, w_n) := (0,\dots, 0, u_{a_{-1}+1}, \dots, u_n),$ in which the last $(n-a_{-1})$ terms are the same last $(n-a_{-1})$ positive numbers in $\bbu$.   
        
        In general (if $\calB$ is not standard), we can equivalently define $\bbv_\calB = (v_1, \dots, v_n)$ as the point that is the permutation of the point $(w_1, \dots, w_n)$  
	satisfying $v_i = w_t$ if $i \in B_j$, where $t = a_{-1} + a_1 \cdots + a_{j}.$ Hence, it follows that 
	\begin{equation}\label{eq:extreme=vB}
	\calX(\bbu) = \{\bbv_\calB \mid \type(\calB) \in \Omega_\bbm\},
\end{equation}
where $\calX(\bbu)$ is the set of  $\bbu$-extreme points. 
    \end{rem}
 It follows from Proposition \ref{prop:polytope} and the above remark that 
    \[ \pf(\bbm,\bbd) = \conv(\bbv_\calB \mid \type(\calB) \in \Omega_\bbm).\] 
The next theorem shows that $\{\bbv_\calB \mid \type(\calB) \in \Omega_\bbm\}$ is precisely the set of vertices of $\pf(\bbm, \bbd)$ and describes the normal cone at each vertex.

\begin{thm}\label{thm: fulldimcones}
    Let $(\bbm, \bbd)$ be an MD pair. Then, the map $\calB \mapsto \bbv_\calB$ defines a bijection between the set $\{\calB \mid \type(\calB) \in \Omega_\bbm\}$ and the set of the vertices of $\pf(\bbm,\bbd)$. Moreover, we have that $\ncone(\bbv_\calB,\pf(\bbm,\bbd)) = \ssigma_\calB.$
\end{thm}

\begin{proof} Assume that $\bbu=(u_1,\dots,u_n)$ is the non-increasing vector corresponding to the MD pair $(\bbm, \bbd)$.
    Suppose that $\calB = (B_{-1}, B_0, B_1, \dots, B_k)$ satisfies $ \type(\calB) = (a_{-1}, a_0, \dots, a_k) \in \Omega_\bbm$. By Theorem \ref{dissection} and Lemma \ref{lem:nfan}, it suffices to show that, for every $\bbc \in \ssigma_\calB^\circ$ and every $\calB'$ such that $\type(\calB') \in \Omega_\bbm$, one has
    \begin{align}\label{neq: vertex-normal-cone}
        \bbc \cdot \bbv_\calB > \bbc \cdot \bbv_{\calB'} \text{ if } \calB \neq \calB'.
    \end{align}

    First, we may assume without loss of generality that $\calB$ is standard. (If $\calB$ is not standard, we can reorder the coordinates to make $\calB$ standard.)  Hence, by Remark \ref{rem:vertex}, we have that \[ \bbv_\calB = (v_{\calB,1}, \dots, v_{\calB,n}) = (0,\dots, 0, u_{a_{-1}+1}, \dots, u_n).\]
    Let $\bbc = (c_1, \dots, c_n) \in \ssigma_\calB^\circ$. Then, by Lemma \ref{lem: preordcone}/\ref{itm:minimaldescription}, the vector $\bbc$ satisfies 
   \[c_p <  c_q \text{ if } p \in B_i, q \in B_j, \text{ and } i < j\]
   where we set $c_0 = 0.$ One sees that $\bbc$ has the following properties:
   \begin{enumerate}[label=(C\arabic*)]
   \item\label{itm:c-} $c_p < 0$ for all $1 \le p \le a_{-1}$ (if there is any), 
   \item \label{itm:c+} $0 < c_p$ for all $p \geq a_{-1} + 1$, and 
   \item \label{itm:corder} $c_p < c_q$ provided $v_{\calB,p} < v_{\calB,q}$.
   \end{enumerate}
    By property \ref{itm:corder}, one sees that we can further assume without loss of generality that $\bbc$ is a non-decreasing vector. (By reordering coordinates, we can make $\bbc$ non-decreasing without changing $\calB$ or $\bbv_\calB$.) 

    Let $\calB' = (B'_{-1}, B'_0, B'_1, \dots, B'_t)$ be a skewed binary partition that satisfies $\type(\calB') = (a'_{-1}, a'_0, \dots, a'_t) \in \Omega_\bbm$.
Let $\calD$ be the standard skewed binary partition such that $\type(\calD) = \type(\calB').$ Then by Remark \ref{rem:vertex} again, we have that \[ \bbv_{\calD} = (0,\dots, 0, u_{a'_{-1}+1}, \dots, u_n).\]

We claim that
\begin{equation}\label{neq: vertex-normal-cone0}
\bbc \cdot \bbv_\calB \ge \bbc \cdot \bbv_\calD \geq \bbc \cdot \bbv_{\calB'},
\end{equation}
and both equalities hold if and only if $\calB = \calB'$. It is clear that \eqref{neq: vertex-normal-cone} follows from this claim. Hence, it is left to prove the claim.  Since $\bbv_\calD$ is obtained by permuting the coordinates of $\bbv_{\calB'}$ in a non-decreasing fashion, the second inequality follows from the rearrangement inequality. 
 
For the first inequality, we will show that strict inequality holds if $\calB \neq \calD.$ Suppose $\calB \neq \calD$. Because both $\calB$ and $\calD$ are standard, we must have that $a_{-1} \neq a'_{-1}$. If $a_{-1} > a'_{-1}$, then $\bbv_\calB$ and $\bbv_\calD$ differ in the terms between the $(a'_{-1}+1)$st position and the $(a_{-1})$th position, and we have
\[ \bbc \cdot \bbv_\calB - \bbc \cdot \bbv_\calD = \sum_{i=a'_{-1}+1}^{a_{-1}} c_i (0-u_i) = -\sum_{i=a'_{-1}+1}^{a_{-1}} c_i u_i >0,\]
where the last equality follows from property \ref{itm:c-}. If $a_{-1} < a'_{-1}$, then $\bbv_\calB$ and $\bbv_\calD$ differ in the terms between the $(a_{-1}+1)$st position and the $(a'_{-1})$th position, and we have
\[ \bbc \cdot \bbv_\calB - \bbc \cdot \bbv_\calD = \sum_{i=a_{-1}+1}^{a'_{-1}} c_i (u_i-0) = \sum_{i=a_{-1}+1}^{a'_{-1}} c_i u_i >0,\]
where the last equality follows from property \ref{itm:c+}. Therefore, we showed that the first inequality in \eqref{neq: vertex-normal-cone0} holds and its equality holds if and only if $\calB =\calD.$

We have established both inequalities in \eqref{neq: vertex-normal-cone0}. Now suppose both equalities hold:
\[ \bbc \cdot \bbv_\calB = \bbc \cdot \bbv_\calD = \bbc \cdot \bbv_{\calB'}.\]
As we argued above, the first equality implies that $\calB = \calD.$ Hence, we have that $\bbv_\calB'$ is a permutation of $\bbv_\calD=\bbv_\calB.$ However, since $\bbv_\calB$ is non-decreasing and the vector $\bbc$ satisfies property \ref{itm:corder}, it follows from the rearrangement inequality that we must have $\calB'=\calB.$ This completes the proof of the claim, and thus the proof of this theorem. 
\end{proof}

Theorem \ref{thm: fulldimcones} describes the full dimensional cones in $\Sigma(\pf(\bbm, \bbd))$, the normal fan of $\pf(\bbm, \bbd)$. Since the lower dimensional cones are faces of the full dimensional cones in $\Sigma(\pf(\bbm, \bbd))$, it follows from Lemma \ref{lem: preordcone}/\ref{itm: conefaces} that they correspond to contractions of the preorders $\preceq_\calB$ on $[0,n]$ where $\calB$ is a skewed binary partition such that $\type(\calB) \in \Omega_\bbm.$ Note that it follows from Lemma \ref{lem:contraction-skewed} that any such a contraction can be represented as a preorder associated with a skewed binary partition. Therefore, we give the following definition, and have an immediate corollary to Theorem \ref{thm: fulldimcones} and Lemma \ref{lem: preordcone}/\ref{itm: conefaces}.

\begin{defn}
    Let $\bbm$ be a multiplicity vector. 
    We define $\bwp(\bbm)$ to be the poset of all skewed binary partitions 
    $\calB$ such that $\calB$ is a contraction of some skewed binary partition $\calA$ satisfying $\type(\calA) \in \Omega_\bbm$, 
    ordered by contraction, i.e., $\calB, \calB' \in \bwp(\bbm)$ satisfty $\calB \leq \calB'$ if $\calB$ is a contraction of $\calB'$.
\end{defn}

\begin{cor}\label{cor: combinatorialType}
        Let $(\bbm,\bbd)$ be an MD pair. Then the map $\calB \mapsto \ssigma_\calB$ is a poset isomorphism from $\bwp(\bbm)$ to $\Sigma(\pf(\bbm,\bbd))$.  Therefore, the face poset of $\pf(\bbm,\bbd)$ is isomorphic to the dual of $\bwp(\bbm).$
\end{cor}

Thus, the combinatorial types of parking function polytopes only depend on the multiplicity vector, i.e., two parking functions polytopes $\pf(\bbu_1)$ and $\pf(\bbu_2)$ have isomorphic face posets if $\bbm(\bbu_1) = \bbm(\bbu_2).$ As mentioned earlier, there are $2^n -1$ distinct multiplicity vectors $\bbm$ of magnitude $n$ which define $n$-dimensional parking functions $\pf(\bbm,\bbd)$. Figure \ref{fig: pf-types-dim2} shows that there are exactly three different combinatorial types of $2$-dimensional parking function polytopes $\pf(\bbm,\bbd)$, and two distinct multiplicity vectors correspond to different types.
\begin{figure}
    \centering
    \includegraphics[width=0.9\linewidth]{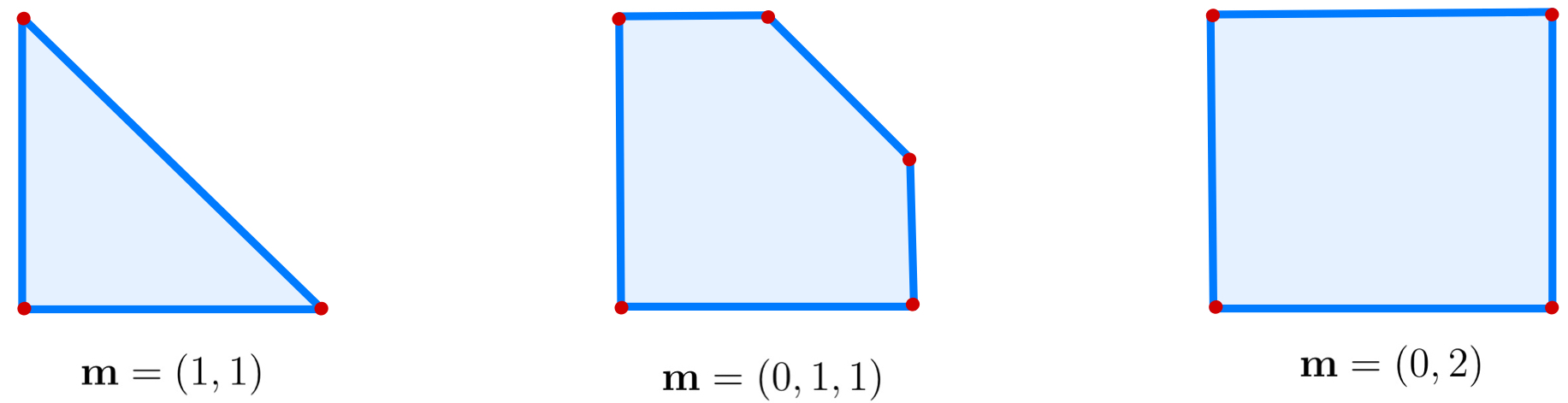}
    \caption{Three different combinatorial types of $2$-dimensional $\pf(\bbm,\bbd)$}
    \label{fig: pf-types-dim2}
\end{figure}

We finish this subsection by proving Proposition \ref{prop: pf-stellahedron}, an important consequence of Theorem \ref{thm: fulldimcones}.

\begin{proof}[Proof of Proposition \ref{prop: pf-stellahedron}]
  The multiplicity vector of $(1,2,\dots,n)$ is $(0, 1, 1, \dots, 1)$. Let $\Psi_n^{(0)} = \Omega_{(0,1,1,\dots,1)}$, and let
  $\Psi_n$ be the set of all skewed binary compositions $\bbb$ that is in $\Omega_\bbm$ for some multiplicity vector $\bbm$ of magnitude $n$. By Definition \ref{def: vertex-type}, one sees that $\Psi_n$ consists of all skewed binary compositions $\bbb=(b_{-1}, b_0, b_1, \dots, b_k)$ of $n$ satisfying 
  \[ (b_{-1}, b_0) \in \N \times \{0^\circ\} \cup \P \times \{0\}, \text{ and } b_i \in \P \text{ for all $1 \le i \le k.$}\]
  and $\Psi_n^{(0)}$ consists of all skewed binary compositions $\bbb=(b_{-1}, b_0, b_1, \dots, b_k)$ in $\Psi_n$ satisfying 
  \[ (b_{-1}, b_0) \in \N \times \{0^\circ\}, \text{ and } b_i =1 \text{ for all $1 \le i \le k.$}\] 
  Note that the number $k$ in both expressions is an arbitrary integer between $0$ and $n$.
  
  Recall that to show a fan $\Sigma'$ is a coarsening of another fan $\Sigma$, it suffices to show that every maximal cones of $\Sigma'$ is the union of a collection of maximal cones in $\Sigma.$ Hence, by Theorem \ref{thm: fulldimcones},  it is enough to show that 
\begin{equation}\label{eq:unioncones}
      \textbox{for every skewed binary partitions $\calB$ with $\type(\calB) \in \Psi_n$, there exists a collection of skewed binary partitions $\calA_1$, $\calA_2$, \dots, $\calA_k$ such that $\type(\calA_i) \in \Psi_n^{(0)}$ and $\displaystyle \ssigma_\calB = \cup_{i=1}^k \ssigma_{\calA_i}.$}
  \end{equation} 
Suppose $\calB = (B_{-1}, B_0, B_1, \dots, B_k)$ with $\type(\calB)=\bbb=(b_{-1}, b_0, b_1, \dots, b_k) \in \Psi_n$. We consider two cases: (i) $(b_{-1},b_0) \in \N \times \{0^\circ\}$, and (ii) $(b_{-1},b_0) \in \P \times \{0\}$. For convenience, we denote by $\Psi_n^{(i)}$ ($\Psi_n^{(ii)}$, respectively) the set of $\bbb \in \Psi_n$ that in case (i) (and case (ii), respectively). Hence, we have that $\Psi_n$ is the disjoint union of $\Psi_n^{(i)}$ and $\Psi_n^{(ii)}.$

Case (i): Suppose  $\bbb \in \Psi_n^{(i)}$ and thus $(b_{-1},b_0) \in \N \times \{0^\circ\}$. Then $B_0 = \{0\}.$ For each tuple $\mathbf{\tau}= (\tau_1, \tau_2, \dots, \tau_k) \in \fS(B_1) \times \fS(B_2) \times \cdots\times \fS(B_k)$ of permutations on $B_1$, $B_2$, \dots, $B_k$, we define
\[ \calB_{\mathbf{\tau}} = (B_{-1}, B_0, \{ \tau_1(1)\}, \dots, \{ \tau_1(b_1)\}, \{\tau_2(1)\}, \dots, \{ \tau_2(b_2)\},\dots, \{\tau_k(1)\},\dots, \{ \tau_k(b_k)\}),\]
which is a skewed binary partition with type $(b_{-1}, b_0, 1, 1, \dots, 1) \in \Psi_n^{(0)}.$ It is easy to verify that $\ssigma_\calB$ is the union of $\ssigma_{\calB_{\mathbf{\tau}}}$ over all tuples $\mathbf{\tau}= (\tau_1, \tau_2, \dots, \tau_k)$ in $\fS(B_1) \times \fS(B_2)\times \cdots \times \fS(B_k)$. Hence, we proved \eqref{eq:unioncones} for this case.

Case (ii):  Suppose  $\bbb \in \Psi_n^{(ii)}$ and thus $(b_{-1},b_0) \in \P \times \{0\}$. Then $B_0 = \emptyset$ and $0 \in B_{-1}.$ For each subset $S \subseteq B_{-1}\setminus\{0\}$, we define
\[ \calB_{S} = \begin{cases}(S, \{0\}, B_{-1}\setminus \{ \{0\} \cup S\}, B_1, B_2, \dots, B_k), & \text{if $S \neq B_{-1}\setminus\{0\}$} \\
(S, \{0\}, B_1, B_2, \dots, B_k), & \text{if $S = B_{-1}\setminus\{0\}$}.
\end{cases}\] 
One sees that $\calB_S$ is a skewed binary partition whose type is in $\Psi_n^{(i)}.$ Moreover, we can verify that $\ssigma_\calB$ is the union of $\ssigma_{\calB_{S}}$ over all subsets $S$ of $B_{-1}\setminus \{0\}.$ Hence, Given that we have proved \eqref{eq:unioncones} for all $\bbb \in \Psi_n^{(i)}$, \eqref{eq:unioncones} holds for case (ii) as well. 
\end{proof}

\subsection{Characterization of \texorpdfstring{$\bwp(\bbm)$}{SBP(m)}}

Given Corollary \ref{cor: combinatorialType}, it is natural to ask whether we can give a description of skewed binary partitions $\calB$ in $\bwp(\bbm)$. In order to have a better idea of how we can describe them, we present an example below.

\begin{ex}\label{ex:contractions} Consider the MD pair $(\bbm, \bbd)$ where $\bbm=(2, 3, 1, 2)$ and $\bbd = (4,6,8)$.
Let $\calB = (\{0,2,3\}, \emptyset, \{1,6,7\},\{8\}, \{4,5\})$ be the skewed binary partition shown on the top of Figure \ref{fig:contractions}, which is of type $(2,0,3,1,2)$ and corresponds to the sliced preorder cone
\[\ssigma_\calB = \{(c_1, \dots, c_8) \in \R^8  \ |\ 0,c_2,c_3 \leq c_1,c_6,c_7 \leq c_8 \leq c_4, c_5 \}.\]
By Example \ref{ex:vertexcorrespondence} and Theorem \ref{thm: fulldimcones}, we have that $\ssigma_\calB$ is the normal cone at the vertex $\bbv_\calB = (4,0,0,8,8,4,4,6)$. 

The table below show examples of skewed binary partitions in a maximal chain of $\bwp(\bbm)$, each of which corresponds to a contraction of $\preceq_\calB.$ They also correspond to the normal cones in $\Sigma(\pf(\bbm,\bbd))$ of dimensions 8 to 0 (in decreasing order of dimensions). The first three skewed binary partitions are displayed in Figure \ref{fig:contractions}. 

\centering
\begin{tabular}{r|l}
ordered bi-weekly partition  &  type   \\ \hline 
$(\{0,2,3\}, \emptyset, \{1,6,7\},\{8\}, \{4,5\})$ &  $(2, 0, 3, 1, 2)$ \\ \hline
$(\{0,2,3\}, \emptyset, \{6,7\},\{1,8\}^\star, \{4,5\})$ &  $(2, 0, 2, 2^\star,2)$ \\ \hline
$(\{2,3\}, \{0,7\}^\star, \{6\},\{1,8\}^\star, \{4,5\})$ &  $(2, 1^\circ, 1, 2^\star,2)$ \\ \hline
$(\{3\}, \{0,2,7\}^\star, \{6\},\{1,8\}^\star, \{4,5\})$ & $(1, 2^\circ, 1, 2^\star,2)$\\ \hline
$(\{3\}, \{0,2,7\}^\star, \{1,6,8\}^\star, \{4,5\})$ & $(1, 2^\circ, 3^\star,2)$\\ \hline
$(\{3\}, \{0,2,7\}^\star, \{1,5,6,8\}^\star, \{4\})$ & $(1, 2^\circ, 4^\star,1)$\\ \hline
$(\{3\}, \{0,1,2,5,6,7,8\}^\star, \{4\})$ & $(1, 6^\circ, 1)$\\ \hline
$(\emptyset,\{0,1,2,3,5,6,7,8\}^\star, \{4\})$ & $(0, 7^\circ, 1)$\\ \hline
$(\emptyset,\{0,1,2,3,4,5,6,7,8\}^\star)$ & $(0, 8^\circ)$
\end{tabular}
\end{ex}

Due to the symmetry of parking function polytope, one sees that if $\calB$ is in $\bwp(\bbm),$ so is every skewed binary partition of the same type as $\calB$. In the following proposition, we characterize the types of the ordered skewed binary partitions corresponding to the normal cones in $\Sigma(\pf(\bbm,\bbd))$.

\begin{prop}\label{prop: type-characterization}
    Suppose that $\bbm = (m_0, \dots, m_\ell)$ is a multiplicity vector of magnitude $n$, and $\calB$ is a skewed binary partition of $[0,n]$. Then $\calB$ is in $\bwp(\bbm)$ if and only if $\type(\calB) = (b_{-1}, b_0, \dots, b_p)$ is a skewed binary composition satisfying the following conditions.
    \begin{enumerate}[label=(\arabic*)]
        \item\label{itm: firsttwoblocks} $0 < |b_{-1}|+|b_0| \leq m_0$ if and only if $b_0 = 0$.
        \item\label{itm: positiveblocks} $m_0 < |b_{-1}| + |b_0| + |b_1|$ \footnote{If $\calB = (B_{-1}, B_0)$, then this inequality becomes $m_0 < |b_{-1}| + |b_0|$.} and for every positive integer $i \leq \ell$, there exists at most one positive integer $j$ such that
        \begin{align}
            \label{eq:partialsumcondition} 
            m_0 + \cdots + m_{i-1} \leq |b_{-1}| + \cdots + |b_{j-1}| < |b_{-1}| + \cdots + |b_{j}| \leq  m_0 + \cdots + m_{i}.
        \end{align}
        \item\label{itm: blocktypes} If $j$ is a positive integer such that there exists a positive integer $i$ satisfying \eqref{eq:partialsumcondition}, then $b_j \in \P.$ Otherwise, $b_j \in \P^\star$ for $1 \leq j \leq p.$
    \end{enumerate}
\end{prop}

Before proving Proposition \ref{prop: type-characterization}, we first establish a few auxiliary results. 

\begin{lem}\label{lem: m0-first-three-blocks}
    Let $\bbm = (m_0, \dots, m_\ell)$ be a multiplicity vector, and $\calA = (A_{-1}, A_0, \dots, A_q)$ a skewed binary partition such that $\type(\calA) = (a_{-1}, a_0, \dots, a_q) \in \Omega_\bbm.$ If $\calB$ is a contraction of $\calA$, then $\type(\calB) = (b_{-1}, b_0, \dots, b_p)$ satisfies $m_0 < |b_{-1}| + |b_0| + |b_1|$.

    Moreover, if $p \ge 1$ is a positive integer, then for any $j \in [p],$ the right vertex $B_{j}$ of the graph $G(\hat{\calA},\hat{\calB})$ must be adjacent to a left vertex $A_s$ for some positive integer $s \in [q]$.
\end{lem}

\begin{proof}
    We first consider the case where $\calB = (B_{-1},B_0).$ Then $B_{-1}\cup B_0 = [0,n]$ and so $n = |b_{-1}| + |b_0|.$ By Remark \ref{rem:mzero1}, we have $m_0 < n = |b_{-1}| + |b_0|.$  

    Suppose $\calB = (B_{-1}, B_0, \dots, B_p)$ for some $p \geq 1.$ Note that one of $B_{-1}$ and $B_0$ can be empty, and thus may not be a vertex of $G(\hat{\calA},\hat{\calB})$. We claim that
    \begin{align}\label{eq: A-1-proper-subset}
	      A_{-1} \cup A_{0}  \text{ is a proper subset of } B_{-1}\cup B_0 \cup B_1,
    \end{align}
    which implies that 
    \[ |a_{-1}| + |a_0| < |b_{-1}| + |b_0| + |b_1|\le |b_{-1}| + |b_0| + |b_1| + \cdots + |b_j|.\] 
    One sees that the first conclusion of the lemma follows from the above inequality and Remark \ref{rem:mzero2}, and the second conclusion of the lemma follows from the above inequality and the non-crossing property of $G(\hat{\calA},\hat{\calB})$. Therefore, it is left to prove the claim. 
    By the non-crossing property of $G(\hat{\calA},\hat{\calB})$, it suffices to show that the vertex $B_1$ is adjacent to $A_t$ for some positive integer $t \in [p]$. Since $0 \not\in B_1$ and $A_{0}$ is either $\{0\}$ or $\emptyset$, the vertex $B_1$ of $G(\hat{\calA},\hat{\calB})$ is not adjacent to $A_0$.  Assume for the sake of contradiction that the vertex $B_{1}$ of $G(\hat{\calA},\hat{\calB})$ is only adjacent to the non-homogeneous vertex $A_{-1}$. Then, the non-crossing property of $G(\hat{\calA},\hat{\calB})$ implies that $B_{-1}$ (if nonempty) and $B_{0}$ (if nonempty) are only adjacent to $A_{-1}$ as well. Thus, by Theorem $\ref{thm: graph-characterization}/\ref{itm: right-homog1}$, the vertices $B_{-1}$ (if nonempty), $B_0$ (if nonempty) and $B_1$ are non-homogeneous. However, since at least one of $B_{-1}$ and $B_0$ is nonempty, the left non-homogeneous vertex $A_{-1}$ is adjacent to at least two non-homogeneous vertices in $\{B_{-1}, B_0, B_1\}$, contradicting Theorem $\ref{thm: graph-characterization}/\ref{itm: left-nonhomog1}$. This completes the proof.
\end{proof}

\begin{lem}\label{lem: degree-two}
	Let $\bbm = (m_0, \dots, m_\ell)$ be a multiplicity vector, and $\calA$ a skewed binary partition such that $\type(\calA) \in \Omega_\bbm.$ Suppose that $\calB = (B_{-1}, B_0, \dots, B_p)$ is a skewed binary partition that is a contraction of $\calA$, for some positive integer $p$. Let $j \in [p]$ be a positive integer. Then the right vertex $B_j$ of $G(\hat{\calA},\hat{\calB})$ has degree at least two if and only if there exists a nonnegative integer $k < \ell$ such that $\type(\calB) = (b_{-1}, \dots, b_p)$ satisfies 
    \begin{equation}\label{eq:bbounds}
        |b_{-1}| + |b_0| + \cdots + |b_{j-1}| < m_0 + \cdots +m_k  < |b_{-1}| + |b_0| + \cdots + |b_{j}|.
    \end{equation}
\end{lem}

\begin{proof} Let $j \in [p]$ be a positive integer. Let us write $\calA = (A_{-1}, A_0, \dots, A_q)$ and $\type(\calA) = (a_{-1}, a_0, \dots, a_q).$ 
Since $\type(\calA) \in \Omega_\bbm,$ we may write 
    \begin{equation}\label{eq: typeA}
    (a_{-1}, a_0, \dots, a_q) = (a_{-1}, a_0, m_0 + \cdots + m_{g} - a_{-1}, m_{g+1}, \dots, m_\ell),
    \end{equation}
    where $a_0 \in \{0, 0^\circ\}$, $g = \ell-q+1$, and 
    \begin{equation}\label{eq:a-1bound} 
    |a_{-1}| \ge m_0 + m_1 + \cdots +m_{g-1}. 
    \end{equation}
    Clearly, $A_{0}$ is either $\{0\}$ or $\emptyset$. 

    $(\Longrightarrow)$ Suppose that the right vertex $B_j$ of $G(\hat{\calA},\hat{\calB})$ has degree at least two. By Lemma \ref{lem: m0-first-three-blocks}, we have that the right vertex $B_{j}$ of $G(\hat{\calA},\hat{\calB})$ must be adjacent to a vertex $A_s$ for some positive integer $s \in [q]$. Since $B_j$ has degree at least two, we may let $A_t$ where $t \neq s$ be another left vertex of $G(\hat{\calA},\hat{\calB})$ adjacent to $B_j$. We first consider $t \geq 1$, in which case we may assume without loss of generality that $1 \leq s < t.$ 
    Note that, for $f \in [0,q]$ and $h \in [0,p]$, one has
    \begin{align}\label{eq: partial-sums}
        \sum_{r = -1}^f|A_r|= 1 + \sum_{r = -1}^{f}|a_r| \ \text{ and } \sum^{h}_{r = -1}|B_r| = 1+\sum^{h}_{r = -1}|b_r|.
    \end{align}
    Since both $G(\hat{\calA},\hat{\calB})$ is non-crossing, and the intersections $A_s \cap B_j$ and $A_t\cap B_j$ are nonempty, we deduce from \eqref{eq: partial-sums} and Lemma \ref{lem: block-intersection}/\eqref{itm: blocks-intersection} that
    \begin{align*}
        \sum_{r = -1}^{j-1}|b_r| < \sum^{s}_{r = -1}|a_r| \leq \sum^{t-1}_{r = -1}|a_r| < \sum_{r = -1}^{j}|b_r|.
    \end{align*}
 It follows from the expression \eqref{eq: typeA} that
 $\displaystyle \sum^{s}_{r = -1}|a_r| = \sum^{g+s-1}_{r= 0}m_r$. Hence, we obtain \eqref{eq:bbounds} by letting $k=g+s-1$.

   Now consider $t \leq 0.$ Since $0 \not\in B_j$ and $A_0$ is either $\{0\}$ or $\emptyset$, it follows that $t = -1.$ By the non-crossing property of $G(\hat{\calA},\hat{\calB})$, the vertices $B_{-1}$, $B_0, \dots, B_{j-1}$ have degree one and are adjacent to $A_{-1}$, and the block $A_{0} = \emptyset.$ Thus, $B_{-1}\cup B_0 \cup \cdots \cup B_{j-1}$ is a proper subset of $A_{-1}$. As $\type(\calA) \in \Omega_\bbm,$ we must have $m_0 > 0$ and $\type(\calA) = (m_0, 0, m_1, \dots, m_\ell).$ 
   Hence,
    \[|b_{-1}| + |b_0| + \cdots + |b_{j-1}| < |a_{-1}| = m_0 < |b_{-1}| + |b_0| + \cdots + |b_{j}|.\]

    $(\Longleftarrow)$ Conversely, suppose that there exists a nonnegative integer $k \leq \ell$ such that $\type(\calB) = (b_{-1}, \dots, b_p)$ satisfies \eqref{eq:bbounds}.
   
    Suppose $g \le k.$ Then it follows from the expression \eqref{eq: typeA} that
 $\displaystyle \sum^{k-g+1}_{r = -1}|a_r| = \sum^{k}_{r= 0}m_r$.
 Using this and \eqref{eq: partial-sums}, we rewrite \eqref{eq:bbounds} as
 \[ \sum_{r=-1}^{j-1} |B_r| < \sum^{k-g+1}_{r = -1}|A_r| < \sum_{r=-1}^{j} |B_r|.  \]
 Applying the non-crossing property of $G(\hat{\calA},\hat{\calB})$ and Lemma \ref{lem: block-intersection}/\eqref{itm: some-blocks-intersection}, we obtain that $B_j$ is adjacent to $A_{k-g + 1}$ and $A_{i}$ for some $i > k-g + 1.$ Thus, in this case, the vertex $B_j$ has degree at least two. 

   On the other hand, suppose $k < g.$ Combining the left inequality in \eqref{eq:bbounds}, \eqref{eq:a-1bound} and \eqref{eq: partial-sums}, we get that
   \[ |A_{-1}| = |a_{-1}| + 1  \ge m_0 + m_1 + \cdots + m_{g-1} +1 \ge 1+ \sum_{r=0}^{k} m_r > \sum_{r=1}^{j-1} |B_r|. \]
   Then the non-crossing property of $G(\hat{\calA},\hat{\calB})$ implies that $B_j$ is adjacent to $A_{-1}.$ Moreover, it follows from Lemma \ref{lem: m0-first-three-blocks} that the right vertex $B_{j}$ of $G(\hat{\calA},\hat{\calB})$ must be adjacent to a vertex $A_s$ for some positive integer $s \in [q]$. Hence, the vertex $B_j$ also has degree at least two in this case. This completes the proof.
\end{proof}

\begin{proof}[Proof of Proposition \ref{prop: type-characterization}] Let $\calB$ be a skewed binary partition of $[0,n]$ and suppose that $\type(\calB) = (b_{-1}, b_0, \dots, b_p)$.
    
    $(\Longrightarrow)$ Suppose that $\calB \in \bwp(\bbm)$. 
    Then $\calB$ is a contraction of $\calA = (A_{-1}, \dots, A_q)$ for some $\calA$ such that $\type(\calA) = (a_{-1}, a_0, \dots, a_q) \in \Omega_\bbm.$ Note that $A_{0}$ is either $\{0\}$ or $\emptyset$. 

    We first show that $(b_{-1}, \dots, b_p)$ satisfies condition \ref{itm: firsttwoblocks}.
    Recall from the definition of skewed binary compositions that we always have $(b_{-1}, b_0) \in  \ (\N \times \N^\circ) \cup (\P \times \{0\})$. 
    Suppose $b_0 \neq 0.$ We will show that \[ |b_{-1}| + |b_{0}| >m_0 \text{ or } |b_{-1}| + |b_{0}| =0.\] 
    Clearly, the above is true if $m_0 = 0.$ Hence, we may assume $m_0 > 0.$ It then follows from Definition \ref{def: vertex-type} that $0 \in A_{-1}$ or $|a_{-1}| > m_0,$ where either case implies that $|A_{-1}| \ge m_0+1$. 
    Since $b_0 \neq 0,$ we have $b_0 \in \N^\circ.$ Hence, $b_0 \in \P^\circ$ or $b_0 = 0^\circ$.  If $b_0 \in \P^\circ$, then
     $B_0$ is a (nonempty) homogeneous block.
     By Theorem \ref{thm: graph-characterization}/\ref{itm: right-homog1}, the right homogeneous vertex $B_0$ of $G(\hat{\calA},\hat{\calB})$ must be adjacent to a vertex other than $A_{-1}.$ The non-crossing property of $G(\hat{\calA},\hat{\calB})$ then implies that $A_{-1}$ can only be adjacent to $B_{-1}$ or $B_{0}$ (or both). Thus, $A_{-1}$ is a proper subset of $B_{-1}\cup B_0.$ Therefore, 
    \[ |b_{-1}| + |b_0|= |B_{-1}| + |B_0| -1 > |A_{-1}|-1 \ge m_0  \]
    as desired. 
    Now assume $b_0 =0^\circ$. If $b_{-1}=0$, then we have $|b_{-1}|+|b_0|=0.$ Hence, we assume $b_{-1}>0$. Thus, we have that both $B_0$ and $B_{-1}$ are non-homogeneous vertices of $G(\hat{\calA},\hat{\calB})$. Since, by Theorem \ref{thm: graph-characterization}/\ref{itm: left-nonhomog1}, the left non-homogeneous vertex $A_{-1}$ of $G(\hat{\calA},\hat{\calB})$ cannot be adjacent to both non-homogeneous vertices $B_{-1}$ and $B_0,$ it follows from the non-crossing property of $G(\hat{\calA},\hat{\calB})$ that $A_{-1}$ is not adjacent to $B_0=\{0\}$. Hence,  $0 \not\in A_{-1}$ and $0 \in A_0.$ Therefore, Because $\type(\calA) \in \Omega_\bbm$, it follows that $|a_{-1}| > m_0.$ Moreover, in $G(\hat{\calA},\hat{\calB})$, the vertices $A_0$ and $B_0$ are adjacent. By the non-crossing property of $G(\hat{\calA},\hat{\calB})$, the vertex $A_{-1}$ can only be adjacent to $B_{-1}$ or $B_{0}$ (or both). Thus, $A_{-1} \subseteq B_{-1}\cup B_0$. This implies that $|b_{-1}| + |b_0| \geq |a_{-1}| > m_0,$ as desired.

    Conversely, suppose that $b_0 = 0$. Then $b_{-1} \in \P$. Thus, $B_0 = \emptyset$, and $B_{-1}$ contains $0$ and at least one positive integer. By Theorem \ref{thm: graph-characterization}/\ref{itm: right-nonhomog1},  the right non-homogeneous $B_{-1}$ of $G(\hat{\calA},\hat{\calB})$ is only connected to one non-homogeneous left vertex. This implies that $B_{-1}$ is a subset of a block $A_j$ for some $j \in [-1,q].$ Therefore, we must have that $A_0 \neq \{0\}$, since $A_0 = \{0\}$ would intersect but would not contain $B_{-1}$. Because $\type(\calA) \in \Omega_\bbm$, it follows that $m_0 > 0$, $A_0 = \emptyset,$ and $|a_{-1}| = m_0$. By the non-crossing property of $G(\hat{\calA},\hat{\calB})$ and Theorem \ref{thm: graph-characterization}/\ref{itm: right-nonhomog1}, the vertex $B_{-1}$ is only adjacent to $A_{-1}.$ Hence, we have $B_{-1} \subseteq A_{-1},$ which then implies $0 < |b_{-1}|+ |b_0|  = |b_{-1}| \leq |a_{-1}| = m_0$ as desired. 
    
    Next, we show that $\type(\calB)$ satisfies condition \ref{itm: positiveblocks}. By the first part of Lemma \ref{lem: m0-first-three-blocks}, we have $m_0 < |b_{-1}| + |b_0| + |b_1|.$ Assume by way of contradiction that there exists a positive integer $i \leq \ell$ together with two consecutive positive integers $j, j+1$ satisfying inequality \eqref{eq:partialsumcondition}. Let $k_t = |b_{-1}| + \cdots + |b_{j+t-1}|$ for $t = 0, 1, 2$. The assumption implies that 
    \begin{equation}\label{eq: in-a-box}
        m_0 + \cdots + m_{i-1} \leq k_0 < k_1 < k_2 \leq m_0 + \cdots + m_i.
    \end{equation}
    Then, by Lemma \ref{lem: degree-two}, both vertices $B_j$ and $B_{j+1}$ of $G(\hat{\calA},\hat{\calB})$ must have degree one. Because every left vertex of $G(\hat{\calA},\hat{\calB})$ is non-homogeneous, it follows from Theorem \ref{thm: graph-characterization}/\ref{itm: right-homog1} that both $B_j$ and $B_{j+1}$ are non-homogeneous. Theorem \ref{thm: graph-characterization}/\ref{itm: left-nonhomog1}\ref{itm: right-nonhomog1} then imply that the vertex $B_{j}$ is adjacent to $A_s$ for a unique $s$ and $B_{j+1}$ is adjacent to $A_t$ for a unique $t$, where $s \neq t$.
    Thus, $B_j \subseteq A_s$ and $B_{j+1} \subseteq A_t$. Applying the second part of Lemma \ref{lem: m0-first-three-blocks} and the non-crossing property of $G(\hat{\calA},\hat{\calB})$, we obtain that $1 \le s < t$. Then, it follows from Lemma \ref{lem: block-intersection}/\eqref{itm: blocks-inclusion}\eqref{itm: blocks-revinclusion} and \eqref{eq: partial-sums} that
    \[k_1 = \sum_{r = -1}^{j}|b_r| \le \sum^{s}_{r = -1}|a_r| \leq \sum^{t-1}_{r = -1}|a_r| \le \sum_{r = -1}^{j}|b_r|  = k_1.\]
    One sees that all the equalities in the above inequality hold. Hence, $\displaystyle k_1 = \sum^{s}_{r = -1}|a_r|$. However, it follows from the expression \eqref{eq: typeA} that
    $\displaystyle k_1= \sum^{s}_{r = -1}|a_r| = \sum^{g+s-1}_{r= 0}m_r$,  contradicting the assumption \eqref{eq: in-a-box}. This shows that $\type(\calB)$ satisfies condition \ref{itm: positiveblocks}.

    We now show that $\type(\calB)$ satisfies condition \ref{itm: blocktypes}. Let $j \in [p]$ be a positive integer. We define $k_t = |b_{-1}| + \cdots + |b_{j+t-1}|$ for $t = 0, 1$ as before. If there exists a positive integer $i$ satisfying inequality \eqref{eq:partialsumcondition}, i.e., 
    \begin{equation}\label{eq:partialsumcondition0}
    m_0 + \cdots + m_{i-1} \leq k_0 < k_1 \leq m_0 \cdots + m_i. 
    \end{equation}
then similar to the proof in the last part, we can show that $B_j$ is non-homogeneous. Therefore, $b_j \in \P.$ Now suppose that there is no positive integers $i$ satisfying inequality \eqref{eq:partialsumcondition} or \eqref{eq:partialsumcondition0}. By Lemma \ref{lem: m0-first-three-blocks}, we have $m_0 < |b_{-1}| + |b_0| + |b_1| \leq k_1.$ We also clearly have that $k_1 \le n = m_0 + m_1 + \cdots + m_\ell.$ Let $i$ be the unique integer in $[\ell]$ such that $m_0 + \cdots + m_{i-1} < k_1 \leq m_0 \cdots + m_i.$ Then we must have that 
    \[k_0 < m_0 + \cdots + m_{i-1} < k_1.\]
    It follows from Lemma \ref{lem: degree-two} that the vertex $B_j$ of $G(\hat{\calA},\hat{\calB})$ has degree at least two. Thus, by Theorem \ref{thm: graph-characterization}/\ref{itm: right-nonhomog1}, the block $B_j$ is homogeneous. Hence, $b_j \in \P^\star.$ 

    $(\Longleftarrow)$ Conversely, suppose that $\calB=(B_{-1}, B_0, B_1, \dots, B_p)$ is a skewed binary partition such that $\type(\calB) = (b_{-1}, \dots, b_p)$ satisfies conditions \ref{itm: firsttwoblocks} - \ref{itm: blocktypes}. By symmetry, we may assume without loss of generality that $\calB$ is standard. To see that $\calB \in \bwp(\bbm)$, we will construct $\calA$ such that $\calB \leq \calA$ and $\type(\calA) \in \Omega_\bbm.$ We consider two cases: Case 1: $m_0 < |b_{-1}| +|b_0| \le n$, and Case 2: $0\le |b_{-1}|+|b_0| \le m_0.$

    \noindent \textbf{Case 1:} Suppose  $m_0 < |b_{-1}| +|b_0| \le n$. Let $g$ be the unique positive integer in which 
\begin{equation}\label{eq:gbound}
	m_0+m_1 + \cdots + m_{g-1} \le |b_{-1}| + |b_0| < m_0 + m_1 + \cdots + m_g 
\end{equation}
and define $\calA=(A_{-1}, A_0, A_1, \dots, A_q)$ to be the standard skewed binary partition such that
\[\type(\calA) = (a_{-1}, \dots, a_q) = (|b_{-1}|+|b_0|, 0^\circ, m_0+m_1 + \cdots + m_g - |b_{-1}|-|b_0|, m_{g+1}, \dots, m_\ell),\] where $q = \ell - g+1.$ One sees that $\type(\calA)=(a_1, \dots, a_q)$ satisfies the following conditions:
\begin{align}
	\sum^{i}_{r = 0}m_r = \sum^{i-g+1}_{r = -1}|a_r| , &\quad \text{ if $i \in [g, \ell]$} \label{eq:m2a} \\
	\sum^{i}_{r = 0}m_r \le |b_{-1}| + |b_0| = |a_{-1}|+|a_0| = \sum^{i-g+1}_{r = -1}|a_r| , &\quad \text{if $i=g-1$}.
	\label{eq:m2a-g} 
\end{align}
    To see that $\calB \leq \calA,$ we will show that $G(\hat{\calA},\hat{\calB})$ satisfies all of the conditions in Theorem \ref{thm: graph-characterization}.
    First, by condition \ref{itm: firsttwoblocks}, we have $b_0 \neq 0.$ Hence, $(b_{-1}, b_0) \in \N \times \N^\circ.$ Thus, $B_0$ contains $0$ and is nonempty, but $B_{-1}$ could be empty. Because $a_0 = 0^0,$ we have $A_0 = \{0\}$. By Lemma \ref{lem: standard-crossing}, we have that $G(\hat{\calA},\hat{\calB})$ is non-crossing if and only if the edge $\{A_0, B_0\}$ connecting $A_0$ and $B_0$ does not intersect with any other edges in $G(\hat{\calA}, \hat{\calB})$. However, since $|a_{-1}| + |a_0| = |b_{-1}| + |b_0|$, we have that $A_{-1} \cup A_0 = B_{-1} \cup B_0$, which implies the desired condition on the edge $\{A_0, B_0\}$. Thus, condition \ref{itm: non-crossing} of Theorem \ref{thm: graph-characterization} is satisfied.

    We now verify conditions \ref{itm: right-nonhomog1} and \ref{itm: right-homog1} of Theorem \ref{thm: graph-characterization} for right vertices of $G(\hat{\calA},\hat{\calB})$.
    \begin{itemize}
	    \item If $B_{-1}$ is nonempty, then it is a non-homogeneous vertex. Since $A_0=\{0\}$ clearly has no intersection with $B_{-1}$, we have that $B_{-1}$ is only adjacent to $A_{-1}$. Therefore, $\deg^\star(B_{-1})=0$ and $\deg^\vee(B_{-1})=1.$

	    \item Since $A_0 =\{0\} \subseteq B_0,$ we have that $B_{0}$ is adjacent to the non-homogeneous vertex $A_{0}$. The only other right vertex $B_0$ might be adjacent to is $A_{-1}$. Clearly, $A_{-1}$ is adjacent to $B_0$ if and only if $B_0$ contains an element other than $0$, which is equivalent to that $B_0$ is homogeneous. One sees that  conditions \ref{itm: right-nonhomog1} and \ref{itm: right-homog1} of Theorem \ref{thm: graph-characterization} are satisfied for $B_0$.

	    \item Suppose that $j \in [p]$ and $B_j$ is non-homogeneous. Then, it follows from \eqref{eq:gbound}, \eqref{eq:m2a}, \eqref{eq:m2a-g} and conditions \ref{itm: positiveblocks}--\ref{itm: blocktypes} that there exists a positive integer $i$ in $[g,\ell]$ satisfying
        \begin{align}\label{eq: b-j}
		\sum^{i-1}_{r = 0}m_r \le \sum^{i-g}_{r = -1}|a_r|  \leq |b_{-1}| + \cdots + |b_{j-1}| < |b_{-1}| + \cdots + |b_{j}| \leq  \sum^{i}_{r = 0}m_r = \sum^{i-g+1}_{r = -1}|a_r|.
        \end{align}
	Thus, by Lemma \ref{lem: block-intersection}/\eqref{itm: blocks-revinclusion}, we have that $B_j \subseteq A_{i-g+1}$. Hence, the right non-homogeneous vertex $B_j$ of $G(\hat{\calA},\hat{\calB})$ satisfies $\deg^\vee(B_j) = 1$ and $\deg(B_j) = 1$, condition \ref{itm: right-nonhomog1} of Theorem \ref{thm: graph-characterization}. 

\item
	Now suppose that $j \in [p]$ and $B_j$ is homogeneous. It follows from conditions \ref{itm: positiveblocks}--\ref{itm: blocktypes} that there exists a positive integer $i \in [\ell]$ such that
	 \[ |b_{-1}| + \cdots + |b_{j-1}| <  \sum^{i-1}_{r = 0}m_r < |b_{-1}| + \cdots + |b_{j}| \leq \sum^{i}_{r = 0}m_r.\]
	 Next, by \eqref{eq:gbound}, one sees that $i \ge g+1$ or $i -1 \ge g.$ Therefore, using \eqref{eq:m2a}, we obtain
\begin{align}
            |b_{-1}| + \cdots + |b_{j-1}| <  \sum^{i-1}_{r = 0}m_r = \sum_{r = -1}^{i-g}|a_r| < |b_{-1}| + \cdots + |b_{j}|.
        \end{align}
It follows from Lemma \ref{lem: block-intersection}/\eqref{itm: blocks-intersection} that the right homogeneous vertex $B_j$ of $G(\hat{\calA},\hat{\calB})$ are adjacent to the vertices $A_{i-g}$ and $A_{i-g+1}$. Thus, we have $\deg(B_j) \geq 2.$ Hence, the vertex $B_j$ of $G(\hat{\calA},\hat{\calB})$ satisfies condition \ref{itm: right-homog1} of Theorem \ref{thm: graph-characterization}.  
    \end{itemize}
    Therefore, conditions \ref{itm: right-nonhomog1} and \ref{itm: right-homog1} of Theorem \ref{thm: graph-characterization} are proved.

    Because every block of $\calA$ is non-homogeneous, $G(\hat{\calA},\hat{\calB})$ automatically satisfies condition \ref{itm: left-homog1} of Theorem \ref{thm: graph-characterization}. We now verify condition \ref{itm: left-nonhomog1} of Theorem \ref{thm: graph-characterization}. 
    \begin{itemize}
	    \item We have established that $A_{-1}$ is adjacent to $B_0$ if and only if $B_0$ is homogeneous. Hence, the only non-homogeneous vertex $A_{-1}$ can be adjacent to is $B_{-1}.$ Thus, $\deg^\vee(A_{-1}) \le 1.$ 
        \item Since $A_0 = \{0\} \subseteq B_0$, we have $\deg^\vee(A_{0}) \le \deg(A_{0})=1$. 
	   \item Let $s \in [q]$. Suppose that  the left non-homogeneous vertex $A_s$ is adjacent to a non-homogeneous vertex $B_j$ for some $-1 \le j \le p$. Because $A_{-1} \cup A_0 = B_{-1} \cup B_0$, one has $j \geq 1.$ Since we have already shown that $G(\hat{\calA},\hat{\calB})$ satisfies condition \ref{itm: right-nonhomog1} of Theorem \ref{thm: graph-characterization}, it follows that $B_j \subseteq A_s$. Thus, by Lemma \ref{lem: block-intersection}/\eqref{itm: blocks-revinclusion}, we have that $j$ satisfies \eqref{eq: b-j} with $s = i-g+1.$ Moreover, by conditions \ref{itm: positiveblocks}--\ref{itm: blocktypes}, such a $j$ is the unique positive integer satisfying \eqref{eq: b-j}. That is, $B_j$ is the only non-homogeneous vertex adjacent to $A_s$. This shows $\deg^\vee(A_s) \le 1.$
    \end{itemize}    
    Therefore, condition \ref{itm: left-nonhomog1} of Theorem \ref{thm: graph-characterization} is satisfied. This completes our proof for case 1.

    \noindent \textbf{Case 2:} Suppose $0\le |b_{-1}|+|b_0| \le m_0$. Define $\calA=(A_{-1}, A_0, A_1, \dots, A_q)$ to be the standard skewed binary partition such that
    \[\type(\calA) = (a_{-1}, a_0, a_1, \dots, a_q) =
    \begin{cases}
        (m_0, 0, m_1, \dots, m_\ell) & \text{ if } m_0 \neq 0\\
        (m_0, 0^\circ, m_1, \dots, m_\ell) & \text{ if } m_0 = 0
    \end{cases}.\]
    For convenience, we define 
    \[ \hat{b}_0 := |b_{-1}|+|b_0|, \text{ and } \hat{b}_i := b_i \text{ for all $1 \le i \le p$.} \]
    If $\hat{b}_0=|b_{-1}|+|b_0|=0,$ one sees that $(b_{-1},b_0)=(0, 0^\circ)$, and $0 < \hat{b}_0 \le m_0$, by condition \ref{itm: firsttwoblocks}, we have $b_0 = 0$ and thus $(b_{-1},b_0) \in \P \times \{0\}.$ In either case, we have that one of $B_{-1}$ and $B_0$ is the empty set and the other one is the non-homogeneous block $[0, \hat{b}_0]$. Therefore, $\hat{\calB}$ is the standard binary partition $(\hat{B}_0, \hat{B}_1, \dots, \hat{B}_p)$, in which 
	    \[\hat{B}_0 = [0, \hat{b}_0] \text{ and } \hat{B}_i = B_i \text{ for $1 \le i \le p$}.\]
	    Similarly, we see that $\hat{\calA}$ is the standard binary partition $(\hat{A}_0, \hat{A}_1, \dots, \hat{A}_q)$, in which 
	    \[\hat{A}_0 = [0, m_0] \text{ and } \hat{A}_i = A_i \text{ for $1 \le i \le q$}.\]
	    With these characterization of $\hat{A}$ and $\hat{B}$ we will be able to show that $G(\hat{\calA},\hat{\calB})$ satisfies all of the conditions in Theorem \ref{thm: graph-characterization}. We will omit the details of the proof for this case, as the arguments are mostly similar to what we did for Case 1.  
\end{proof}

\subsection{Consequences} In this part, we will discuss a few corollaries to the results presented in the last two subsections.

\subsubsection{Inequality Description} For any fan $\Sigma$, we denote by $\Sigma^1$ the collection of one-dimensional cones in $\Sigma$.
In order to obtain an inequality description for parking function polytopes, we first give a characterization of the set $\Sigma^1(\pf(\bbm,\bbd))$ of one-dimensional cones in the normal fan $\Sigma(\pf(\bbm,\bbd))$. For any $I \subseteq [n]$, we let $\bbe_I = \sum_{i \in I} \bbe_i.$ Recall $\cone(\bbr)$ denotes the one-dimensional cone generated by the vector $\bbr$.

\begin{cor}\label{cor: one-dim-comes}
    Let $n \geq 2$ be a positive integer. Suppose that $(\bbm,\bbd)$ is an MD pair where $\bbm = (m_0, \dots, m_\ell)$ is a multiplicity vector of magnitude $n.$ 
    Then $\Sigma^1(\pf(\bbm,\bbd))$ is contained in the following set of one dimensional cones
    \begin{equation}
	    \label{eq:sup1dimcones}
	    \{ - \cone(\bbe_i) \mid i \in [n]\} \cup \{ \cone(\bbe_I) \mid \emptyset \neq I \subseteq [n]\}
    \end{equation}
    Moreover, we have the following precise description for $\Sigma^1(\pf(\bbm,\bbd))$.
    \begin{enumerate}
	    \item If $\ell= 1$ and $m_1 = n,$ then 
		    \[ \Sigma^1(\pf(\bbm,\bbd)) =  \{ - \cone(\bbe_i) \mid i \in [n]\} \cup \{ \cone(\bbe_I) \mid I \subseteq [n], |I| =1 \}.   \]
	    \item If $\ell = 1$ and $m_1 = 1$, then 
		    \[ \Sigma^1(\pf(\bbm,\bbd)) =  \{ - \cone(\bbe_i) \mid i \in [n]\} \cup \{ \cone(\bbe_I) \mid I \subseteq [n], |I| =n \}.   \]
        \item If $\ell=1$ and $2 \leq m_1 \leq n-1$,  then 
		\[ \Sigma^1(\pf(\bbm,\bbd)) =  \{ - \cone(\bbe_i) \mid i \in [n]\} \cup \{ \cone(\bbe_I) \mid I \subseteq [n], |I| =1 \text{ or } n \}.   \]
		\item If $\ell\geq 2$, then 
			\begin{align*}
			\Sigma^1(\pf(\bbm,\bbd)) =&  \quad  \{ - \cone(\bbe_i) \mid i \in [n]\} \cup \{ \cone(\bbe_I) \mid I \subseteq [n], |I| =1 \text{ or } n \} \\ 
		& \cup \{ \cone(\bbe_I) \mid I \subseteq [n], m_\ell +1 \le |I| \le n-m_0-1 \}. 
	\end{align*}
	\end{enumerate}
    \end{cor}

\begin{proof}
    By Proposition \ref{prop: dim-preordcone}, we have that the one dimensional cones in $\Sigma(\pf(\bbm,\bbd))$ correspond to the skewed binary partitions $\calB \in \bwp(\bbm)$ such that $\type(\calB) = (b_{-1}, b_0, \dots, b_k)$ satisfies
 \begin{equation*}
        1 = \left( \sum_{b_i \in \P} b_i \right) +  \#(b_i \in \P_{\ge 2}^\star).
    \end{equation*}
This implies that $\type(\calB)$ lies in the set
\begin{align}\label{eq: possible-type-one-dim-cones}
    \{(1, (n-1)^\circ), (0, (n-1)^\circ, 1)\} 
     \cup \{ (0, (n-i)^\circ, i^\star) \mid 2 \le i \le n\}.
\end{align}
By checking when each element in the above set satisfies the three conditions given in Proposition \ref{prop: type-characterization}, we obtain that a cone $\sigma$ belongs to $\Sigma^1(\pf(\bbm,\bbd))$ if and only if $\sigma = \ssigma_\calB$ for some skewed binary partition $\calB$ of $[0,n]$ such that $\type(\calB)$ satisfies the following conditions: 
\begin{enumerate}
       \item If $\ell= 1$ and $m_1 = n,$ then 
        $\type(\calB) \in \{(1,(n-1)^\circ), (0,(n-1)^\circ,1)\}.$
        \item If $\ell = 1$ and $m_1 = 1$ then 
        $\type(\calB) \in \{(1, (n-1)^\circ), (0,0^\circ, n^\star)\}.$
        \item If $\ell=1$ and $2 \leq m_1 \leq n-1$,  then $\type(\calB) \in \{(1,(n-1)^\circ), (0,(n-1)^\circ,1), (0,0^\circ, n^\star)\}.$
        \item If $\ell\geq 2$, then $\type(\calB) \in X\cup Y$ where
        \begin{align*}
            X &:= \{(1, (n-1)^\circ), (0,(n-1)^\circ, 1), (0,0^\circ, n^\star)\}\\
	    Y&:= \{ (0, (n-i)^\circ, i^\star) \mid m_\ell +1 \le i \le n-m_0-1\}. 
        \end{align*}
\end{enumerate}
Finally, we check that 
\begin{itemize}
	\item $\type(\calB) = (1, (n-1)^\circ)$ if and only if $\ssigma_\calB = -\cone(\bbe_i)$ for some $i \in [n]$;
	\item $\type(\calB) = (0,(n-1)^\circ,1)$ if and only if $\ssigma_\calB = \cone(\bbe_i)$ for some $i \in [n]$;
	\item and for each $2 \le i \le n$, $\type(\calB) =  (0, (n-i)^\circ, i^\star)$ if and only if $\ssigma_\calB = \cone(\bbe_I)$ for some $i$-subset $I$ of $[n]$.
\end{itemize}
The conclusions of this corollary follow.  
\end{proof}

Recall the definition of $\bbu$-extreme points in Definition \ref{defn:extreme}. It follows from Theorem \ref{thm: fulldimcones} and \eqref{eq:extreme=vB} that $\calX(\bbu)$ is the set of vertices of $\pf(\bbu)$.  
Applying Lemma \ref{lem:det-ineq} with this fact and Corollary \ref{cor: one-dim-comes}, we are able to give the following inequality description for parking function polytopes.

\begin{prop}\label{prop: ineqDescription}
	Suppose that $\bbu = (u_1, \dots, u_n)$ is a non-decreasing vector in $\R^n_{\geq 0} \setminus \{ \0\}$. 
    Then the $\bbu$-parking function polytope $\pf(\bbu)$ has the following inequality description: 
\[ \pf(\bbu) = \left\{ 
	\bbx \in \R^n \ \left| \ \begin{matrix}
		\sum_{i \in I}x_i \leq \sum^{|I|-1}_{i = 0} u_{n-i}, & \text{ for every $\emptyset \neq I \subseteq [n]$} \\[2mm]
		x_i  \ge 0, & \text{ for each $i \in [n]$}
	\end{matrix}
\right.
	\right\}
\]
Moreover, for each $i \in [n]$, the inequality $x_i \ge 0$ always defines a facet of $\pf(\bbu)$. 
For each nonempty subset $I$ of $[n],$ the inequality $\sum_{i \in I}x_i \leq \sum^{|I|-1}_{i = 0} u_{n-i}$ defines a facet of $\pf(\bbu)$ if and only if the following conditions on $|I|$ are satisfied:
\begin{enumerate}
\item \label{itm:cube} 
	If $\bbu=(d,d,\dots, d)$ for some positive $d$, then $|I| =1$.
\item \label{itm:simplex} 
	If $\bbu=(0,\dots,0, d)$ for some positive $d$, then $|I| =n$.
\item \label{itm:hypersimplex} 
	If $\bbu$ is of the form $(0, \dots, 0, d, \dots, d)$ for some positive $d$ with at least one $0$ and at least two $d$'s, then $|I|=1$ or $n$.
\item \label{itm:lge2} 
	If $\bbu$ contains at least two distinct positive entries and let $\bbm=(m_0, m_1, \dots, m_\ell)$ be its multiplicity vector, then $|I| \in \{1, n\} \cup [m_\ell +1, n-m_0-1]$.
\end{enumerate}
\end{prop}


\begin{ex}\label{ex: ineqDescription}
    We present situations \eqref{itm:cube} and \eqref{itm:simplex} of Proposition \ref{prop: ineqDescription}.

    Suppose $\bbu=(d, \dots, d)$ for some positive $d$. Note that the multiplicity vector of $\bbu$ is $\bbm=(0,n)$. It is easy to see from the original definition of parking function polytopes that $\pf(d,d,\dots,d)$ is the cube $\square_n^d$ whose vertex set is $\{0, d\}^n$. Applying Proposition \ref{prop: ineqDescription}, we obtain that $\pf(\bbu)$ has an inequality description:
    \[ x_i \ge 0 \text{ and } x_i \le d \text{ for all } i \in [n],\]
    in which each inequality is facet-defining, as expected.  

    Suppose $\bbu=(0, \dots, 0, d)$ for some positive $d$. Note that the multiplicity vector of $\bbu$ is $\bbm=(n-1,1)$. It is easy to see from the original definition of parking function polytopes that $\pf(0,\dots, 0, d)$ is the simplex whose vertice set is $\{ \0, d \bbe_1, d \bbe_2, \dots, d \bbe_n\}.$ Applying Proposition \ref{prop: ineqDescription}, we obtain that $\pf(\bbu)$ has an inequality description:
    \[ x_i \ge 0 \text{ for all } i \in [n], \text{ and } x_1 + x_2 + \cdots + x_n \le d.\]
    in which each inequality is facet-defining, as expected.   
\end{ex}

\subsubsection{Simple and simplicial parking function polytopes}
It is apparent that every polytope of dimension at most 2 is simple. When a polytope has dimension greater than 2, it becomes nontrivial to verify its simplicity. Theorem \ref{thm: fulldimcones} allows us to determine exactly when $\pf(\bbm, \bbd)$ is a simple polytope.

\begin{cor}\label{cor: simple}
    Let $(\bbm, \bbd)$ be an MD pair where $\bbm = (m_0, m_1, \dots, m_\ell)$. Then $\pf(\bbm,\bbd)$ is simple if and only if either $\bbm = (0,n)$ or $(n-1, 1)$ or $m_1 = \cdots = m_{\ell-1} = 1$ for some $\ell \geq 2.$
\end{cor}
\begin{proof}
    Recall that an $n$-dimensional polytope in $\R^n$ is simple if and only if, for every vertex of the polytope, the normal cone at the vertex is simplicial. Since $\{\bbv_\calA \mid \type(\calA) \in \Omega_\bbm\}$ is the set of vertices of $\pf(\bbm,\bbd)$ and $\ncone(\bbv_\calA,\pf(\bbm,\bbd)) = \ssigma_\calA$, it suffices by Lemma \ref{lem: preordcone}/\ref{itm: simplicial} to show that the preorder $\preceq_\calA$ defines a poset whose Hasse diagram is a tree for all $\calA$ satisfying $\type(\calA) \in \Omega_\bbm$ if and only if either $\bbm = (0,n)$ or $(n-1, 1)$ or $m_1 = \cdots = m_{\ell-1} = 1$ for some $\ell \geq 2.$ It is straightforward to verify that the "if" direction of the above statement is correct. Hence, we just need to verify the "only if" part. Suppose $\bbm=(m_0, m_1, \dots, m_\ell)$ satisfies (i) $\bbm= (m_0, m_1)$ where $m_0 \ge 1$ and $m_1 \ge 2$ or (ii) $\ell\ge2$ and $m_i \ge 2$ for some $1 \le i \le \ell-1.$ Let $\bba = \bba_0$ as defined in Definition \ref{def: vertex-type}. One checks that the Hasse diagram of the poset associated with any $\preceq_\calA$ such that $\type(\calA)=\bba$ is not a tree. This completes the proof.
\end{proof}

Note that every polytope of dimension at most two is simplicial. The next corollary provides a characterization of the simplicial parking function polytopes of dimension greater than two.

\begin{cor}\label{cor: simplicial-pf}
    Let $n \geq 3$ be an integer. Suppose that $\bbm$ is a multiplicity vector of magnitude $n$ and $(\bbm, \bbd)$ is an MD pair. Then, $\pf(\bbm, \bbd)$ is an $n$-dimensional simplicial polytope if and only if $\bbm = (n-1,1),$ i.e., $\pf(\bbm,\bbd)$ is a simplex.
\end{cor}

\begin{proof} 
As shown in Example \ref{ex: ineqDescription}, if $\bbm = (n-1, 1)$, the parking function polytope $\pf(\bbm, \bbd)$ is a simplex and thus is simplicial. 

Conversely, suppose $\pf(\bbm, \bbd)$ is simplicial. Let $\bbu=(u_1, \dots, u_n)$ be the vector whose MD pair is $(\bbm, \bbd)$. By Proposition \ref{prop: ineqDescription}, the inequality
\[ x_1 \ge 0 \]
defines a facet, which we denote by $F$, of $\pf(\bbu).$ As vertices of $F$ are the $\bbu$-extreme points $\bbx$ that attains the equality of the above inequality, i.e., $x_1=0,$ one observes that 
\[ F = \{0\} \times \pf(u_2, u_3, \dots, u_n).\]
Given that $\pf(\bbu)=\pf(\bbm,\bbd)$ is simplicial, we must have that $\pf(u_2, u_3, \dots, u_n)$ is an $(n-1)$-dimsional simplex, which is simplicial. Iterating the arguments above, we will be able to show that $\pf(u_{n-1}, u_n)$ is a $2$-dimensional simplex, i.e., a triangle. However, as we have shown in Figure \ref{fig: pf-types-dim2} that there are only three combinatorial different $2$-dimensional parking polytopes, and among these the only triangle is the one whose associated multiplicity vector is $(1,1).$ Hence, the multiplicity vector of $(u_{n-1}, u_n)$ is $(1,1)$. This implies that $u_{n-1}=0$ and $u_n > 0.$ Since the vector $\bbu=(u_1,\dots, u_n)$ satisfies
\[ 0 \le u_1 \le u_2 \le \dots \le u_{n-1} \le u_n,\]
we conclude that $\bbu=(0,0, \dots, 0, u_n),$ whose multiplicity vector is $(n-1,1).$
\end{proof}

\section{\texorpdfstring{$h$-vectors}{h-vectors}}

Given a poset $(Q,\leq_Q)$ where $Q \subset \N,$ we say that the ordered pair $(i,j)$ is a \textit{descent} of $(Q,\leq_Q)$ if $i \lessdot_Q j$ and $j < i$, and say that $(i,j)$ is an \textit{ascent} if $i \lessdot_\calB j$ and $j > i$.

As noted in Corollary \ref{cor: simple}, $\pf(\bbm, \bbd)$ is simple if and only if either $\bbm = (0,n)$ or $(n-1, 1)$ or $m_1 = \cdots = m_{\ell-1} = 1$ for some $\ell \geq 2.$ This implies that for every $\calB \in \Omega_\bbm$, the preorder $\leq_\calB$ is a poset and its Hasse diagram is a tree. We will denote the number of descents and ascents of the poset $([0,n], \leq_\calB)$ by $\des(\calB)$ and $\asc(\calB)$, respectively. The following lemma, which is a slight variation of \cite[Theorem 4.2]{Postnikov2006}, expresses the $h$-polynomials of simple parking function polytopes in terms of descents and ascents.

\begin{lem}\label{lem: hpoly-descents}
    If $\pf(\bbm, \bbd)$ is an $n$-dimensional simple polytope, then its $\bbh$-polynomial equals 
    \begin{equation}\label{eq: hpolynomial}
        h(t) = \sum_{\type(\calB) \in \Omega_\bbm}t^{\des(\calB)} = \sum_{\type(\calB) \in \Omega_\bbm}t^{\asc(\calB)}.
    \end{equation}
\end{lem}

\begin{cor}\label{cor: h-group-by-types}
    Let $\bbm = (m_0, m_1, \dots, m_\ell)$ and $r = m_1 + \cdots + m_\ell$. If $\pf(\bbm, \bbd)$ is an $n$-dimensional simple polytope, then its $\bbh$-polynomial equals 
    \begin{equation}\label{eq: hpolynomial2}
        h(t) = \sum^{r}_{i = 0}\left(\sum_{\type(\calB) = \bba_i}t^{\des(\calB)}\right) = \sum^{r}_{i = 0}\left(\sum_{\type(\calB) = \bba_i}t^{\asc(\calB)}\right)
    \end{equation}
    where $\bba_i$ is given as in Definition \ref{def: vertex-type}.
\end{cor}

Let $Q$ be a poset on $[n]$. We define $\fS_n(Q) := \{\sigma (Q)\ |\ \sigma \in \fS_n\}$ to be the set of all posets on $[n]$ having the same Hasse diagram as $Q$. 

\begin{defn}
    For $p,q\in \N$, let $T(p,q)$ be the poset on $[p+q]$ defined by the covering relations $j \lessdot j+1$ for all $j \in [p-1]$ and $p \lessdot k$ for all $k \in [p+1, q].$
\end{defn}

\begin{defn}\label{def: genEulerianPolynomial}
    Let $(Q, \leq_T)$ be a poset on $[n]$ whose Hasse diagram is a tree (a graph with no cycle). We define the \emph{generalized Eulerian polynomial} on $Q$ to be 
    \[A(Q,t) := \sum_{T \in \fS_n(Q)} t^{\asc(T)}.\]
\end{defn}

Let $\tau = n\cdots21 \in \fS_n.$ Then $\asc(\tau(T)) = \des(T)$ for all $ T \in \fS_n (Q)$. Thus,
\[A(Q,t) := \sum_{T \in \fS_n(Q)} t^{\asc(T)} = \sum_{T \in \fS_n(Q)} t^{\asc(\tau(T))} = \sum_{T \in \fS_n(Q)} t^{\des(T)}.\]

The generalized Eulerian polynomial $A(Q,t)$ has degree $n-1$ and is palindromic. We also have that $A(T(p,1),t)$ is the (usual) Eulerian polynomial $A_{p+1}(t)$ of degree $p.$

\begin{rem}\label{rmk: h-square-simplex}
    $\pf(\bbm, \bbd)$ is an $n$-cube if $\bbm = (0,n)$, and is an $n$-simplex if $\bbm = (n-1,1)$. In these cases, their $h$-polynomials are known to be $(1+t)^n$ and $1+t + \cdots t^n$, respectively.
\end{rem}

We now describe how to obtain a more explicitly formula for the $h$-polynomials of all other simple parking function polytopes, i.e., $\pf(\bbm, \bbd)$ with $m_1 = \cdots = m_{\ell-1} = 1$ for some $\ell \geq 2.$ Note that $n = m_0+\ell -1 + m_\ell$ and $r = m_1 + \cdots + m_\ell = \ell -1 + m_\ell$. Thus, by Corollary \ref{cor: h-group-by-types}, we may write 
\begin{align}\label{eq: h-sum-of-g_i}
    h(t) = \sum^{\ell -1 +m_\ell}_{i = 0}g_i(t) \text{ where } g_i(t) := \sum_{\type(\calB) = \bba_i}t^{\asc(\calB)}.
\end{align}

For $\ell-1 \leq i \leq \ell -1 + m_\ell,$ the poset $([0,n], \preceq_\calB)$ with $\type(\calB) = \bba_i$ is given on the right of Figure \ref{fig: types-b-i}. Thus, 
\begin{align}\label{eq: third-g_i}
    g_i(t) = \binom{n}{m_0 + i}t^{\ell - 1 + m_\ell -i} \text{ for } \ell-1 \leq i \leq \ell -1 + m_\ell.
\end{align}

For $1 \leq i \leq \ell - 2,$ the poset $([0,n], \preceq_\calB)$ with $\type(\calB) = \bba_i$ is given on the left of Figure \ref{fig: types-b-i}. Thus, 
\begin{align}\label{eq: second-g_i}
    g_i(t) = \binom{n}{m_0 + i}tA(T(\ell - i -1, m_\ell),t) \text{ for } 1 \leq i \leq \ell - 2.
\end{align}

\begin{figure}
    \centering
\begin{tikzpicture}[scale = 1]

        \node (min12) at (-6.5+9,1) {$*$};
        \draw (-6.5+9,1) circle (7pt);
        \node (min13) at (-5.5+9,1) {$\cdots$};
        \node (min14) at (-4.5+9,1) {$*$};
        \draw (-4.5+9,1) circle (7pt);
        \node (s21) at (-5.5+9,2) {$0$};
        \draw (-5.5+9,2) circle (7pt);
        \node (max61) at (-6.5+9,3) {$*$};
        \draw (-6.5+9,3) circle (7pt);
        \node (max63) at (-5.5+9,3) {$\cdots$};
        \node (max64) at (-4.5+9,3) {$*$};
        \draw (-4.5+9,3) circle (7pt);
        \draw  (min12)--(s21) (min14)--(s21) (s21)--(max61) (s21)--(max64);
        \draw [decorate, blue,decoration={brace,amplitude=5pt,raise=4ex}]
  (2,3) -- (5,3) node[midway,yshift=3em]{$m_\ell+\ell -1 -i$ nodes};

        \node (min22) at (-6.5+2.5,1) {$*$};
        \draw (-6.5+2.5,1) circle (7pt);
        \node (min23) at (-5.5+2.5,1) {$\cdots$};
        \node (min24) at (-4.5+2.5,1) {$*$};
        \draw (-4.5+2.5,1) circle (7pt);
        \node (s22) at (-5.5+2.5,2) {$0$};
        \draw (-5.5+2.5,2) circle (7pt);
        \node (s32) at (-5.5+2.5,3) {$*$};
        \draw (-5.5+2.5,3) circle (7pt);
        \node (s42) at (-5.5+2.5,4) {$\vdots$};
        \node (s52) at (-5.5+2.5,5) {$*$};
        \draw (-5.5+2.5,5) circle (7pt);
        \node (max72) at (-6.5+2.5,6) {$*$};
        \draw (-6.5+2.5,6) circle (7pt);
        \node (max73) at (-5.5+2.5,6) {$\cdots$};
        \node (max74) at (-4.5+2.5,6) {$*$};
        \draw (-4.5+2.5,6) circle (7pt);
        \draw (min22)--(s22) (min24)--(s22) (s22)--(s32) (s32)--(s42) (s42)--(s52) (s52)--(max72) (s52)--(max74);
        \draw [decorate, blue,decoration={brace,amplitude=5pt,raise=4ex}]
  (-1-3.5,6) -- (2-3.5,6) node[midway,yshift=3em]{$m_\ell$ nodes};
        \draw [decorate, blue,decoration={brace,amplitude=5pt,raise=4ex}]
  (-3.5,2.7) -- (-3.5,5.2) node[midway,xshift=-6em, yshift=0em]{$\ell-1-i$ nodes};
\end{tikzpicture}
\captionsetup{justification=centering}
    \caption{$([0,n], \preceq_\calB)$ with $\type(\calB) = \bba_i \in \Omega_\bbm$ satisfying $1 \leq i \leq \ell-2$ (left) and $\ell-1 \leq i \leq m_\ell$ (right), where $*$ denotes an integer in $[n]$}
    \label{fig: types-b-i}
\end{figure}
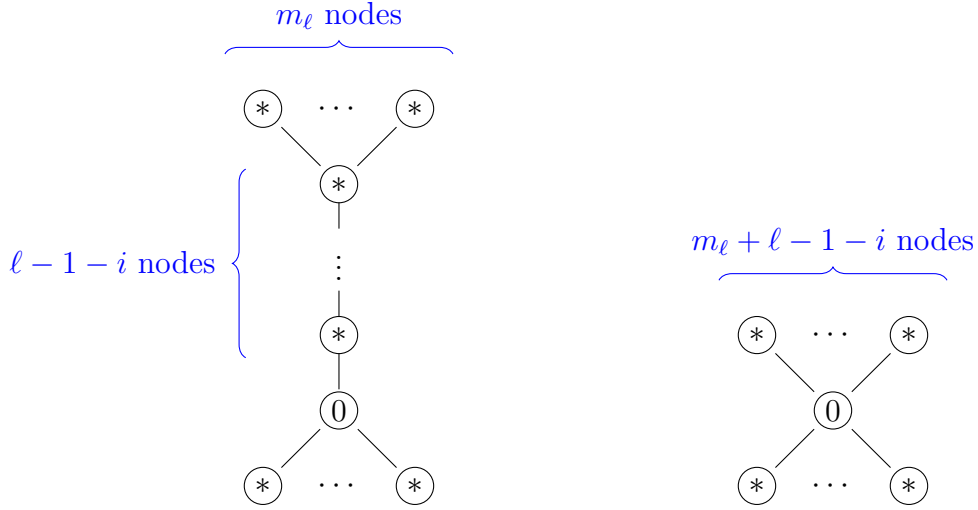

We now aim to compute the polynomial $g_0(t)$ and show that $g_0(t) = tA(Q,t)$ where $Q$ is a poset on $[n]$ depending on whether $\bbm$ satisfies $m_0 = 0$ or not. If $m_0 = 0,$ then the poset $([0,n], \preceq_\calB)$ with $\type(\calB) = \bba_0$ is shown on the left of Figure \ref{fig: types-b0}. Thus, by setting $Q = T(\ell -1, m_\ell),$ we have
\begin{align}\label{eq: g0-m0-0}
    g_0(t) = tA(T(\ell -1, m_\ell),t) \text{ if } m_0 = 0.
\end{align}

If $m_0 \neq 0,$ then the poset $([0,n], \preceq_\calB)$ with $\type(\calB) = \bba_0$ is shown on the right of Figure \ref{fig: types-b0}. Thus, by setting $Q$ to be the induced poset on $[n]$ of the poset on the right of Figure \ref{fig: types-b0}, we have
\begin{align}\label{eq: g0-m0-neq-0}
    g_0(t) = tA(Q,t) \text{ if } m_0 \neq 0.
\end{align}

\begin{figure}[ht]
    \centering
\begin{tikzpicture}[scale = 1]
        \node (min11) at (-7+9,1) {0};
        \draw (-7+9,1) circle (7pt);
        \node (min12) at (-6+9,1) {$*$};
        \draw (-6+9,1) circle (7pt);
        \node (min13) at (-5+9,1) {$\cdots$};
        \node (min14) at (-4+9,1) {$*$};
        \draw (-4+9,1) circle (7pt);
        \node (s21) at (-5.5+9,2) {$*$};
        \draw (-5.5+9,2) circle (7pt);
        \node (s31) at (-5.5+9,3) {$*$};
        \draw (-5.5+9,3) circle (7pt);
        \node (s41) at (-5.5+9,4) {$\vdots$};
        \node (s51) at (-5.5+9,5) {$*$};
        \draw (-5.5+9,5) circle (7pt);
        \node (max61) at (-6.5+9,6) {$*$};
        \draw (-6.5+9,6) circle (7pt);
        \node (max63) at (-5.5+9,6) {$\cdots$};
        \node (max64) at (-4.5+9,6) {$*$};
        \draw (-4.5+9,6) circle (7pt);
        \draw (min11)--(s21) (min12)--(s21) (min14)--(s21) (s21)--(s31) (s31)--(s41) (s41)--(s51) (s51)--(max61) (s51)--(max64);
        \draw [decorate, blue,decoration={brace,amplitude=5pt,raise=4ex}]
  (2,6) -- (5,6) node[midway,yshift=3em]{$m_\ell$ nodes};
        \draw [decorate, blue,decoration={brace,amplitude=5pt, mirror,raise=4ex}]
  (4.5,1.7) -- (4.5,5.2) node[midway,xshift=5em, yshift=0em]{$\ell-1$ nodes};

        \node (s22) at (0.5-3.5,1) {$0$};
        \draw (0.5-3.5,1) circle (7pt);
        \node (s32) at (0.5-3.5,2) {$*$};
        \draw (0.5-3.5,2) circle (7pt);
        \node (s42) at (0.5-3.5,3.5) {$\vdots$};
        \node (s52) at (0.5-3.5,5) {$*$};
        \draw (0.5-3.5,5) circle (7pt);
        \node (max71) at (-0.5-3.5,6) {$*$};
        \draw (-0.5-3.5,6) circle (7pt);
        \node (max73) at (0.5-3.5,6) {$\cdots$};
        \node (max74) at (1.5-3.5,6) {$*$};
        \draw (1.5-3.5,6) circle (7pt);
        \draw  (s22)--(s32) (s32)--(s42) (s42)--(s52) (s52)--(max71) (s52)--(max74);
        \draw [decorate, blue,decoration={brace,amplitude=5pt,raise=4ex}]
  (-1-3.5,6) -- (2-3.5,6) node[midway,yshift=3em]{$m_\ell$ nodes};
\end{tikzpicture}
\captionsetup{justification=centering}
    \caption{$([0,n], \preceq_\calB)$ with $\type(\calB) = \bba_0 \in \Omega_\bbm$ satisfying $m_0 = 0$ (left) and $m_0 \neq 0$ (right), where $*$ denotes an integer in $[n]$}
    \label{fig: types-b0}
\end{figure}
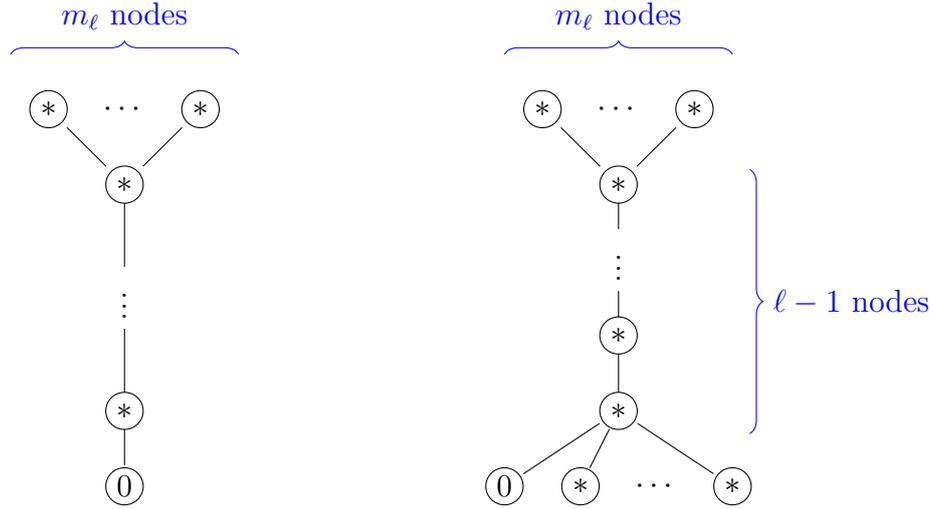

\noindent Consequently, by equations \eqref{eq: h-sum-of-g_i} and \eqref{eq: third-g_i}--\eqref{eq: second-g_i},
\begin{align}
    h(t) &= g_0(t) + \left(\sum^{\ell -2}_{i = 1}g_i(t)\right) + \sum^{\ell - 1 + m_\ell}_{i = \ell -1}g_i(t)\nonumber\\
    &= tA(Q,t) + \left(\sum^{\ell -2}_{i = 1}\binom{n}{m_0 + i}tA(T(\ell - i -1, m_\ell),t)\right) + \sum^{\ell - 1 + m_\ell}_{i = \ell -1}\binom{n}{m_0 + i}t^{\ell - 1 + m_\ell -i}\nonumber\\
\label{eq: explicitGenEulerian}    &= tA(Q,t) + \left(\sum^{\ell -2}_{i = 1}\binom{n}{i+m_\ell}tA(T(i, m_\ell),t)\right) + \sum^{m_\ell}_{i = 0}\binom{n}{i}t^{i}.
\end{align}
Hence, if $m_0 = 0,$ then by \eqref{eq: g0-m0-neq-0} we have
\begin{align}
    h(t) &= tA(Q,t) + \left(\sum^{\ell -2}_{i = 1}\binom{n}{i+m_\ell}tA(T(i, m_\ell),t)\right) + \sum^{m_\ell}_{i = 0}\binom{n}{i}t^{i}\nonumber\\
    &= tA(T(\ell -1, m_\ell),t) + \left(\sum^{\ell -2}_{i = 1}\binom{n}{i+m_\ell}tA(T(i, m_\ell),t)\right) + \sum^{m_\ell}_{i = 0}\binom{n}{i}t^{i} \nonumber\\
\label{eq: explicitGenEulerian-m0=0}\tag{M0}    &= \left(\sum^{m_\ell}_{i = 0}\binom{n}{i}t^{i}\right) + \sum^{\ell -1}_{i = 1}\binom{n}{i+m_\ell}tA(T(i, m_\ell),t).
\end{align}

Since $h$-polynomials and generalized Eulerian polynomials are palindromic, we have 
\begin{align}
\label{eq: h-palindromic}    h(t) &= t^{n}h(t^{-1})\\
\label{eq: g0-palindromic}    A(Q,t) &= t^{n-1}A(Q,t^{-1})\\
\label{eq: Eulerian-palindromic}   A(T(i, m_\ell),t) &=  t^{i + m_\ell-1} A(T(i, m_\ell),t^{-1}).
\end{align}
Thus, by equations \eqref{eq: explicitGenEulerian} and \eqref{eq: h-palindromic}--\eqref{eq: Eulerian-palindromic}, 
\begin{align}
    h(t) &= t^{n}h(t^{-1})\nonumber\\
     &= t^n\left[ t^{-1}A(Q,t^{-1}) + \left(\sum^{\ell -2}_{i = 1}\binom{n}{i+m_\ell}t^{-1}A(T(i, m_\ell),t^{-1})\right) + \sum^{m_\ell}_{i = 0}\binom{n}{i}t^{-i}\right]\nonumber\\
    &= t^{n-1}A(Q,t^{-1}) + \left(\sum^{\ell -2}_{i = 1}\binom{n}{i+m_\ell}t^{n-1}A(T(i, m_\ell),t^{-1})\right) + \sum^{m_\ell}_{i = 0}\binom{n}{i}t^{n-i} \nonumber\\
\label{eq: explicitGenEulerian-palindromic}    &= A(Q,t) + \left(\sum^{\ell -2}_{i = 1}\binom{n}{i+m_\ell}t^{m_0+\ell -i-1}A(T(i, m_\ell),t)\right) + \sum^{m_\ell}_{i = 0}\binom{n}{i}t^{n-i}
\end{align}
Equations \eqref{eq: explicitGenEulerian} and \eqref{eq: explicitGenEulerian-palindromic} allow us to express $A(Q,t)$ as
\begin{align}
    A(Q,t) &= \frac{1}{t-1}\left[\left[\sum^{\ell -2}_{i = 1}\binom{n}{i+m_\ell}(t^{m_0+\ell -i-1}-t)A(T(i, m_\ell),t)\right] + \sum^{m_\ell}_{i = 0}\binom{n}{i}(t^{n-i}-t^i)\right].\nonumber
\end{align}
Hence, if $m_0 \neq 0,$ then
\begin{align}\label{eq: eq: explicitGenEulerian-m0-neq-0}\tag{M1}
    h(t) &= g(t) +  \left(\sum^{m_\ell}_{j = 0}\binom{n}{j}t^j\right) + t\sum^{\ell - 2}_{i = 1}\binom{n}{i+m_\ell}A(T(i,m_\ell),t),
\end{align}
where 
\begin{align}
    g(t) = g_0(t) &= tA(Q,t)\nonumber\\
\label{eq: g0}\tag{G0}    &= \left[\sum^{z}_{i = 0}\binom{n}{i}\left(\sum^{n-i}_{j = i+1}t^j\right)\right] + \sum^{\ell - 2}_{i = 1}\binom{n}{i+m_\ell}\left(\sum^{n-i-m_\ell}_{j = 2}t^j\right)A(T(i,m_\ell),t)
\end{align}
and $z = \min\left(m_\ell, n-m_\ell-1\right).$ We conclude our calculation in the following theorem.

\begin{thm}\label{thm: h-polynomial} Let $(\bbm, \bbd)$ be an MD pair where $\bbm = (m_0, m_1, \dots, m_\ell)$ for some $\ell \geq 2.$ Suppose that $\pf(\bbm, \bbd)$ is $n$-dimensional and simple. Then its $h$-polynomial is given by 
    \begin{align}\label{eq: h-explicit-thm}\tag{H}
    h(t) &=
        \begin{cases}
          \mathlarger{\left(\sum^{m_\ell}_{j = 0}\binom{n}{j}t^j\right) + t\sum^{\ell - 1}_{i = 1}\binom{n}{i+m_\ell}A(T(i,m_\ell),t)}  &\text{ if } m_0 = 0\\
         \mathlarger{g(t) +  \left(\sum^{m_\ell}_{j = 0}\binom{n}{j}t^j\right) + t\sum^{\ell - 2}_{i = 1}\binom{n}{i+m_\ell}A(T(i,m_\ell),t)}   &\text{ otherwise }
        \end{cases}
    \end{align}
where
    \begin{align*}
    g(t) &= \left[\sum^{z}_{i = 0}\binom{n}{i}\left(\sum^{n-i}_{j = i+1}t^j\right)\right] + \sum^{\ell - 2}_{i = 1}\binom{n}{i+m_\ell}\left(\sum^{n-i-m_\ell}_{j = 2}t^j\right)A(T(i,m_\ell),t) 
    \end{align*}
and $z = \min\left(m_\ell, n-m_\ell -1\right).$
\end{thm}

We now aim to express $A(T(p,q),t)$ in terms of Eulerian polynomials. This will allow us to express equation \eqref{eq: h-explicit-thm} in Theorem \ref{thm: h-polynomial} in terms of Eulerian polynomials. To do this, we first observe that, by symmetry, $A(Q,t) = A(Q^*,t)$ for all poset $Q$ on $[n]$ whose Hasse diagram is a tree, where $Q^*$ denotes the dual poset of $Q.$ 

Let $p,q,n \in \P$ satisfy $p+q = n.$ If $p = 1,$ one can easily compute by a direct counting argument that $A(T(1,q),t) = 1 + t + \cdots t^{n-1}.$ If $p \geq 2,$ then $T(p,q)^*$ is the induced poset on $[n]$ of a poset $(\leq_\calB, [0,n])$ satisfying $\type(\calB) = \bbb_0 \in \Omega_\bbm$ where $\bbm = (q, 1, \dots, 1).$ Thus, we can deduce $A(T(p,q),t)$ by applying equation \eqref{eq: g0} to $\bbm = (q, 1, \dots, 1)$ to get
\begin{align}
    A(T(p,q),t) &= A(T(p,q)^*,t)\nonumber\\
    &= \frac{g(t)}{t}\nonumber\\
    &=\left[\sum^{1}_{i = 0}\binom{n}{i}\left(\sum^{n-i-1}_{j = i}t^j\right)\right] + \sum^{p - 2}_{i = 1}\binom{n}{i+1}\left(\sum^{n-i-2}_{j = 1}t^j\right)A(T(i,1),t), \text{ if } p \geq 2.\nonumber
\end{align}

Therefore, we have the following result.
\begin{lem}\label{lem: T(pq)}
    Let $p,q,n \in \P$ satisfy $p+q = n.$ Then 
    \begin{align}\label{eq: genEulerianExplicit}
    A(T(p,q),t)
    &=\left[\sum^{y}_{i = 0}\binom{n}{i}\left(\sum^{n-i-1}_{j = i}t^j\right)\right] + \sum^{p - 2}_{i = 1}\binom{n}{i+1}\left(\sum^{n-i-2}_{j = 1}t^j\right)A_{i+1}(t)
\end{align}
where $y = \min(1,p-1)$ and $A_k(t)$ denotes the Eulerian polynomial of degree $k-1.$
\end{lem}

Using equation \eqref{eq: genEulerianExplicit} to express $A(T(i,m_\ell), t)$ in equation \eqref{eq: h-explicit-thm} of Theorem \ref{thm: h-polynomial}, we can write the $h$-polynomial of $\pf(\bbm,\bbd)$ in terms of Eulerian polynomials. Together with Remark \ref{rmk: h-square-simplex}, we consequently have the following corollary.

\begin{cor}
Suppose that $\pf(\bbu)$ is $n$-dimensional and simple. Then its $h$-polynomial has the form $h(t) = r_0(t) + \sum^n_{i = 1}r_i(t)A_i(t)$ where $A_k(t)$ is the Eulerian polynomial of degree $k-1$ and $r_k(t)$ is a polynomial with nonnegative integral coefficients of degree $\leq n.$
\end{cor}

For instance, the $h$-polynomials of $\pf(1, \dots, n)$ and $\pf(0, \dots, n-1)$ equal
\[1+ \sum^{n}_{k = 1}\binom{n}{k}t A_k(t) \text{ and } 1+ tA_n(t) + \sum^{n-2}_{k = 1}\binom{n}{k}t A_k(t), \text{ respectively}.\]

\section{Connection to other polytopes}\label{sec: connection-to-other-polytopes}

Let $\fS_n(\bbu) := \conv(\tau(\bbu)\ |\ \tau \text{ is a permutation in } \fS_n)$ be the $\fS_n$-permutohedron generated by $\bbu,$ where $\tau(\bbu) := (u_{\tau(1)}, \dots,u_{\tau(n)})$. It follows from Proposition \ref{prop: ineqDescription} that the parking function polytope $\pf(\bbu)$ can be equivalently defined as 
\begin{align}
\label{eq: pfu-description1}    \pf(\bbu) &= \{\bbx \in \R^n_{\geq 0}\ |\ \exists\, \bbw \in \fS_n(\bbu) \text{ such that } \bbw - \bbx \in \R^n_{\geq 0}\}\\
\label{eq: pfu-description2} &= (\R^n_{\leq 0} + \fS_n(\bbu))\cap \R^n_{\geq 0}.
\end{align}

\subsection{Viewed as polymatroid} \label{subsec:polymatroid}

Equations \eqref{eq: pfu-description1}--\eqref{eq: pfu-description2} allow us to see that every $\pf(\bbu)$ is a polymatroid introduced by Edmonds in \cite{Edmonds1970}. 
Recall that a \emph{polymatroid rank function on $[n]$} is a function 
$w: 2^{[n]}\longrightarrow \R$ on the power set of $[n]$ satisfying the following conditions:
\begin{enumerate}
    \item \label{itm: nonneg} (Nonnegative) $0 \leq w(I)$ for all $I \subseteq [n].$
    \item \label{itm: non-decreasing} (Non-decreasing) $w(I_1) \leq w(I_2)$ for all $I_1 \subseteq I_2 \subseteq [n]$
    \item \label{itm: submodular} (Submodular) $w(A) + w(B) \geq w(A\cup B) + w(A\cap B)$ for all $A, B \subseteq [n].$
\end{enumerate}
The \emph{polymatroid} $P(w)$ associated to a polymatroid rank function $w$ on $[n]$ is the polytope in $\R^{n}$ defined by the system of linear inequalities
\begin{equation}\label{eq:defnpolymatroid}
P(w)=\left\{ \bbx \in \R^n \ \left| \  0 \leq \sum_{i \in I}x_i \leq w_\bbu(I) \text{ for all nonempty } I \subseteq [n] \right. \right\}.
\end{equation}
The name ``polymatroid'' comes from its direct relation to matroids, the study concerning the abstraction of independent sets. 

For a given non-decreasing $\bbu \in \R_{\ge 0}^n$,  let $w_\bbu: 2^{[n]} \rightarrow \R$ be the function defined by
\[w_\bbu(\emptyset) = 0 \text{ and } w_\bbu(I) = \sum^{|I|}_{i = 0}u_{n-i} \text{ for all nonemmpty } I \subseteq [n].\]
It is easy to verify that $w_\bbu$ satisfies conditions \eqref{itm: nonneg}--\eqref{itm: submodular}, and thus is a polymatroid rank function. One can also verify that the inequality description of $\pf(\bbu)$ given in Proposition \ref{prop: ineqDescription} is equivalent to \eqref{eq:defnpolymatroid}. Hence, $\pf(\bbu) = P(w_\bbu)$ is a polymatroid.

\begin{rem}
    It was shown by Fujishige \cite[Theorem 17.1]{Fujishige2005submodular} that every polymatroid has edges paralell to $\bbe_i$ or $\bbe_i - \bbe_j$ for some positive integers $1 \leq i < j \leq n.$ Thus, every polymatroid, and consequently every parking function polytope, is a type $B$ generalized permutohedron.
\end{rem}

Let $\one_k := (0, \dots, 0, 1,\dots, 1) \in \R^n$ be the vector with the last $k$ coordinates equal to $1$ and the remaining coordinates equal to $0$. Then, every parking function polytope is a Minkowski sum of $\pf(\one_k)$ as shown in the following proposition.

\begin{prop}\label{prop: Minkowskisum-hypersimplices}
    Let $\bbu = (u_1, \dots, u_n) \in \R^n$. Then
    \begin{align}\label{eq: Minkowskisum-hypersimplices}
        \pf(\bbu) = u_{1}\pf(\one_n) + \sum_{k = 1}^{n-1} (u_{k+1}-u_{k})\pf(\one_{n-k}).
    \end{align}
\end{prop}

\begin{proof}
    Let $Q$ be the right-hand side of equation \eqref{eq: Minkowskisum-hypersimplices}. Suppose that $\bby$ is a point in $Q.$ Then, $\bby = \bbx_1 + \cdots + \bbx_n$ where $\bbx_1 \in u_1\pf(\one_n)$ and $\bbx_k \in (u_{k+1} - u_{k})\pf(\one_k)$ for $1 \leq k \leq n-1.$ Using Proposition \ref{prop: ineqDescription} for $\pf(\bbu)$ and $\pf(\one_k)$, one deduces that $\bbx_1 + \cdots + \bbx_n$ satisfies the inequality description for $\pf(\bbu).$ Thus, we have $\bby \in \pf(\bbu)$. This shows $Q \subseteq \pf(\bbu).$

    Now suppose that $\bbx \in \calX(\bbu)$ is a $\bbu$-extreme point, a vertex of $\pf(\bbu)$. It is not difficult to see that 
    \[\bbx \in u_{1}\calX(\one_n) + \sum_{k = 1}^{n-1} (u_{k+1}-u_{k})\calX(\one_{n-k}) \subseteq Q.\] 
    This implies that $\pf(\bbu) \subseteq Q.$ 
\end{proof}

\subsection{Viewed as type \texorpdfstring{$A$}{A} generalized permutohedra}
A type $A$ generalized permutohedron in $\R^n$ is a polytope whose edges are paralell to $e_i - e_j$ for some $1 \leq i < j \leq n.$ One may also equivalently define it as a polytope whose normal fan coarsen the normal fan of the permutohedron $\fS_n(1, \dots, n)$ generated by $(1, \dots, n)$, i.e., a polytope that is a deformation of $\fS_n(1, \dots, n)$. Type $A$ generalized  permutohedra form a widely studied family of polytopes that interconnect with many mathematical concepts, such as matroids and Weyl groups. We refer readers to \cite{Postnikov2005} for more details on type $A$ generalized permutohedra.

As we shall see later in the subsection, every parking function polytope can be realized as a projection of a type $A$ generalized permutohedron. 
Moreover, every integral parking function polytope is integrally equivalent to an integral type $A$ generalized permutohedron. This realization was also pointed out in \cite{bayer2024parkingfunctionpolytopes} for $\pf(\bbm, \bbd)$ where $\bbm = (0,1,\dots, 1)$ and $(1,\dots, 1)$. Thus, many properties of generalized permutohedra also apply to parking function polytopes. In particular, one can compute the Ehrhart polynomials and the volume of parking function polytopes using existing formulas for type $A$ generalized  permutohedra.

Postnikov \cite[Proposition 6.3]{Postnikov2005} and Ardila et al. \cite[Proposition 2.3]{Ardila2010} show that every type $A$ generalized permutohedron can be written a Minkowski sum (and difference) of simplices. That is, every type $A$ generalized  permutohedron has the form
\begin{align}\label{eq: typeA-Minkowski-sum}
    \sum_{I \in 2^{[n]}\backslash\{\emptyset\}}y_I\Delta_I \text{ for some } y_I \in \R,
\end{align}
where $\Delta_I := \conv(\bbe_i\ | \ i \in I)$. Moreover, in Sections 9 and 11 of \cite{Postnikov2005}, Postnikov gave formulas for the volume and the Ehrhart polynomial of type $A$ generalized permutohedron in \eqref{eq: typeA-Minkowski-sum}. We shall apply Postnikov's formulas to parking function polytopes to obtain their volume and Ehrhart polynomials. To do this, we first write $\pf(\bbu)$ as in \eqref{eq: typeA-Minkowski-sum}. We introduce a simplex in $\R^{n}$ that is related to $\Delta_I$:
\[\Delta_I^0 := \conv(\mathbf{0}, \Delta_I).\]
\begin{prop}\label{prop: pfu-Minkowski-sum}
	The parking function polytope 
    \[\pf(\bbu) = \sum_{I \in 2^{[n]}\backslash\{\emptyset\}} y_I\Delta_I^0\]
    where 
    \begin{align}\label{eq: y_I-simplices}
	    y_I := \sum^{|I|-1}_{j = 0}\binom{|I| - 1}{j}(-1)^ju_{|I|-j}.
    \end{align}
\end{prop}

Equation \eqref{eq: y_I-simplices} implies that for $I_1, I_2 \in 2^{[n]}\backslash\{\emptyset\}$ such that $|I_1|= |I_2|,$ one has $y_{I_1} = y_{I_2}.$ Moreover, if $\pf(\bbu)$ is integral, then $y_I$ are integers for all $I \in 2^{[n]}\backslash\{\emptyset\}$.

Before proceeding to the proof of the proposition, we apply the result to express the parking function polytopes of a family that includes the classical cases introduced in \cite{Amanbayeva_2021} and generalized cases from \cite{bayer2024parkingfunctionpolytopes, hanada2024}.

\begin{ex}
	Let $n \geq 2$ be an integer. Suppose that the coordinates of $\bbu$ form an arithmetic sequence of nonnegative real numbers. That is, $\bbu = (p, p+q, p+2q, \dots, p+(n-1)q)$ for some nonnegative real numbers $p,q$ such that $(p,q) \neq (0,0).$ 
 Then,
    \begin{align*}
    y_I =
        \begin{cases}
            p & \text{ if } |I| = 1\\
            q & \text{ if } |I| = 2\\
            0 & \text{ otherwise }
        \end{cases}.
    \end{align*}
    Thus, 
    \[\pf(\bbu) = \sum_{i = 1}^n p\Delta^0_{\{i\}} + \sum_{1\leq i < j \leq n}q \Delta^0_{\{i,j\}}.\]
\end{ex}

\begin{proof}[Proof of Proposition \ref{prop: pfu-Minkowski-sum}]
	Recall that Postnikov \cite[Proposition 6.3]{Postnikov2005} and Ardila et al. \cite[Proposition 2.3]{Ardila2010} give an expression for type A generalized permutohedron in $\R^{[0,n]}$ as
    \begin{align}
        \label{eq: typeA-projection} P(\{y'_{J}\}) = \sum_{J \in 2^{[0,n]}} y'_J\Delta_J
    \end{align}
    for some $y'_J \in \R$. 
By \cite[Section 6]{Postnikov2005}, the polytope $P(\{y'_{J}\})$ can be equivalently expressed as the set of all $(x_0, x_1, \dots, x_n) \in \R^{[0,n]}_{\geq 0}$ satisfying the following inequalties

\[x_0 + \cdots + x_n = \sum_{J \subseteq {[0,n]}} y'_J \text{ and } \sum_{i \in I}x_i \leq \left( \sum_{J \subseteq {[0,n]}} y'_J \right) - \left(\sum_{J \subseteq [0,n]\backslash I} y'_J \right) \text{ for all } \emptyset \neq I \subseteq [n].\]
Now we choose $\{ y'_J\}$ as follows 
\begin{equation}\label{eq:defny'}
       y'_J = 
        \begin{cases}
		y_{J \setminus \{0\}} \text{ as defined in \eqref{eq: y_I-simplices}}, &\text{ if $0 \in J$;} \\
		0 &\text{ otherwise.}
        \end{cases}
    \end{equation}
Then, for every proper subset $I$ of $[n]$  
    \begin{align*}
        \sum_{J \subseteq [0,n]\backslash I} y'_J = \sum_{J \subseteq [n]\backslash I} y_{J}
        &=\sum_{J \subseteq [n]\backslash I}\left[\sum^{|J|-1}_{j = 0}\binom{|J| - 1}{j}(-1)^ju_{|J|-j}\right]\\
        &=\sum^{n-|I|}_{k  = 1}u_{k}\left[\sum^{n-|I|-k}_{j=0}\binom{n-|I|}{j+k}\binom{j+k-1}{j}(-1)^j\right]\\
        &=\sum^{n-|I|}_{k  = 1}u_{k},
    \end{align*}
    since 
\[\sum^{n-|I|-k}_{j=0}\binom{n-|I|}{j+k}\binom{j+k-1}{j}(-1)^j = 1\]
is the coefficient of $x^{n-|I|-k}$ in the product of the generating functions $(1+x)^{n-|I|}$ and $(1+x)^{-k}.$ 
Therefore, with this particular choice of $\{y_J'\},$ 
the polytope $P(\{y'_{J}\})$ is the set of all $(x_0, x_1, \dots, x_n) \in \R^{[0,n]}_{\geq 0}$ satisfying the following inequalties
\begin{align*}
    x_0 + \cdots + x_n &= \sum_{J \subseteq {[0,n]}} y'_J = u_1 + \cdots + u_n, \\
    \sum_{i \in I}x_i &\leq \left( \sum_{J \subseteq {[0,n]}} y'_J \right) - \left(\sum_{J \subseteq [0,n]\backslash I} y'_J \right)= \sum^{|I|-1}_{j = 0}u_{n-j}  & \quad \text{ for all } \emptyset \neq I \subseteq [n].
\end{align*}

Let $\pi: \R^{[0,n]} \longrightarrow \R^n$ be the projection of $\R^{[0,n]}$ onto $\R^n$ defined by $\pi(x_0, x_1, \dots, x_n) = (x_1, \dots, x_n).$
It is then not difficult to see that $\pi(P(\{y'_J\}))$ is a polytope in $\R^n$ defined by
\begin{align*}
0 \leq x_i \text{ for all } i \in [n] \text{ and } \sum_{i \in I}x_i \leq \sum^{|I|-1}_{j = 0}u_{n-j} \text{ for all } \emptyset \neq I \subseteq [n].
\end{align*}
This is the inequality description for $\pf(\bbu)$ given in Proposition \ref{prop: ineqDescription}. Thus, $\pi(P(\{y'_J\})) = \pf(\bbu)$ and so we have 
\[\pf(\bbu) = \pi(P(\{y'_J\})) = \sum_{J \in 2^{[0,n]}\backslash\{\emptyset\}} \pi(y'_J\Delta_J) = \sum_{I \in 2^{[n]}\backslash\{\emptyset\}} y_I\Delta_I^0\]
as desired.
\end{proof}

\begin{rem}\label{rem:pfu-typeA}
	The projection $\pi$ defined in the proof of Proposition \ref{prop: pfu-Minkowski-sum} is an invertible linear transformation from $\mathrm{aff}(P(\{y'_J\})) = \{\bbx \in \R^{[0,n]}\mid x_0+ x_1 + \cdots + x_n = u_1 + \cdots + u_n\}$ to $\mathrm{aff}(\pf(\bbu)) = \R^n$ that gives a bijection between lattice points in these spaces. This implies that every integral parking function polytope $\pf(\bbu)$ is integrally equivalent to the type $A$ generalized permutohedron $P(\{y'_J\})$ defined in equation \eqref{eq: typeA-projection} where $\{y_J'\}$ is chosen as in \eqref{eq:defny'}. More explicitly, if $\pf(\bbu) = \sum_{i=1}^m y_{i}\Delta^0_{I_i}$, then it is integrally equivalent to the type A generalized permutohedron defined by
		$\sum_{i=1}^m y_{i}\Delta_{\{0\} \cup I_i}$.
		 Hence, they have the same Ehrhart polynomial. 
\end{rem}

This allows us to compute the Ehrhart polynomials of parking function polytopes using Postnikov's formula \cite[Theorem 11.3]{Postnikov2005}.

\begin{cor}\label{cor: ehrhart-pfu}
    Suppose that $\pf(\bbu) = y_{1}\Delta^0_{I_1} + \cdots y_{m}\Delta^0_{I_m}$ where ${I_1}, \dots, {I_m}$ are distinct nonempty subsets of $[n]$ such that  $y_{1}, \dots, y_{m}$ given in equation \eqref{eq: y_I-simplices} are all nonzero integers. Then the Ehrhart polynomial of $\pf(\bbu)$ is given by
    \begin{align}\label{eq: ehrhart-pfu}
       i(\pf(\bbu),t) = \sum_{\bba \in D(\pf(\bbu))}\binom{ty_1 + a_1 - 1}{a_1}\cdots \binom{ty_m + a_m - 1}{a_m}
    \end{align}
where $D(\pf(\bbu))$ is the set of all $\bba = (a_1, \dots, a_m) \in \Z^m_{\geq 0}$ satisfying $\sum_{j \in J}a_j \leq |\bigcup_{j \in J}I_j|$ for all nonempty subsets $J \subseteq [m]$.
\end{cor}

\begin{proof}
	By Remark \ref{rem:pfu-typeA}, the parking function polytope $\pf(\bbu) \subset \R^n$ has the same Ehrhart polynomial as the type A generalized permutohedron  
		\[ Q := y_{1}\Delta_{\{0\} \cup I_1} + \cdots y_{m}\Delta_{\{0\} \cup I_m}.\] 
	Let $y_0 = 1$ and $I_0 = [n]$, and define
	\[ Q^+_t := y_0\Delta_{\{0\} \cup I_0} + t Q = \Delta_{\{0\} \cup I_0} + ty_{1}\Delta_{\{0\} \cup I_1} + \cdots ty_{m}\Delta_{\{0\} \cup I_m}.\]
	Applying Theorem 11.3 of \cite{Postnikov2005} to $Q_t^+$, we deduce that
    \begin{align}
        i(\pf(\bbu),t) &= i(Q,t)\nonumber\\
\label{eq: ehrhart_typeA}        &= \sum_{\bba \in \Bar{D}(Q^+_t)}\binom{1+ a_0-1}{a_0}\binom{ty_1 + a_1 - 1}{a_1}\cdots \binom{ty_m + a_m - 1}{a_m}
    \end{align}
    where $\Bar{D}(Q^+_t)$ is the set of $\bba = (a_0, a_1, \dots, a_m) \in \Z^{m+1}_{\geq 0}$ satisfying $a_0 +a_1 + \cdots + a_m = n$ and $\sum_{j \in J}a_j \leq |\bigcup_{j \in J}I_j|$ for all nonempty subsets $J \subseteq [m]$. Note that the sequences $\bba \in \Bar{D}(Q_t^+)$ are called ``$G$-draconian sequences of $Q_t^+$''. Since $\binom{1+ a_0-1}{a_0} = 1$ for all nonnegative integers $a_0,$ we can rewrite \eqref{eq: ehrhart_typeA} as
    \[i(\pf(\bbu),t) = \sum_{\bba \in D(\pf(\bbu))}\binom{ty_1 + a_1 - 1}{a_1}\cdots \binom{ty_m + a_m - 1}{a_m}\]
    where $D(\pf(\bbu))$ is the set of all $\bba = (a_1, \dots, a_m) \in \Z^{m}_{\geq 0}$ satisfying $\sum_{j \in J}a_j \leq |\bigcup_{j \in J}I_j|$ for all nonempty subsets $J \subseteq [m]$. This gives the formula in equation \eqref{eq: ehrhart-pfu} as desired.
\end{proof}

We note that one can use Proposition \ref{prop: pfu-Minkowski-sum} to derive another formula for the Ehrhart polynomials of parking function polytopes by applying Theorem $a/(c)$ in \cite{Euretall2024} given by Eur et al. for type $B$ generalized permutohedra. However, we find that Postniknov's formula for type $A$ generalized permutohedra gives a more desirable expression in the sense that when all $y_I$ given in equation \eqref{eq: y_I-simplices} are nonnegative, one easily sees from formula \eqref{eq: ehrhart-pfu} that all the coefficients of $i(\pf(\bbu),t)$ are positive. This positivity of coefficients is, however, not easily seen when expressing $i(\pf(\bbu),t)$ using the formula provided by Eur et al.

The leading coefficient of the Ehrhart polynomial in \eqref{eq: ehrhart-pfu} gives a formula for the volume of $\pf(\bbu)$.

\begin{cor}\label{cor: volume-pfu}
    Suppose that $\pf(\bbu)$ is $n$-dimensional and that $\pf(\bbu) = y_{1}\Delta^0_{I_1} + \cdots y_{m}\Delta^0_{I_m}$ where ${I_1}, \dots, {I_m}$ are distinct nonempty subsets of $[n]$ such that  $y_{1}, \dots, y_{m}$ given in equation \eqref{eq: y_I-simplices} are all nonzero. Then, the volume of $\pf(\bbu)$ is given by
    \begin{align}\label{eq: volume-pfu}
       \sum_{\substack{\bba \in D(\pf(\bbu)) \\ a_1 +\cdots + a_m = n} }\frac{y_1^{a_1}}{a_1!}\cdots \frac{y_m^{a_m}}{a_m!}.
    \end{align}
\end{cor}

\begin{proof}
    Suppose that $\pf(\bbu)$ is $n$-dimensional. The leading coefficient of the Ehrhart polynomial of $\pf(\bbu)$ in \eqref{eq: ehrhart-pfu} are contributed by the terms corresponding to those $\bba \in D(\pf(\bbu))$ satisfying $a_1 + \cdots + a_m = n$. Since the volume of $\pf(\bbu)$ is equal to the leading coefficient of its Ehrhart polynomial, we have
    \begin{align}
       \vol(\pf(\bbu)) = \sum_{\substack{\bba \in D(\pf(\bbu)) \\ a_1 +\cdots + a_m = n} }\frac{y_1^{a_1}}{a_1!}\cdots \frac{y_m^{a_m}}{a_m!}.
    \end{align}
    By Theorem 3.2 of \cite{Ewald1996} concerning the volume of Minkowski sums of polytopes, the volume of $\pf(\bbu)$ is a polynomial in variables $y_1, \dots, y_m$. This implies that formula \eqref{eq: volume-pfu} also extends to non-integral $\pf(\bbu).$
\end{proof}

\bibliographystyle{abbrv}
\bibliography{BibContainer}

\end{document}